\documentclass[12pt,reqno]{amsart}

\usepackage{mathtools}

\usepackage{mathdots}

\usepackage{color}

\usepackage{pdfsync}

\usepackage{enumitem}

\usepackage{wasysym}

\usepackage{amssymb}

\usepackage{mathrsfs}

\DeclareMathAlphabet{\mathpzc}{OT1}{pzc}{m}{it}

\usepackage[all]{xy}

\usepackage{tikz}
\usetikzlibrary{arrows,matrix,decorations.pathmorphing,decorations.pathreplacing,positioning,shapes.geometric,shapes.misc,decorations.markings,decorations.fractals,calc,patterns,bending}
\makeatletter
\tikzset{
    on each segment/.style={
        decorate,
        decoration={
            show path construction,
            moveto code={},
            lineto code={
                \path [#1]
                (\tikzinputsegmentfirst) -- (\tikzinputsegmentlast);
            },
            curveto code={
                \path [#1] (\tikzinputsegmentfirst)
                .. controls
                (\tikzinputsegmentsupporta) and (\tikzinputsegmentsupportb)
                ..
                (\tikzinputsegmentlast);
            },
            closepath code={
                \path [#1]
                (\tikzinputsegmentfirst) -- (\tikzinputsegmentlast);
            },
        },
    },
}
\makeatother

\usepackage{graphicx}

\usepackage{float}

\usepackage[bottom]{footmisc}

\usepackage{moreenum}

\usepackage{makecell}

\usepackage[T1]{fontenc}

\entrymodifiers={+!!<0pt,\fontdimen22\textfont2>}

\def\cA{\mathscr{A}}
\def\cB{\mathscr{B}}
\def\cC{\mathscr{C}}
\def\cD{\mathscr{D}}
\def\cE{\mathscr{E}}
\def\cF{\mathscr{F}}
\def\cG{\mathscr{G}}
\def\cH{\mathscr{H}}
\def\cI{\mathscr{I}}

\def\cK{\mathscr{K}}

\def\cP{\mathscr{P}}

\def\cU{\mathscr{U}}
\def\cV{\mathscr{V}}

\def\cZ{\mathscr{Z}}

\def\BC{\mathbb{C}}

\def\BR{\mathbb{R}}

\def\BZ{\mathbb{Z}}

\def\fa{\mathfrak{a}}
\def\fb{\mathfrak{b}}
\def\fc{\mathfrak{c}}

\def\fr{\mathfrak{r}}

\def\2silt{\operatorname{2-silt}}

\def\add{\operatorname{add}}

\def\adots{\mathinner{\mkern1mu\raise1.0pt\vbox{\kern7.0pt\hbox{.}}\mkern2mu\raise4.0pt\hbox{.}\mkern2mu\raise7.0pt\hbox{.}\mkern1mu}}

\def\ast{{\textstyle *}}

\def\D{\cD}

\def\dddots{\mathinner{\mkern1mu\raise10.0pt\vbox{\kern7.0pt\hbox{.}}\mkern2mu\raise5.3pt\hbox{.}\mkern2mu\raise1.0pt\hbox{.}\mkern1mu}}
\def\dddotssmall{\mathinner{\mkern1mu\raise7.0pt\vbox{\kern7.0pt\hbox{.}}\mkern-1mu\raise4pt\hbox{.}\mkern-1mu\raise1.0pt\hbox{.}\mkern1mu}}

\def\dim{\operatorname{dim}}

\def\DPic{\operatorname{DPic}}
\def\dual{\operatorname{D}}

\def\Ext{\operatorname{Ext}}

\newcommand{\greenstart}{\operatorname{enter}}
\newcommand{\greenend}{\operatorname{exit}}

\def\H{\operatorname{H}}

\def\Hom{\operatorname{Hom}}
\def\hyper{[hyper]}
\def\id{\operatorname{id}}

\def\index{\operatorname{index}}

\def\inf{\operatorname{inf}}

\def\K{\cK}
\def\K0{\operatorname{K}_0}
\def\Kb{\cK^b}

\def\Ker{\operatorname{Ker}}

\def\Ksp{\operatorname{K}_0^{\operatorname{sp}}}

\def\lim{\operatorname{lim}}
\def\linearspan{\operatorname{span}}

\newcommand\LTensor[1]{\overset{{\rm L}}{\underset{#1}{\otimes}}}

\def\mod{\operatorname{mod}}
\def\Mod{\operatorname{Mod}}

\def\mrig{\operatorname{mrig}}
\def\mut{1}

\def\opp{\operatorname{o}}

\def\proj{\operatorname{proj}}

\def\PSL2{\operatorname{PSL}_2}

\def\RHom{\operatorname{RHom}}

\def\SL2{\operatorname{SL}_2}
\def\somefiniteness{[fp]}

\def\sup{\operatorname{sup}}

\newcommand\Tensor[1]{\underset{#1}{\otimes}}

\def\top{\operatorname{top}}

\def\intt{\operatorname{int}}
\def\Cart{C} 

\newcommand\silt[2][]{P^{#1}\big(#2\big)}
\newcommand\siltlittle[2][]{P^{#1}(#2)}
\newcommand\simple[1][]{S(#1)}

\numberwithin{equation}{section}
\renewcommand{\theequation}{\arabic{section}.\arabic{equation}}

\renewcommand{\thesubsection}{(\arabic{section}.\roman{subsection})}

\renewcommand{\labelenumi}{(\roman{enumi})}

\newtheorem{Lemma}{Lemma}[section]
\newtheorem{Theorem}[Lemma]{Theorem}
\newtheorem{Proposition}[Lemma]{Proposition}
\newtheorem{Corollary}[Lemma]{Corollary}

\theoremstyle{definition}
\newtheorem{Definition}[Lemma]{Definition}
\newtheorem{Setup}[Lemma]{Setup}

\newtheorem{Remark}[Lemma]{Remark}

\newtheorem{Observation}[Lemma]{Observation}

\newtheorem{Notation}[Lemma]{Notation}

\newtheorem*{bfhpg*}{}

\newenvironment{VarDescription}[1]%
  {\begin{list}{}{%
    \settowidth{\labelwidth}{\textbf{#1:}}%
    \setlength{\leftmargin}{\labelwidth}\addtolength{\leftmargin}{\labelsep}}}%
  {\end{list}}

\begin{document}

\setlength{\parindent}{0pt}
\setlength{\parskip}{7pt}

\title[Derived actions of green groupoids of 2-Calabi--Yau categories]{Green groupoids of 2-Calabi--Yau categories, derived Picard actions, and
hyperplane arrangements}

\author{Peter J\o rgensen}

\address{School of Mathematics and Statistics,
Newcastle University, Newcastle upon Tyne NE1 7RU, United Kingdom}
\email{peter.jorgensen@ncl.ac.uk}

\urladdr{http://www.staff.ncl.ac.uk/peter.jorgensen}

\author[Milen Yakimov]{Milen Yakimov}

\address{Department of Mathematics, Louisiana State University, Baton Rouge, LA 70803, U.S.A.}
\email{yakimov@math.lsu.edu}






\begin{abstract} 
We present a construction of (faithful) group actions via derived equivalences in the 
general categorical setting of algebraic 2-Calabi--Yau triangulated categories. 

To each algebraic 2-Calabi--Yau category $\cC$ satisfying standard mild assumptions, we associate a 
groupoid $ \cG_{ \cC }$, named the green groupoid of $\cC$, defined in an intrinsic homological way.
Its objects are given by 
a set of representatives $\mrig \cC$ of the equivalence classes of basic maximal rigid objects of $\cC$, 
arrows are given by mutation, and relations are given by equating monotone (green) paths in the silting order.
In this generality we construct a homomorphsim from the green groupoid $ \cG_{ \cC }$ to the derived Picard groupoid of the collection 
of endomorphism rings of representatives of $\mrig \cC$ in a Frobenius model of $\cC$; the latter 
canonically acts by triangle equivalences between the derived categories of the rings. 

We prove that the constructed representation of the green groupoid $ \cG_{ \cC }$ is faithful if the 
index chamber decompositions of the split Grothendieck groups of basic maximal rigid objects of $\cC$
come from hyperplane arrangements. If $\Sigma^2 \cong \id$ and  $\cC$ has 
finitely many equivalence classes of basic maximal rigid objects, we prove that 
$ \cG_{ \cC }$  is isomorphic to a Deligne groupoid of a hyperplane arrangement 
and that the representation of this groupoid is faithful. 
\end{abstract}

\maketitle

\setcounter{section}{-1}
\section{Introduction}
\label{sec:introduction}
In 1987 Happel \cite{H0} proved that the Bernstein--Gelfand--Ponomarev reflection functors provide derived equivalences between 
path algebras of Dynkin quivers. Subsequently, derived equivalences have played a profound role in many areas, such as 
mirror symmetry, ca\-te\-go\-ri\-fi\-ca\-ti\-on in representation theory, birational algebraic geometry, cluster algebras, and others. 

Motivated by 4 dimensional TQFT and mirror symmetry, (faithful) actions of mapping class groups of surfaces and fundamental 
groups of complements of hyperplane arrangements via derived equivalences have been constructed 
in a wide range of situations:

{\bf{Algebraic Geometry:}} 
\begin{enumerate}
\item[(1)] Seidel and Thomas \cite{ST} constructed braid group actions on the derived categories
of the minimal resolutions of quotient singularities $\BC/G$ for finite groups $G \subset SL_2(\BC)$. Faithfulness 
was established by them in certain cases and by Brav and Thomas \cite{BT} in full generality.
The method of \cite{ST} for constructing braid group actions on derived categories 
by spherical twists was much applied and developed later.
\item[(2)] Donovan and Wemyss \cite{DW,W} constructed actions of Deligne groupoids via derived equivalences between 
between 3-fold flopping contractions (the minimal models for certain complete local 3-dimensional hypersurface singularities), 
or equivalently between the corresponding maximal modification
algebras. Faithfulness was proved by Hirano and Wemyss \cite{HW}. August \cite{Au2} proved analogs of these results in the parallel 
picture of contraction algebras and constructed derived equivalence functors intertwining this picture to that in \cite{DW,W}. 
\end{enumerate}

{\bf{Representation Theory:}}
\begin{enumerate}
\item[(3)] Miyachi and Yekutieli \cite{MY} constructed faithful actions of the automorphism groups of certain AR quivers on the 
bounded derived categories of basic finite dimensional hereditary algebras and proved that those generate the corresponding derived Pickard groups
together with the images of the ordinary Picard groups. 
\item[(4)] Khovanov and Seidel \cite{KS} constructed a faithful braid group action on the bounded derived category of the path algebra of the double of the $A_n$ quiver, 
modulo a graded ideal $I$ such that $J^3 \subsetneq I \subsetneq J^2$ for the Jacobson radical $J$ of the path algebra. 
\item[(5)] Rouquier and Zimmermann \cite{RZ} constructed braid group actions on the bounded derived categories of Brauer tree algebras. 
\end{enumerate}
{\bf{Topology:}} 
\begin{enumerate}
\item[(6)] Lipshitz, Ozsv\'ath and Thurston \cite{LOT} constructed a faithful action of the mapping class group of a surface with boundary 
on the bounded derived category of a certain finite dimensional algebra $B$ over $\BZ_2$ constructed from bordered Floer homology.
\end{enumerate}

{\bf{Cluster algebras:}}
\begin{enumerate}
\item[(7)] Keller and Yang \cite{KY} extended Happel's result \cite{H0} to derived equivalences between the Ginzburg dg algebras for 
mutations of quivers with potentials. Grant and Marsh \cite{GM} applied them to get actions of the braid groups of type 
ADE on the derived categories of the Ginzburg dg algebras in those mutation classes and proved that the actions are 
faithful in the $A$ case. 
\item[(8)] Qui \cite{Q2} constructed faithful actions of certain subgroups of mapping class groups of decorated marked surfaces (generated
by braid twists) on derived categories of Ginzburg dg algebras and proved that those groups are isomorphic to the spherical twist groups of the derived categories.
\end{enumerate}

From another perspective, the above results are thought of as statements that the groups of autoequivalences of certain (bounded) derived categories
are large, as opposed to other cases when such were proved to be small, e.g. the Bondal--Orlov theorem \cite{BO} that $\D^b( \mathrm{coh} \, X)$ 
has no nontrivial autoequivalences for projective varieties $X$ with ample canonical or anticanonical sheaf.

In this paper we present a construction of (faithful) group actions via derived equivalences in the general categorical setting of algebraic 2-Calabi--Yau triangulated categories.

{\bf{Assumptions.}} For an algebraically closed field $k$, let $\cC$ be a $k$-linear $2$-Calabi--Yau triangulated category satisfying the following standard mild assumptions:
\begingroup
\renewcommand{\labelenumi}{(\alph{enumi})}
\begin{enumerate}
\setlength\itemsep{4pt}

  \item  $\cC$ is algebraic, i.e. it has a Frobenius model $\cE$.
  
  \item  $\cC$ has finite dimensional $\Hom$-spaces over $k$ and split idempotents.

  \item  $\cC$ has non-zero basic maximal rigid objects, i.e.\ its maximal rigid subcategories are non-zero and have finitely many indecomposable objects.

\end{enumerate}
\endgroup
Such categories naturally arise in different situations of both geometric and algebraic origin.
For all local complete 3-dimensional (commutative) noetherian Gorenstein isolated singularities $R$, 
Burban, Iyama, Keller and Reiten \cite{BIKR} proved that the stable module category of the category of maximal Cohen--Macauley modules of $R$ has the above properties
when it has non-zero rigid objects.
By work of Eisenbud \cite{E} the same holds for the categories coming from all local complete odd dimensional hypersurface singularities.
Additive categorifications of cluster algebras \cite{BIRS,BMRRT} are also constructed in this framework.

Our results do not use any specific information about the Frobenius model $\cE$ and are uniformly applicable to $2$-Calabi--Yau categories
coming from different origins.

{\bf{The green groupoid $\cG_{ \cC }$.}} We associate a groupoid (i.e.\ a category with all morphisms invertible) $\cG_{ \cC }$ to each such category $\cC$.  It is defined in an intrinsic homological way; the Frobenius models do not play a role in it, see Definition \ref{def:green_groupoid}.  The object set of $ \cG_{ \cC }$ is a 
set of representatives $\mrig \cC$ of the equivalence classes of basic maximal rigid objects of $\cC$. 
The generators for the morphisms of $\cG_{ \cC }$ are the one-step mutations between the   
basic maximal rigid objects of $\cC$ in the sense of Iyama and Yoshino \cite{IY}. 
The relations of $\cG_{ \cC }$ are obtained by equating green paths 
\[
  m(1) \xrightarrow{} m(2) \xrightarrow{} \cdots \xrightarrow{} m(p)
\]
in $\mrig \cC$ with the same initial and terminal objects.  In more detail, for $\ell \in \mrig \cC$ there is a 
bijection $m \mapsto \silt[\ell]{m}$ between $\mrig \cC$ and the equivalence classes 
of 2-step silting complexes over $\cC(\ell, \ell)$, see \cite{AIR}. Following \cite[2.7]{Br} and \cite[thm.\ 4.10]{BST}, 
a green path in $\mrig \cC$ is a sequence of one-step mutations as above such that 
$\silt[m(1)]{m(1)} > \silt[m(1)]{m(2)} > \cdots > \silt[m(1)]{m(p)}$ 
in the silting partial order.

The relations of $\cG_{ \cC }$ have intrinsic homological meaning and this is at the heart of our methods to construct representations 
of $\cG_{ \cC }$ by derived equivalences and to control their faithfulness. For this reason it is essential to work with the full groupoid as opposed 
to its vertex groups.

Due to the intrinsic homological definition of the green groupoid $\cG_{ \cC }$, we believe that 
it will be useful in studying other properties of Calabi--Yau categories.

{\bf{The derived Picard groupoid $\cP_{ \cE }$.}} We identify the sets of objects of $\cC$ and $\cE$. Following 
Miyachi--Yekutieli \cite{MY,Y} and Rouquier--Zimmermann \cite{RZ}, 
define the derived Picard groupoid of $\cE$ to be the groupoid with object set $\mrig \cC$ and morphism sets given by 
the isomorphism classes of two-sided tilting complexes ${}_{ \cE( m,m ) }T_{ \cE( \ell,\ell ) }$ in $\D\big( \cE( m,m )^{ \opp } \Tensor{k} \cE( \ell,\ell ) \big)$
for $\ell, m \in \mrig \cC$. (Derived Picard groups were considered in \cite{Y,RZ} and the general construction appeared in \cite{MY}.)


{\bf{$g$-vector fans of split Grothendieck groups.}} For a basic maximal rigid object $\ell \in \cC$, 
denote by $\Ksp( \add_{ \cC }\ell )_{ \BR }$ the extension to $\BR$ of the split Grothendieck  group of $\add_{ \cC }\ell$ (a finite dimensional real vector space). Following Dehy and Keller \cite{DK}, consider the $g$-vector fan of $\cC$ with respect to $\ell$, 
\[
  G^{ \ell }_{ \cC } = \{\, B^{ \ell }( m ) \mid \mbox{$m \in \mrig \cC$} \,\},
\]
where each chamber $B^{ \ell }( m )$ is a simplicial cone with extremal rays given by the indices of the indecomposable summands of $m$ with respect to $\ell$.  
We say that such a fan comes from a hyperplane arrangement in $\Ksp( \add_{ \cC }\ell )_{ \BR }$ if its 
cones are precisely the cones of the hyperplane arrangement.

{\bf{Representations.}} Our first main theorem constructs groupoid representations and proves faithfulness
in a general situation.

{\bf Theorem A} (=Theorems \ref{thm:G11}, \ref{thm:green-Delign}, and \ref{thm:H18}).
\begin{enumerate}
\item {\em{For all 2-Calabi--Yau categories $\cC$ satisfying assumptions (a-c), there exists a unique functor
$\cG_{ \cC } \xrightarrow{ G } \cP_{ \cE }$, which is the identity on objects and determined on morphisms by
\[
  G( \alpha ) = \cE( m,m^{ \ast } )
\]
for all one-step mutations $m \xrightarrow{ \alpha } m^{ \ast }$. 
For every green path, $\mu = m(1) \xrightarrow{} \cdots \xrightarrow{} m(p)$, we have
\[
  G( \mu ) = \cE\big( m(1),m(p) \big).
\]
}}
\item {\em{If a 2-Calabi--Yau category $\cC$ has the two properties that for all basic maximal rigid objects $\ell \in \cC$, 
\begin{VarDescription}{hyper\quad}
\setlength\itemsep{2pt}

  \item[{\rm [hyper]}] the $g$-vector fan $G^{ \ell }_{ \cC }$ comes from a hyperplane arrangement, and
  
  \item[{\rm [fp]}] each finitely generated $\cC(\ell,\ell)$-module is finitely presented as an $\cE(\ell,\ell)$-module, 

\end{VarDescription}
then the green groupoid $\cG_{ \cC }$ of $\cC$ is isomorphic to the Deligne groupoid (see \cite{D}) of each of the hyperplane arrangements in {\rm [hyper]}, and the functor $G$ in part (i) is faithful.}}
\end{enumerate}

Note the the green groupoid $\cG_{ \cC }$ only depends on the 2-Calabi--Yau category $\cC$, while 
the functor $\cG_{ \cC } \xrightarrow{ F } \cP_{ \cE }$ depends on the choice of Frobenius model $\cE$. We expect that it is possible to gain significant knowledge about the green groupoid $\cG_{ \cC }$ by using the functors $F$ from Theorem A(i) for different Frobenius models $\cE$.

{\bf{Green groupoids and hyperplane arrangements.}} Condition [fp] in Theorem A(ii) is very mild. Our second 
main theorem shows that condition [hyper] in Theorem A(ii) is satisfied in wide generality. As a consequence,
this result gives an effective method for identifying green groupoids $\cG_{ \cC }$ with Deligne groupoids
for a very general class of  2-Calabi--Yau categories.

{\bf Theorem B} (=Theorem \ref{thm:gfan-to-hyperplane}).
{\em{If a 2-Calabi--Yau category $\cC$ has finitely many isomorphism classes of basic maximal rigid objects and 
$\Sigma^2 n \cong n$ for all rigid objects $n$ of $\cC$, then the following hold:}}
\begin{enumerate}
\setlength\itemsep{4pt}
\item {\em{$\cC$ satisfies condition {\rm [hyper]} in Theorem A(ii), i.e.\ the $g$-vector fan $G^{ \ell }_{ \cC }$ comes from a simplicial hyperplane arrangement in 
$\Ksp( \add_{ \cC }\ell )_\BR$ for all $\ell \in \mrig \cC$. The green groupoid $\cG_{ \cC }$ is isomorphic to the Deligne 
groupoids of those hyperplane arrangements.}} 
\item {\em{The Cartan matrix $C_\ell$ of every 2-Calabi--Yau tilted algebra $\cC(\ell, \ell)$ (see Equation \eqref{equ:Cartan}) is nondegenerate, and furthermore, the symmetrization $C_\ell + C_\ell^T$ is positive definite.}}
\end{enumerate}
All 2-Calabi--Yau categories coming from local complete odd dimensional hypersurface singularities \cite{E}
have the property $\Sigma^2 \cong \id$ and thus 
satisfy the second condition in Theorem B. Both conditions in Theorem B are satisfied by the  2-Calabi--Yau categories considered in \cite{Au2,HW,W}. 

The faithfulness result in Theorem B(i) is not applicable to the cluster categories of \cite{BMRRT}.  However, Theorem A(i) is applicable to them and it is quite possible that 
faithfulness holds in that general case too.

Theorem B(ii) provides a huge supply of symmetric finite dimensional algebras $\cC(\ell, \ell)$ with $\tau^2 \cong \id$ and nondegenerate Cartan 
matrices. This construction is of independent interest due to the role of these algebras in representation theory. The conclusion of Theorem B(ii) is violated if the first condition is dropped: Plamondon \cite[ex.\ 4.3]{Pl} constructed 2-Calabi--Yau categories $\cC$ with $\Sigma^2 n \cong n$ for each $n$, for which the tilted algebras $\cC(\ell, \ell)$ have degenerate Cartan matrices.

The positivity result
in Theorem B(ii) gives rise to positive definite symmetric bilinear forms on $\Ksp( \add_{ \cC }\ell )_{ \BR }$, see Equation \eqref{equ:ell-form}, which we prove to be invariant under the Dehy--Keller linear isomorphisms 
\[
\Ksp( \add_{ \cC }\ell )_{ \BR } \stackrel{\cong}{\rightarrow} \Ksp( \add_{ \cC } m)_{ \BR }
\]
for $\ell,m \in \mrig \cC$ (Proposition \ref{prop:gfan-to-hyperplane2}(ii) and Lemma \ref{lem:equiv-index}). Since there is also 
a canonical skewsymmetric bilinear form on $\Ksp( \add_{ \cC }\ell )_{ \BR }$ (cf.\ \cite{Pa}), and both forms have
categorical incarnations, it appears that there is 
a sort of categorical K\"ahler geometry in the background for the 2-Calabi--Yau categories $\cC$ satisfying the two 
conditions in Theorem B.

We expect that many interesting classes of groups can be realized as green groups (vertex groups of green groupoids) of various 
2-Calabi--Yau categories $\cC$ and that Theorem A(i) can be used to produce actions of those on derived categories.
In this direction, there is already a vast array of Frobenius models of 2-Calabi--Yau categories 
constructed to date (\cite{BKM,BIRS,BIKR,GLS2} and many others) to which Theorem A(i) is applicable. The faithfulness 
of the groupoid actions in Theorem A(i) likely holds in wider generality than that established in Theorem A(ii).

We illustrate the construction of the green groupoid and Theorem A for cluster categories. Section \ref{sec:A2}
explicitly describes this groupoid for the cluster category $\cC(A_2)$ of type $A_2$ and the algebras $\cE(\ell,\ell)$ 
entering in the derived Picard groupoid for the Frobenius model of Baur, King, and Marsh \cite{BKM} of $\cC(A_2)$.
The algebras $\cE(\ell,\ell)$ are isomorphic to each other and are dimer algebras. The green groupoid
of $\cC(A_2)$ is shown to be a 5 point blow-up of the braid group $\cB_2$, while the green groups of $\cC(A_2)$ are shown to 
be isomorphic to  $\cB_2$.

\section{Setup}
\label{sec:setup}
\begin{Setup}
\label{set:blanket}
The rest of the paper uses the following setup.
\begin{enumerate}
\setlength\itemsep{4pt}

  \item  $k$ is an algebraically closed field.
  
  \item  $\cE$ is a $k$-linear Frobenius category (see Remark \ref{rmk:Frobenius}).  We assume that it has a fixed object $r$ such that $\cP = \add_{ \cE }( r )$ is the full subcategory of projective-injective objects.
  
  \item  $\cC = \underline{ \cE } = \cE/[\cP]$ is the stable category, where $[\cP]$ denotes the ideal of morphisms factoring through an object of $\cP$.  Note that:

  \item  $\cC$ is a $k$-linear triangulated category.  We assume that it has finite dimensional $\Hom$-spaces over $k$, split idempotents, and is $2$-Calabi--Yau.

\end{enumerate}
We will use $\cC( -,- )$ as a shorthand for $\Hom_{ \cC }( -,- )$, and if $\alpha$ is a morphism in $\cE$ then $\underline{ \alpha }$ denotes its image in $\cC$.  The suspension functor in a triangulated category will be denoted by $\Sigma$.  The word {\em mutation} means {\em mutation at an indecomposable object}.  Module categories and derived categories consist of right modules unless otherwise stated.
\end{Setup}


\begin{Remark}
\label{rmk:Frobenius}
Recall that a {\em Frobenius category} is an exact category, which has enough projective and enough injective objects and satisfies $\cP = \cI$, where $\cP$ is the full subcategory of projective objects, $\cI$ the full subcategory of injective objects.  The objects in $\cP = \cI$ are called {\em projective-injective objects}.  See \cite[app.\ A]{K} and \cite[\S 2, p.\ 100]{Q} for exact categories in general, \cite[sec.\ I.2]{H} for Frobenius categories in particular.
\end{Remark}

\begin{Definition}
[Maximal rigid objects]
\label{def:maximal_rigid}
An object $\ell \in \cC$ is {\em maximal rigid} if it is rigid, that is $\cC( \ell,\Sigma \ell ) = 0$, and maximal with this property in the sense that $\cC\big( \ell \oplus m,\Sigma( \ell \oplus m ) \big) = 0 \Rightarrow m \in \add_\cC \ell$, see \cite[p.\ 954]{BMV}.

We say that $\cC$ is {\em maximal rigid finite} if it has finitely many isomorphism classes of basic maximal rigid objects.

Note that if $\cC$ has a cluster tilting object, then its maximal rigid objects are precisely its cluster tilting objects, see \cite[thm.\ 2.6]{ZZ}.  If $\cC$ has no cluster tilting objects, then it may still have maximal rigid objects, see \cite[sec.\ 2]{BMV}.
\end{Definition}

\begin{Definition}
[Good objects]
\label{def:good}
Let $\ell$ be an object of $\cC$, or equivalently, an object of $\cE$ (recall that $\cC$ and $\cE$ have the same objects).  We say that $\ell$ is {\em good} if $\ell \cong \ell' \oplus r$ in $\cE$ for some $\ell' \in \cE$, where $r$ is as in Setup \ref{set:blanket}(ii).
\end{Definition}

\begin{Definition}
\label{def:mrig}
\begin{itemize}
\setlength\itemsep{4pt}

  \item  $\mrig \cC =$ a set containing one good object from each isomorphism class of basic maximal rigid objects of $\cC$, see \cite[def.\ 2.1]{ZZ}.  Note that $\mrig \cC$ is finite if and only if $\cC$ is maximal rigid finite. 
  
  \item  If $\underline{A}$ is a finite dimensional $k$-algebra, then $\2silt \underline{A} =$ the set of isomorphism classes of basic $2$-term silting complexes in the homotopy category $\Kb( \proj \underline{A} )$, see \cite[sec.\ 3]{AIR}.

\end{itemize}
\end{Definition}

\section{The derived Picard groupoid}
\label{sec:DPic}

The following definition uses $\D\big( \cE( m,m )^{ \opp } \Tensor{k} \cE( \ell,\ell ) \big)$, the derived category of $\cE( m,m )$-$\cE( \ell,\ell )$-bimodules.  Module structures are henceforth denoted by subscripts.

\begin{Definition}
[The derived Picard groupoid]
\label{def:derived_Picard_groupoid}
The {\em derived Picard groupoid} (see \cite[sec.\ 5]{MY})  $\cP_{ \cE }$ of $\cE$ has:
\begin{itemize}
\setlength\itemsep{4pt}

  \item  Objects: The elements of $\mrig \cC$.

  \item  Morphisms: If $\ell, m \in \mrig \cC$, then $\cP_{ \cE }( \ell,m )$ is the set of isomorphism classes of two-sided tilting complexes ${}_{ \cE( m,m ) }T_{ \cE( \ell,\ell ) }$ in $\D\big( \cE( m,m )^{ \opp } \Tensor{k} \cE( \ell,\ell ) \big)$, see \cite[def.\ 3.4]{R}.  

\end{itemize}
The composition of $T \in \cP_{ \cE }( \ell,m )$ and $U \in \cP_{ \cE }( m,n )$ is the isomorphism class of $U \LTensor{ \cE( m,m ) } T$.

The vertex group $\cP_{ \cE }( m,m )$ is $\DPic \cE( m,m )$, the derived Picard group of the endomorphism algebra $\cE( m,m )$, see \cite[sec.\ 1]{RZ} and \cite[sec.\ 0]{Y}.

The notation glosses over the dependence of $\cP_{ \cE }$ on the choice of one good object from each isomorphism class of basic maximal rigid objects of $\cC$.  
\end{Definition}

\section{The green groupoid}
\label{sec:green}

This section does not use the Frobenius model $\cE$, but works under Setup \ref{set:blanket}(iv) alone (in which case $\mrig \cC$ should just be a set containing one object from each isomorphism class of basic maximal rigid objects of $\cC$).

The following result is due to \cite[cor.\ 2.5]{ZZ}, which follows the idea of \cite[cor.\ 5.2]{GLS}.

\begin{Lemma}
\label{lem:ZZ}
Let $\ell \in \cC$ be a maximal rigid object.  For each rigid object $m \in \cC$ there is a triangle in $\cC$ with $\ell_i \in \add_{ \cC }( \ell )$:
\[
  \ell_1 \xrightarrow{} \ell_0 \xrightarrow{} m \xrightarrow{} \Sigma \ell_1.
\]
\end{Lemma}

\begin{Definition}
\label{def:silt}
Let $\ell \in \cC$ be a basic maximal rigid object and set $\underline{A} = \cC( \ell,\ell )$.  If $m \in \cC$ is rigid, then the triangle in Lemma \ref{lem:ZZ} induces the complex
\[
  \silt[\ell]{m} = \big( \cC( \ell,\ell_1 ) \xrightarrow{} \cC( \ell,\ell_0 ) \big),
\] 
which is well-defined up to isomorphism in $\Kb( \proj \underline{A} )$.
\end{Definition}

The following result is due to \cite[thm.\ 4.7 and cor.\ 4.8]{AIR} for cluster tilting objects and \cite[thm.\ 2.16 and rmk.\ 2.17]{Au1} for maximal rigid objects.

\begin{Lemma}
\label{lem:AIR}
Let $\ell \in \cC$ be a fixed basic maximal rigid object and set $\underline{A} = \cC( \ell,\ell )$.  The assignment $m \mapsto \siltlittle[\ell]{m}$ preserves the number of indecomposable summands, and
induces a well-defined bijection
\[
  \mrig \cC \xrightarrow{} \2silt \underline{A},
\]
which is compatible with mutation.
\end{Lemma}

\begin{Remark}
The last sentence of the lemma makes sense because basic maximal rigid objects and basic $2$-term silting complexes have unique mutations.

For a basic maximal rigid object this is \cite[cor.\ 3.3]{ZZ}.  For a basic $2$-term silting complex, there is both a left and a right mutation, see \cite[def.\ 2.30]{AI}, and both are silting complexes, but precisely one of them is a $2$-term silting complex.
\end{Remark}

\begin{Definition}
[Silting order]
\label{def:silting_order}
Recall from \cite[def.\ 2.10]{AI} that if $P,Q$ are isomorphism classes of silting objects in $\Kb( \proj \underline{A} )$, then
\[
  P \geqslant Q \; \Leftrightarrow \; \Hom_{ \Kb( \proj \underline{A} ) }( P,\Sigma^{>0}Q ) = 0.
\]
In particular, if $P,Q \in \2silt \underline{A}$, then
\[
  P \geqslant Q \; \Leftrightarrow \; \Hom_{ \Kb( \proj \underline{A} ) }( P,\Sigma Q ) = 0.
\]
The sign ``$>$'' means ``$\geqslant$ but not $=$''.
\end{Definition}

\begin{Definition}
[The quiver $\Gamma_{ \cC }$]
\label{def:Gamma}
Let $\Gamma_{ \cC }$ be the quiver with:
\begin{itemize}
\setlength\itemsep{4pt}

  \item  Vertices: The elements of $\mrig \cC$.

  \item  Arrows:
$
  \xymatrix {
  m
    \ar[r]<0.5ex> &
    m^{ \ast }
    \ar[l]<0.5ex>
            }
$
when $m$ and $m^{ \ast }$ differ by mutation in $\cC$.

\end{itemize}
This quiver can be obtained from the mutation graph of maximal rigid objects of $\cC$ by replacing each edge with two arrows, one in each direction.

If $m \xrightarrow{ \alpha } m^{ \ast }$ is an arrow, then the {\em source} and {\em target} are $s( \alpha ) = m$ and $t( \alpha ) = m^{ \ast }$.
\begin{itemize}
\setlength\itemsep{4pt}

  \item  A {\em path of length $p$} in $\Gamma_{ \cC }$ is an expression $\mu = \alpha_p \cdots \alpha_1$ where the $\alpha_i$ are arrows of $\Gamma_{ \cC }$ and $t( \alpha_i ) = s( \alpha_{ i+1 } )$ for each $i$.
  
  \item  The {\em source} and {\em target} of $\mu$ are $s( \mu ) = s( \alpha_1 )$ and $t( \mu ) = t( \alpha_p )$.  

  \item  If $\mu = \alpha_p \cdots \alpha_1$ and $\nu = \beta_q \cdots \beta_1$ are paths with $t( \mu ) = s( \nu )$, then their {\em composition} is the path $\nu \circ \mu = \beta_q \cdots \beta_1 \alpha_p \cdots \alpha_1$.
  
  \item  A path from $m$ to $m'$ is {\em minimal} if there is no path from $m$ to $m'$ of strictly smaller length.

\end{itemize}
\end{Definition}

\begin{Definition}
[Green paths]
\label{def:green_path}
Following \cite[2.7]{Br} and \cite[thm.\ 4.10]{BST}, a {\em green path in $\Gamma_{ \cC}$} is a path
\[
  m(1) \xrightarrow{} m(2) \xrightarrow{} \cdots \xrightarrow{} m(p-1) \xrightarrow{} m(p)
\]
such that
\[
  \silt[m(1)]{m(1)} > \silt[m(1)]{m(2)} > \cdots > \silt[m(1)]{m(p-1)} > \silt[m(1)]{m(p)}.
\]
\end{Definition}

The following definition uses free groupoids and categories modulo equivalence relations.  See \cite[secs.\ II.7+8]{MacLane} for background information.

\begin{Definition}
[The green groupoid]
\label{def:green_groupoid}
Let $\cF_{ \cC }$ be the free groupoid over $\Gamma_{ \cC }$, that is, the groupoid with:
\begin{itemize}
\setlength\itemsep{4pt}

  \item  Objects: The vertices of $\Gamma_{ \cC }$ (= the elements of $\mrig \cC$).

  \item  Morphisms: Compositions of the arrows of $\Gamma_{ \cC }$ and their formal inverses, modulo the obvious relations.

\end{itemize}
By abuse of notation, a path $\mu$ in $\Gamma_{ \cC }$ from $m$ to $m'$ will be considered to be an element of $\cF_{ \cC }( m,m' )$. 

Let $\equiv$ denote the equivalence relations on the sets $\cF_{ \cC }( m,m' )$ where all morphisms given by green paths from a fixed $m$ to a fixed $m'$ are equivalent to each other, while any other morphism is equivalent only to itself.  Let $\sim$ denote the equivalence relations on the sets $\cF_{ \cC }( m,m' )$ induced by $\equiv$ and the requirement that $\mu \sim \mu' \Rightarrow \nu \mu \lambda \sim \nu \mu' \lambda$ when these compositions make sense.

The {\em green groupoid of $\cC$} is
\[
  \cG_{ \cC } = \cF_{ \cC } / \sim.
\]
Note that $\cG_{ \cC }$ has the same objects as $\cF_{ \cC }$, and that there is a canonical functor $\cF_{ \cC } \xrightarrow{ Q_{ \cC } } \cG_{ \cC }$, which is the identity on objects and has the universal property
\[
\vcenter{
  \xymatrix @+0.5pc {
    \cF_{ \cC } \ar^-{ Q_{ \cC } }[r] \ar[d] & \cG_{ \cC } \ar^-{\exists!}@{-->}[dl] \\
    \cA \\
                    }
        }
\]
when $\cA$ is a category and the vertical arrow is a functor constant on the equivalence classes of $\sim$ (equivalently: Constant on the equivalence classes of $\equiv$).
\end{Definition}


\begin{Definition}
[The green group]
\label{def:green_group}
If the mutation graph of maximal rigid objects of $\cC$ is connected, then so is the quiver $\Gamma_{ \cC }$ and the groupoids $\cF_{ \cC }$ and $\cG_{ \cC }$.  In this case, all the the vertex groups $\cG_{ \cC }( m,m )$ are isomorphic by \cite[p.\ 8]{Higgins}, and we define
the {\em green group of $\cC$} to be any of these vertex groups.
\end{Definition}

\section{Silting}
\label{sec:silting}

This section does not use the Frobenius model $\cE$, but works under Setup \ref{set:blanket}(iv) alone.

The following result is due to \cite[prop.\ 3.10.2]{Pl} for cluster tilting objects in cluster categories.  The proof carries over to the present case.  The $\Hom$ space on the right hand side of the isomorphism is an instance of the $E$-invariant of Derksen and Fei \cite[def.\ 3.3]{DF}.

\begin{Lemma}
\label{lem:Plamondon}
Let $\ell \in \cC$ be a basic maximal rigid object, set $\underline{A} = \cC( \ell,\ell )$, and let $[\add_\cC \ell]$ denote the ideal of $\cC$ of morphisms factoring through an object of $\add_{ \cC }( \ell )$.  If $m,n \in \cC$ are rigid, then
\[
  [\add_\cC \ell]( \Sigma^{ -1 }m,n ) \cong \Hom_{ \Kb( \proj \underline{A} ) }\big( \silt[\ell]{m},\Sigma \silt[\ell]{n} \big).
\]
\end{Lemma}

The following Proposition provides an intrinsic characterization of the
silting partial order in terms of the category $\cC$ without a reference to tilted
algebras.

\begin{Proposition}
\label{prop:G6}
Let $\ell \in \cC$ be a basic maximal rigid object.  Let $m,n \in \cC$ be maximal rigid objects and let
\[
  m_1 \xrightarrow{} m_0 \xrightarrow{} n \xrightarrow{ \underline{ \delta } } \Sigma m_1
\]
be a triangle in $\cC$ with $m_i \in \add_{ \cC }( m )$ (at least one such triangle exists by Lemma \ref{lem:ZZ}).  Then
\begin{equation}
\label{equ:G6a}
  \silt[\ell]{m} \geqslant \silt[\ell]{n} \; \Leftrightarrow \; \cC( \ell,\underline{ \delta } ) = 0.
\end{equation}
\end{Proposition}

\begin{proof}
By Definition \ref{def:silting_order} the condition $\silt[\ell]{m} \geqslant \silt[\ell]{n}$ is equivalent to
\[
  \Hom_{ \Kb( \proj \underline{A} ) }\big( \silt[\ell]{m},\Sigma \silt[\ell]{n} \big) = 0,
\]  
which by Lemma \ref{lem:Plamondon} is equivalent to $[\add_\cC \ell]( \Sigma^{ -1 }m,n ) = 0$.  It is easy to see that this is equivalent to:
\begin{enumerate}
\setlength\itemsep{4pt}

  \item  Given $\underline{ \lambda } \in \cC( \ell,n )$ and $m' \in \add_{ \cC }( m )$, the induced map $\cC( \Sigma^{ -1 }m',\ell ) \xrightarrow{ \cC( \Sigma^{ -1 }m',\underline{ \lambda } ) } \cC( \Sigma^{ -1 }m',n )$ is $0$.
  
\end{enumerate}
The map in (i) is part of the following commutative square, which exists by the $2$-Calabi--Yau condition.
\[
\vcenter{
  \xymatrix @-0.9pc {
    \cC( \Sigma^{ -1 }m',\ell ) \ar^-{ \cC( \Sigma^{ -1 }m',\underline{ \lambda } ) }[rrrr] \ar^{\mbox{\rotatebox{270}{$\cong$}}}[ddd] &&&& \cC( \Sigma^{ -1 }m',n ) \ar^{\mbox{\rotatebox{270}{$\cong$}}}[ddd] \\
\\
\\
    \dual\!\cC( \ell,\Sigma m' ) \ar_-{ \dual\!\cC( \underline{ \lambda },\Sigma m' ) }[rrrr] &&&& \dual\!\cC( n,\Sigma m' ) \\
                    }
        }
\]
Hence (i) is equivalent to:
\begin{enumerate}
\setcounter{enumi}{1}
\setlength\itemsep{4pt}

  \item  Given $\underline{ \lambda } \in \cC( \ell,n )$ and $m' \in \add_{ \cC }( m )$, the induced map $\cC( n,\Sigma m' ) \xrightarrow{ \cC( \underline{ \lambda },\Sigma m' ) } \cC( \ell,\Sigma m' )$ is $0$.

\end{enumerate}
This is again equivalent to:
\begin{enumerate}
\setcounter{enumi}{2}
\setlength\itemsep{4pt}

  \item  Given $m' \in \add_{ \cC }( m )$, in each diagram
\[
\vcenter{
  \xymatrix @+0.3pc {
    && \ell \ar_-{ \underline{ \lambda } }[d] \\
    m_1 \ar[r] & m_0 \ar[r] & n \ar^-{ \underline{ \delta } }[r] \ar_-{ \underline{ \nu } }[d] & \Sigma m_1, \ar@{-->}[dl] \\
    && \Sigma m' \\
                    }
        }
\]
we have $\underline{ \nu }\underline{ \lambda } = 0$.

\end{enumerate}
Finally, (iii) is equivalent to the right hand side of \eqref{equ:G6a}, that is $\cC( \ell,\underline{ \delta } ) = 0$.  Namely, on the one hand, we can choose $\underline{ \nu } = \underline{ \delta }$ in the diagram, so (iii) implies $\cC( \ell,\underline{ \delta } ) = 0$.  On the other hand, the dashed arrow exists to give a commutative triangle because $\cC( m_0,\Sigma m' ) = 0$ since $m$ is rigid and $m_0,m' \in \add_{ \cC }( m )$, so $\cC( \ell,\underline{ \delta } ) = 0$ implies (iii).
\end{proof}

\section{Tilting}
\label{sec:tilting}

The following result is essentially due to \cite[prop.\ 4]{Pa}.  We include a proof because our setup is slightly different, involving maximal rigid objects.

\begin{Proposition}
\label{pro:G1}
If $\ell,m \in \cC$ are good maximal rigid objects, then $\cE( \ell,m )$, viewed as a complex of $\cE( m,m )$-$\cE( \ell,\ell )$-bimodules, is a two-sided tilting complex (see \cite[def.\ 3.4]{R}). 
\end{Proposition}

\begin{proof}
We will prove this by verifying that $\cE( \ell,m )$ satisfies criteria (a)--(c) of \cite[prop.\ 1.8]{M}.  We write $A = \cE( \ell,\ell )$ and $B = \cE( m,m )$.

By Lemma \ref{lem:F9_10} there is a conflation
\begin{equation}
\label{equ:G1b}
  0 \xrightarrow{} \ell_1 \xrightarrow{ \lambda_1 } \ell_0 \xrightarrow{} m \xrightarrow{} 0
\end{equation}
with $\ell_i \in \add_{ \cE }( \ell )$, which induces a short exact sequence
\begin{equation}
\label{equ:G1c}
  0 \xrightarrow{} \cE( \ell,\ell_1 ) \xrightarrow{ \cE( \ell,\lambda_1 ) } \cE( \ell,\ell_0 ) \xrightarrow{} \cE( \ell,m ) \xrightarrow{} 0.
\end{equation}
Hence, viewed as a complex of $A$-right modules, $\cE( \ell,m )$ is quasi-isomorphic to the complex
\[
  \cE( \ell,\ell_1 ) \xrightarrow{ \cE( \ell,\lambda_1 ) } \cE( \ell,\ell_0 ),
\]
which is a bounded complex of finitely generated projective $A$-right modules, therefore perfect.  It follows that, viewed as a complex of $A$-right modules, $\cE( \ell,m )$ is perfect.  A dual argument shows that, viewed as a complex of $B$-left modules, $\cE( \ell,m )$ is perfect.  This establishes criterion (a) of \cite[prop.\ 1.8]{M}.

The conflation \eqref{equ:G1b} induces a triangle $\Sigma^{ -1 }m \xrightarrow{} \ell_1 \xrightarrow{ \underline{ \lambda }_1 } \ell_0 \xrightarrow{} m$ in $\cC$.  We have $\cC( \Sigma^{ -1 }m,m ) = 0$ because $m$ is rigid in $\cC$, so $\cC( \underline{ \lambda }_1,m )$ is surjective whence $\cE( \lambda_1,m )$ is surjective by Lemma \ref{lem:portmanteau}(viii).  Hence \eqref{equ:G1b} induces a short exact sequence, which is the first column of the following diagram.  The second column is a long exact sequence obtained from \eqref{equ:G1c}.
\[
\vcenter{
  \xymatrix @-0.5pc {
    0 \ar[d] && 0 \ar[d] \\  
    \cE( m,m ) \ar[d] \ar^-{\varepsilon}[rr] && \Hom_A \big( \cE( \ell,m ),\cE( \ell,m ) \big) \ar[d] \\
    \cE( \ell_0,m ) \ar_{ \cE( \lambda_1,m ) }[d] \ar^-{\cong}[rr] && \Hom_A \big( \cE( \ell,\ell_0 ),\cE( \ell,m ) \big) \ar[d] \\
    \cE( \ell_1,m ) \ar[d] \ar^-{\cong}[rr] && \Hom_A \big( \cE( \ell,\ell_1 ),\cE( \ell,m ) \big) \ar[d] \\
    0 && \Ext_A^1 \big( \cE( \ell,m ),\cE( \ell,m ) \big) \ar[d] \\
      && 0 \\
                    }
        }
\]
The horizontal maps are obtained from Lemma \ref{lem:portmanteau}(iv), which also gives that the two last ones are bijective.  Hence by the diagram $\varepsilon$ is bijective and $\Ext_A^1\!\big( \cE( \ell,m ),\cE( \ell,m ) \big) = 0$.  Combining with Lemma \ref{lem:F9_10}(iii) shows that the canonical map
\[
  B \xrightarrow{} \RHom_A\!\big( \cE( \ell,m ),\cE( \ell,m ) \big)
\]
is an isomorphism.  This establishes criterion (b) of \cite[prop.\ 1.8]{M}.

A dual argument shows that the canonical map
\[
  A \xrightarrow{} \RHom_{ B^{ \opp } }\!\big( \cE( \ell,m ),\cE( \ell,m ) \big)
\]
is an isomorphism.  This establishes criterion (c) of \cite[prop.\ 1.8]{M}.
\end{proof}

\begin{Lemma}
\label{lem:G10}
Let $\ell,m,n \in \cC$ be maximal rigid objects, and suppose that $\ell$ is good and basic and $m$ is good.  Then
\[
  \silt[\ell]{m} \geqslant \silt[\ell]{n}
  \; \Rightarrow \;
  \cE( m,n ) \LTensor{ \cE( m,m ) } \cE( \ell,m ) \cong \cE( \ell,n ),
\]
where the isomorphism is in $\D\big( \cE( n,n )^{ \opp } \Tensor{k} \cE( \ell,\ell ) \big)$, the derived category of $\cE( n,n )$-$\cE( \ell,\ell )$-bimodules.
\end{Lemma}

In the case where $\cE$ is the category of maximal Cohen--Macaulay modules over the base of a 3-fold flopping contraction, 
the isomorphism in the lemma first appeared in \cite[thm.\ 4.6(2)]{HW}.  That paper used the partial order of tilting modules over 
$\cE( \ell,\ell )$, which is not available for our general Frobenius category $\cE$.  Instead we use the partial order of $2$-term silting objects over $\cC( \ell,\ell )$.

\begin{proof}[Proof of Lemma \ref{lem:G10}]
Set $B = \cE( m,m )$ and consider $\cE( m,n ) \LTensor{ \cE( m,m ) } \cE( \ell,m ) = \cE( m,n ) \LTensor{ B } \cE( \ell,m )$.  Its zeroth homology is $\cE( m,n ) \Tensor{ B } \cE( \ell,m )$, so it is enough to show:
\begin{enumerate}
\setlength\itemsep{4pt}

  \item  $\cE( m,n ) \Tensor{ B } \cE( \ell,m ) \cong \cE( \ell,n )$ as $\cE( n,n )$-$\cE( \ell,\ell )$-bimodules.
  
  \item  $\H\!\big( \cE( m,n ) \LTensor{ B } \cE( \ell,m ) \big)$ is concentrated in degree $0$.

\end{enumerate}
	
By Lemma \ref{lem:F9_10} there exists a conflation in $\cE$ with $m_i \in \add_{ \cE }( m )$:
\begin{equation}
\label{equ:G10a}
  0 \xrightarrow{} m_1 \xrightarrow{} m_0 \xrightarrow{ \mu_0 } n \xrightarrow{} 0,
\end{equation}
which induces a short exact sequence of $B$-right modules:
\begin{equation}
\label{equ:G10b}
  0 \xrightarrow{} \cE( m,m_1 ) \xrightarrow{} \cE( m,m_0 ) \xrightarrow{} \cE( m,n ) \xrightarrow{} 0.
\end{equation}
This is an augmented projective resolution of $\cE( m,n )$ over $B$, whence $\cE( m,n ) \LTensor{ B } \cE( \ell,m )$ is isomorphic to
\begin{equation}
\label{equ:G10c}
  \cE( m,m_1 ) \Tensor{ B } \cE( \ell,m ) \xrightarrow{} \cE( m,m_0 ) \Tensor{ B } \cE( \ell,m )
\end{equation}
in $\cD\big( \cE( \ell,\ell ) \big)$.  Hence condition (ii) is equivalent to:
\begingroup
\renewcommand{\labelenumi}{(\roman{enumi})'}
\begin{enumerate}
\setcounter{enumi}{1}
\setlength\itemsep{4pt}
  
  \item  The homology of \eqref{equ:G10c} is concentrated in degree $0$.

\end{enumerate}
\endgroup

The conflation \eqref{equ:G10a} also induces a triangle in $\cC$:
\[
  m_1 \xrightarrow{} m_0 \xrightarrow{ \underline{ \mu }_0 } n \xrightarrow{ \underline{ \delta } } \Sigma m_1,
\]
and $m_i \in \add_{ \cC }( m )$ by Lemma \ref{lem:portmanteau}(ii).  Hence $\silt[\ell]{m} \geqslant \silt[\ell]{n}$ implies $\cC( \ell,\underline{ \delta } ) = 0$ by Proposition \ref{prop:G6}.  It follows that $\cC( \ell,\underline{ \mu }_0 )$ is surjective, so $\cE( \ell,\mu_0 )$ is surjective by Lemma \ref{lem:portmanteau}(viii).  This gives the exact sequence in the second column of the following diagram, and the exact sequence in the first column is induced by \eqref{equ:G10b}.
\[
\vcenter{
  \xymatrix @-0.5pc {
    && 0 \ar[d] \\
    \cE( m,m_1 ) \Tensor{B} \cE( \ell,m ) \ar^-{ \cong }[rr] \ar[d] && \cE( \ell,m_1 ) \ar[d] \\
    \cE( m,m_0 ) \Tensor{B} \cE( \ell,m ) \ar^-{ \cong }[rr] \ar[d] && \cE( \ell,m_0 ) \ar^{ \cE( \ell,\mu_0 ) }[d] \\
    \cE( m,n ) \Tensor{B} \cE( \ell,m ) \ar^-{ \pi }[rr] \ar[d] && \cE( \ell,n ) \ar[d] \\
    0 && 0 \\
                    }
        }
\]
The horizontal maps are obtained from Lemma \ref{lem:portmanteau}(v), which also gives that the two first ones are bijective.  Hence the diagram shows that $\pi$ is bijective, proving (i), and that the homology of \eqref{equ:G10c} is concentrated in degree $0$, proving (ii)'.
\end{proof}

\section{The action of the green groupoid on derived categories}
\label{sec:action}

\begin{Remark}
\label{rmk:G11}
For the following theorem and its proof, recall the derived Picard groupoid $\cP_{ \cE }$, the free groupoid $\cF_{ \cC }$ over the quiver $\Gamma_{ \cC }$, and the green groupoid $\cG_{ \cC }$, see Definitions \ref{def:derived_Picard_groupoid}, \ref{def:Gamma}, and \ref{def:green_groupoid}.
All three groupoids have the same sets of objects, namely $\mrig \cC$.  
\end{Remark}

\begin{Theorem}
\label{thm:G11}
\begin{enumerate}
\setlength\itemsep{4pt}

  \item  There is a unique functor $\cF_{ \cC } \xrightarrow{ F } \cP_{ \cE }$, which is the identity on objects and determined on morphisms by
\[
  F( \alpha ) = \mbox{the isomorphism class of } \cE( m,m^{ \ast } )
\]
when $m \xrightarrow{ \alpha } m^{ \ast }$ is an arrow in $\Gamma_{ \cC }$.

  \item  If $\mu = m(1) \xrightarrow{} \cdots \xrightarrow{} m(p)$ is a green path in $\Gamma_{ \cC }$, then
\[
  F( \mu ) = \mbox{the isomorphism class of } \cE\big( m(1),m(p) \big).
\]

  \item  The functor $F$ is constant on the equivalence classes of the relations $\sim$ from Definition \ref{def:green_groupoid}, hence has a unique factorisation $G$ through the green groupoid:
\[
\vcenter{
  \xymatrix @+0.5pc {
    \cF_{ \cC } \ar^-{ Q_{ \cC } }[r] \ar_-{F}[d] & \cG_{ \cC } \: \lefteqn{= \cF_{ \cC }/\sim } \ar^<<<<<{G}@{-->}[dl] \\
    \cP_{ \cE }. \\
                    }
        }
\]

\end{enumerate}

\end{Theorem}

\begin{proof}
(i):  This holds by the following observations.
\begin{itemize}
\setlength\itemsep{4pt}

  \item  By Definitions \ref{def:derived_Picard_groupoid} and \ref{def:Gamma}, the groupoids $\cF_{ \cC }$ and $\cP_{ \cE }$ have the same sets of objects, namely $\mrig \cC$.

  \item  By Definitions \ref{def:Gamma} and \ref{def:green_groupoid}, the groupoid $\cF_{ \cC }$ is free on the quiver $\Gamma_{ \cC }$, which has the arrows $m \xrightarrow{} m^{ \ast }$.

  \item  By Proposition \ref{pro:G1} there is a two-sided tilting complex ${}_{ \cE( m^*,m^* ) }\cE( m,m^{ \ast } )_{ \cE( m,m ) }$, whose isomorphism class is a morphism in $\cP_{ \cE }( m,m^{ \ast } )$.  

\end{itemize}

(ii):  If $\mu = m(1) \xrightarrow{ \alpha_1 } \cdots \xrightarrow{ \alpha_{ p-1 } } m(p)$ is a green path in $\Gamma_{ \cC }$ and we set $A_i = \cE\big( m(i),m(i) \big)$, then 
\begin{align*}
  F( \mu ) & = F( \alpha_{ p-1 } \circ \cdots \circ \alpha_1 ) \\
  & = F( \alpha_{ p-1 } ) \circ \cdots \circ F( \alpha_1 ) \\
  & \cong \cE\big( m(p-1),m(p) \big) \LTensor{ A_{ p-1 } } \cdots \LTensor{ A_2 } \cE\big( m(1),m(2) \big) \\
  & \cong \cE\big( m(1),m(p) \big),
\end{align*}
where the last $\cong$ is by repeated application of Lemma \ref{lem:G10}.

(iii):  To show that $F$ is constant on the equivalence classes of $\sim$, it is enough to show that $F$ is constant on the equivalence classes of the relations $\equiv$ from Definition \ref{def:green_groupoid}, but this holds by part (ii).
\end{proof}

\begin{Notation}
\label{not:G11}
By abuse of notation, we henceforth write e.g.\ $F( m \xrightarrow{ \alpha } m^{ \ast } ) = \cE( m,m^{ \ast } )$, ignoring that the output of $F$ is a morphism in $\cP_{ \cE }$, hence only given up to isomorphism.
\end{Notation}

\begin{Corollary}
\label{cor:G11i}
For each $m \in \mrig \cC$, there is a group homomorphism $\cG_{ \cC }( m,m ) \xrightarrow{} \DPic \cE( m,m )$.
\end{Corollary}

\begin{proof}
The functor $\cG_{ \cC } \xrightarrow{ G } \cP_{ \cE }$ of Theorem \ref{thm:G11} induces a homomorphism $\cG_{ \cC }( m,m ) \xrightarrow{ G(-) } \cP_{ \cE }( m,m )$  of vertex groups whose target is $\DPic \cE( m,m )$ by Definition \ref{def:derived_Picard_groupoid}.
\end{proof}

Note that if the mutation graph of maximal rigid objects of $\cC$ is connected, then $\cG_{ \cC }( m,m )$ is the green group of $\cC$, which is unique up to isomorphism, see Definition \ref{def:green_group}.

\begin{Remark}
We will show later that $F$, and hence $G$, are essentially never trivial.  Indeed, suppose the mild technical condition \somefiniteness\ of Definition \ref{def:somefiniteness} is satisfied.  If $\Gamma_{ \cC }$ contains the configuration $
  \xymatrix {
  n
    \ar^{ \beta }[r]<0.5ex> &
    m
    \ar^{ \alpha }[l]<0.5ex>
            }
$
and we consider the path $\mu = \beta\alpha$ from $m$ to $m$ in $\Gamma_{ \cC }$, then $F( \mu )$ is an element of infinite order in the group $\DPic\big( \cE( m,m ) \big)$ by Observation \ref{obs:F22}.
\end{Remark}

\section{The $g$-vector fan}
\label{sec:g-vectors}

This section does not use the Frobenius model $\cE$, but works under Setup \ref{set:blanket}(iv) alone (in which case $\mrig \cC$ should just be a set containing one object from each isomorphism class of basic maximal rigid objects of $\cC$).

Section \ref{sec:green_paths_geometry} will investigate the geometrical properties of green paths in terms of the $g$-vector fan.  To prepare, this section will recall a number of known results on the $g$-vector fan of $\cC$ with respect to a basic maximal rigid object.  

None of the material of this section is original; see in particular \cite[sec.\ 6]{DIJ}.

\begin{Remark}
\label{rmk:K}
Let $\ell \in \cC$ be a basic maximal rigid object and set $\underline{A} = \cC( \ell,\ell )$.

Let $\ell \cong \ell_1 \oplus \cdots \oplus \ell_d$ be an indecomposable decomposition in $\cC$.  Applying the functor $\cC( \ell,- )$ gives an indecomposable decomposition $\underline{ A }_{ \underline{ A } } \cong P_1 \oplus \cdots \oplus P_d$ in $\mod \underline{ A }$.

The split Grothendieck group $\Ksp( \add_{ \cC }\ell )$ is a free abelian group with basis $\{ [\ell_1], \ldots, [\ell_d] \}$.

The homotopy category $\Kb( \proj \underline{ A } )$ has the silting object $\underline{ A }_{ \underline{ A } }$ by \cite[exa.\ 2.2(a)]{AI}, so the triangulated Grothendieck group $\K0\!\big( \Kb( \proj \underline{ A } ) \big)$ is a free abelian group with basis $\{ [P_1], \ldots, [P_d] \}$ by \cite[thm.\ 2.27]{AI}.  Note that square brackets denote class in the relevant Grothendieck group.

There are $\BR$-vector spaces
\begin{align*}
  \Ksp( \add_{ \cC }\ell )_{ \BR } & = \Ksp( \add_{ \cC }\ell ) \Tensor{ \BZ } \BR, \\
  \K0\!\big( \Kb( \proj\,\underline{ A } ) \big)_{ \BR } & = \K0\!\big( \Kb( \proj\,\underline{ A } ) \big) \Tensor{ \BZ } \BR
\end{align*}
and an isomorphism 
\begin{equation}
\label{equ:Ka}
  \xymatrix @+0.5pc {
    \Ksp( \add_{ \cC }\ell )_{ \BR } \ar_-{ \Phi }^-{\sim}[r] & \K0\!\big( \Kb( \proj\,\underline{ A } ) \big)_{ \BR }
                    }
  \;\;,\;\; \Phi( [\ell_i] ) = [P_i].
\end{equation}
\end{Remark}

\begin{Definition}
[The index]
\label{def:index}
Let $\ell \in \cC$ be a basic maximal rigid object, $m \in \cC$ a rigid object, and recall the triangle in Lemma \ref{lem:ZZ}.   The {\em index of $m$ with respect to $\ell$} is
\[
  \index_{ \ell }( m ) = [\ell_0] - [\ell_1].
\] 
It is a well-defined element of $\Ksp( \add_{ \cC }\ell )$.
\end{Definition}

\begin{Definition}
[The $g$-vector fan of $\cC$]
\label{def:cone_B}
Let $\ell \in \cC$ be a basic maximal rigid object with indecomposable decomposition $\ell \cong \ell_1 \oplus \cdots \oplus \ell_d$.

If $m \in \cC$ is rigid with indecomposable decomposition $m \cong m_1 \oplus \cdots \oplus m_e$, then the {\em chamber of $m$ with respect to $\ell$} is
\begin{equation}
\label{equ:Bnotation}
  B^{ \ell }( m ) = \Big\{\, \sum_{i} a_i \index_{ \ell }( m_i ) \,\Big|\, a_i \in \BR_{ \geqslant 0 } \Big\},
\end{equation}
which is a cone in $\Ksp( \add_{ \cC }\ell )_{ \BR }$.

The {\em g-vector fan of $\cC$ with respect to $\ell$} is the set of chambers
\[
  G^{ \ell }_{ \cC } = \{\, B^{ \ell }( m ) \mid \mbox{$m \in \mrig \cC$} \,\}
\]
in $\Ksp( \add_{ \cC }\ell )_{ \BR }$.  The {\em positive and negative chambers} are
\begin{align*}
  B^{ \ell }_+ & = B^{ \ell }( \ell ) = \Big\{\, \sum_{i} a_i [\ell_i] \,\Big|\, a_i \in \BR_{ \geqslant 0 } \Big\}, \\
  B^{ \ell }_- & = B^{ \ell }( \Sigma \ell ) = \Big\{\, \sum_{i} -a_i [\ell_i] \,\Big|\, a_i \in \BR_{ \geqslant 0 } \Big\}.
\end{align*}
\end{Definition}

\begin{Definition}
[The $g$-vector fan of $\underline{ A }$]
\label{def:cone_C}
Let $\underline{ A }$ be a basic finite dimensional $k$-algebra with indecomposable decomposition $\underline{ A }_{ \underline{ A } } \cong P_1 \oplus \cdots \oplus P_d$.

If $Q \in \Kb( \proj\,\underline{ A } )$ has indecomposable decomposition $Q \cong Q_1 \oplus \cdots \oplus Q_e$, then the {\em chamber of $Q$} is
\[
  C( Q ) = \Big\{\, \sum_{i} a_i [Q_i] \,\Big|\, a_i \in \BR_{ \geqslant 0 } \Big\},
\]
which is a cone in $\K0\!\big( \Kb( \proj\,\underline{ A } ) \big)_{ \BR }$.

The {\em g-vector fan of $\underline{ A }$} is the set of chambers
\[
  \{\, C( Q ) \mid \mbox{$Q \in \2silt \underline{ A }$} \,\}
\]
in $\K0\!\big( \Kb( \proj\,\underline{ A } ) \big)_{ \BR }$.  The {\em positive and negative chambers} are
\begin{align*}
  C_+ & = C( \underline{ A } ) = \Big\{\, \sum_{i} a_i [P_i] \,\Big|\, a_i \in \BR_{ \geqslant 0 } \Big\}, \\
  C_- & = C( \Sigma \underline{ A } ) = \Big\{\, \sum_{i} -a_i [P_i] \,\Big|\, a_i \in \BR_{ \geqslant 0 } \Big\}.
\end{align*}
\end{Definition}

\begin{Definition}
[Simplicial cones and chamber decompositions]
A {\em simplicial cone} in $\BR^d$ is a subset of the form 
\[
\{\, a_1 v_1 + \cdots + a_d v_d \mid a_i \in \BR_{ \geqslant 0 } \} \subset \BR^d,
\]
where $\{ v_1, \ldots, v_d \}$ in $\BR^d$ are linearly independent vectors.

A {\em chamber decomposition} of $\BR^d$ is a finite collection of simplicial cones $S = \{ B_1, \ldots, B_N \}$ whose interiors are pairwise disjoint, such that each codimension one face of a cone $B_i$ is the face of exactly one other $B_j$ and $\cup_i B_i = \BR^d$.
\end{Definition}

For a chamber decomposition $S=\{ B_1, \ldots, B_N \}$ of $\BR^d$, we automatically have $\intt(B_i) \cap \intt (B_j) = \varnothing$ for $i \neq j$, where $\intt X$ denotes the interior of a subset $X \subseteq \BR^d$.
Furthermore 
\begin{equation}
\label{cone-decomp}
\BR^d = \intt(B_1) \sqcup \ldots \sqcup \intt(B_N) \sqcup F
\end{equation}
where $F$ is the union of all codimension one faces of $S$.

\begin{Lemma}
\label{lem:K}
Let $\ell \in \cC$ be a basic maximal rigid object and set $\underline{A} = \cC( \ell,\ell )$.
\begin{enumerate}
\setlength\itemsep{4pt}

  \item  There is a commutative square
\[
\vcenter{
  \xymatrix @+0.5pc {
    \{ \mbox{rigid objects in $\cC$} \} \ar^-{ \siltlittle[\ell]{-} }[r] \ar_{ \index_{ \ell } }[d] & \Kb( \proj\,\underline{ A } ) \ar^{ [-] }[d] \\
    \Ksp( \add_{ \cC }\ell )_{ \BR } \ar_-{ \Phi }^-{\sim}[r] & \K0\!\big( \Kb( \proj\,\underline{ A } ) \big)_{ \BR } \lefteqn{.}
                    }
        }
\]

  \item  If $m \in \mrig \cC$ then $\Phi \big( B^{ \ell }( m ) \big) = C\big( \siltlittle[\ell]{m} \big)$.
   
\end{enumerate}
\end{Lemma}

\begin{proof}
(i): Combine Definitions \ref{def:silt} and \ref{def:index} and Equation \eqref{equ:Ka}.

(ii): Combine part (i) with Lemma \ref{lem:AIR} and Definitions \ref{def:cone_B} and \ref{def:cone_C}.
\end{proof}

\begin{Proposition}
\label{pro:cones1}
Let $\ell \in \cC$ be a basic maximal rigid object.
\begin{enumerate}
\setlength\itemsep{4pt}

  \item  If $m \in \mrig \cC$ then $B^{ \ell }( m )$ is a simplicial cone.
  
  \item  If $m \neq m'$ are in $\mrig \cC$ then the interiors of $B^{ \ell }( m )$ and $B^{ \ell }( m' )$ are disjoint.
  
  \item  Two objects $m,m^{ \ast } \in \mrig \cC$ differ by mutation in $\cC$ if and only if $B^{ \ell }( m ), B^{ \ell }( m^{ \ast } )$ are {\em neighbours}, that is, intersect in a common face $F$ of codimension one.

  \item  If $m,m^{ \ast } \in \mrig \cC$ differ by mutation in $\cC$, then the face $F$ from (iii) generates a hyperplane $\linearspan_{ \BR }( F )$, which defines two closed half spaces in $\Ksp( \add_{ \cC }\ell )_{ \BR }$.  One half space (called {\em positive}) contains $B^{ \ell }_+$, and the other half space (called {\em negative}) contains $B^{ \ell }_-$.  There is the following dichotomy:
\medskip
\begin{enumerate}
\setlength\itemsep{4pt}

  \item  Either $B^{ \ell }( m )$ is contained in the positive half space, and $\siltlittle[\ell]{m} > \siltlittle[\ell]{ m^{ \ast } }$, 
  
  \item  or $B^{ \ell }( m )$ is contained in the negative half space, and $\siltlittle[\ell]{m} < \siltlittle[\ell]{ m^{ \ast } }$.

\end{enumerate}

  \item  The category $\cC$ is maximal rigid finite if and only if $\{ B^{ \ell }( m ) \mid m \in \mrig \cC \}$ is a chamber decomposition of $\K0( \add_{ \cC } \ell )_{ \BR }$.

%
%
%
%
  
\end{enumerate}
\end{Proposition}

\begin{proof}
Each of the properties (i)--(v) is obtained by combining Lemma \ref{lem:K}(ii) with the corresponding property of the chambers $C\big( \siltlittle[\ell]{m} \big)$, noting that $\siltlittle[\ell]{-}$ gives the bijective correspondence $\mrig \cC \xrightarrow{} \2silt \underline{ A }$ of Lemma \ref{lem:AIR}.  For (i), the corresponding property of the $C\big( \siltlittle[\ell]{m} \big)$ is \cite[prop.\ 3.5(2)]{ZZ}, for (ii) and (iii) it is \cite[cor.\ 6.7(b)]{DIJ}, for (iv) it is \cite[prop.\ 6.10 and thm.\ 6.11]{DIJ}, and for (v) it is \cite[thm.\ 4.7]{Asai}.
\end{proof}

Let us be more systematic about the terminology, in particular the terms introduced in the proposition:

\begin{Definition}
\label{def:fan_terminology}
\begin{itemize}
\setlength\itemsep{4pt}

  \item  Two chambers are {\em neighbours} if they intersect in a codimension one face.

  \item  A {\em gallery of length $p-1$} from $B$ to $B'$ is a sequence of chambers $B = B(1) \xrightarrow{} \cdots \xrightarrow{} B(p) = B'$ such that $B(i)$ and $B(i+1)$ are neighbours for each $i$.  

  \item  A gallery from $B$ to $B'$ is {\em minimal} if there is no gallery from $B$ to $B'$ of strictly smaller length.

  \item  In Proposition \ref{pro:cones1}(iv), the two half spaces are called the {\em positive half space with respect to $\ell$ of the mutation $m \xrightarrow{} m^{ \ast }$,} respectively the {\em negative half space with respect to $\ell$ of the mutation $m \xrightarrow{} m^{ \ast }$.}

\end{itemize}
\end{Definition}

\begin{Observation}
\label{obs:path1}
If $\ell \in \cC$ is a basic maximal rigid object, then by Proposition \ref{pro:cones1}(iii) the following are equivalent.
\begin{enumerate}
\setlength\itemsep{4pt}

  \item  $\mu = m(1) \xrightarrow{} \cdots \xrightarrow{} m(p)$ is a path in $\Gamma_{ \cC }$.
  
  \item  $B^{ \ell }( \mu ) = B^{ \ell }\big( m(1) \big) \xrightarrow{} \cdots \xrightarrow{} B^{ \ell }\big( m(p) \big)$ is a gallery in $G^{ \ell }_{ \cC }$.  

\end{enumerate}
%
%
%
\end{Observation}

\section{Geometrical properties of green paths}
\label{sec:green_paths_geometry}

This section does not use the Frobenius model $\cE$, but works under Setup \ref{set:blanket}(iv) alone (in which case $\mrig \cC$ should just be a set containing one object from each isomorphism class of basic maximal rigid objects of $\cC$).

This section investigates the geometrical properties of green paths in terms of the $g$-vector fan. Recall that a hyperplane arrangement in $\BR^d$ is a finite collection 
of hyperplanes in $\BR^d$. A hyperplane arrangement $\cH$ in $\BR^d$ is called simplicial if it is central (every hyperplane of $\cH$ contains $0$), essential (the intersection of all hyperplanes of $\cH$ is $\{0\}$), and each cone defined by $\cH$ is simplicial.

\begin{Lemma}
\label{lem:G8}
The following are equivalent.
\begin{enumerate}
\setlength\itemsep{4pt}

  \item  $\mu = m(1) \xrightarrow{} \cdots \xrightarrow{} m(p)$ is a green path in $\Gamma_{ \cC }$.

  \item  $B^{ m(1) }\big( \mu \big) = B^{ m(1) }\big( m(1) \big) \xrightarrow{} \cdots \xrightarrow{} B^{ m(1) }\big( m(p) \big)$ is a gallery in $G^{ m(1) }_{ \cC }$ such that each $B^{ m(1) }\big( m(i) \big)$ is contained in the positive half space with respect to $m(1)$ of the mutation $m(i) \xrightarrow{} m(i+1)$.   

\end{enumerate}
\end{Lemma}

\begin{proof}
By Proposition \ref{pro:cones1}(iv), condition (ii) is equivalent to $B^{ m(1) }( \mu )$ being a gallery in $G^{ m(1) }_{ \cC }$ such that $\silt[m(1)]{m(i)} > \silt[m(1)]{m(i+1)}$ for each $i$.  By Definition \ref{def:green_path} and Observation \ref{obs:path1} this is equivalent to (i).
\end{proof}

\begin{Definition}
[Condition {\hyper}]
\label{def:condition_hyper}
We say that condition \hyper\ is satisfied if,
for each basic maximal rigid object $\ell \in \cC$, the cones of  the $g$-vector fan $G^{ \ell }_{ \cC }$ are those of a hyperplane arrangement $\cH^{ \ell }$.  
\end{Definition}

Every 2-Calabi--Yau category $\cC$ that satisfies condition \hyper\ is automatically  maximal rigid finite (i.e.\ $|\mrig \cC|< \infty$). 

Moreover, if $\cC$ satisfies \hyper, then the corresponding hyperplane arrangements $\cH^{ \ell }$ are necessarily simplicial. 
Indeed, in this case the hyperplanes of $\cH^{ \ell }$ are the hyperplanes determined by all faces of the cones $B^{\ell}(m)$, so
$\cH^{ \ell }$  is central. Obviously $\cH^{ \ell }$ is essential if the number $d$ of indecomposable summands of $\ell$ equals $1$. 
Consider the case $d>1$. If $\cH^{ \ell }$ is not essential, then there exists an indecomposable rigid object $n \in \cC$ such that 
every $m \in \mrig \cC$ has a direct summand isomorphic to $n$. This is impossible since we can mutate $m$ at that summand 
to get rid of it, because $d >1$. 

\begin{Lemma}
\label{lem:H5}
If condition {\rm \hyper} is satisfied, then the following are equivalent.
\begin{enumerate}
\setlength\itemsep{4pt}

  \item  $\mu = m(1) \xrightarrow{} \cdots \xrightarrow{} m(p)$ is a minimal path in $\Gamma_{ \cC }$.
  
  \item  $\mu = m(1) \xrightarrow{} \cdots \xrightarrow{} m(p)$ is a green path in $\Gamma_{ \cC }$.
  
  \item  $B^{ m(1) }( \mu ) = B^{ m(1) }\big( m(1) \big) \xrightarrow{} \cdots \xrightarrow{} B^{ m(1) }\big( m(p) \big)$ is a minimal gallery in $G^{ m(1) }_{ \cC }$.

  \item  If $\ell \in \cC$ is a basic maximal rigid object, then $B^{ \ell }( \mu ) = B^{ \ell }\big( m(1) \big) \xrightarrow{} \cdots \xrightarrow{} B^{ \ell }\big( m(p) \big)$ is a minimal gallery in $G^{ \ell }_{ \cC }$.

\end{enumerate}
\end{Lemma}

\begin{proof}
(iii)$\Leftrightarrow$(i)$\Leftrightarrow$(iv): See Definitions \ref{def:Gamma} and \ref{def:fan_terminology} and Observation \ref{obs:path1}.

(ii)$\Leftrightarrow$(iii):  Let $\mu = m(1) \xrightarrow{} \cdots \xrightarrow{} m(p)$ be a path in $\Gamma_{ \cC }$ and consider the gallery $B^{ m(1) }( \mu ) = B^{ m(1) }\big( m(1) \big) \xrightarrow{} \cdots \xrightarrow{} B^{ m(1) }\big( m(p) \big)$, see Observation \ref{obs:path1}.  Combining Proposition \ref{pro:cones1}(iii)+(iv) with Definition \ref{def:condition_hyper} shows:
\begin{itemize}
\setlength\itemsep{4pt}

  \item  $B^{ m(1) }\big( m(i) \big)$ and $B^{ m(1) }\big( m(i+1) \big)$ intersect in a codimension one face, which is contained in a hyperplane $H_i \in \cH^{ m(1) }$.

  \item  $H_i$ separates the positive and the negative half spaces with respect to $m(1)$ of the mutation $m(i) \xrightarrow{} m(i+1)$. 

\end{itemize}
Hence there following statements are equivalent:
\begingroup
\renewcommand{\labelenumi}{(\alph{enumi})}
\begin{enumerate}
\setlength\itemsep{4pt}

  \item  Each $B^{ m(1) }\big( m(i) \big)$ is contained in the positive half space with respect to $m(1)$ of the mutation $m(i) \xrightarrow{} m(i+1)$.

  \item  Each $B^{ m(1) }\big( m(i) \big)$ is on the same side of $H_i$ as $B^{ m(1) }_+$.

\end{enumerate}
\endgroup
Condition (a) is equivalent to (ii) by Lemma \ref{lem:G8}.  On the other hand, the first chamber in the gallery $B^{ m(1) }( \mu )$ is $B^{ m(1) }\big( m(1) \big) = B^{ m(1) }_+$.  Hence condition (b) is equivalent to the condition that the hyperplanes $H_1, \ldots, H_{ p-1 }$ are pairwise different.  By \cite[lem.\ 4.2]{Paris2} this is equivalent to (iii).
\end{proof}

\begin{Theorem}
\label{thm:green-Delign}
If condition {\rm \hyper} is satisfied and $\ell \in \cC$ is a basic maximal rigid object, then the green groupoid $\cG_{ \cC }$ is isomorphic to the Deligne groupoid 
(see \cite{D}) of the hyperplane arrangement $\cH^{ \ell }$ from Definition \ref{def:condition_hyper}.
\end{Theorem}

The idea to relate groupoids of derived autoequivalences to Deligne groupoids first appeared in \cite{DW}.

\begin{proof}[Proof of Theorem \ref{thm:green-Delign}]
Recall from Definitions \ref{def:Gamma} and \ref{def:cone_B} the quiver $\Gamma_{ \cC }$ and the $g$-vector fan $G_{ \cC }^{ \ell }$.  Applying $B^{ \ell }( - )$ as defined in Equation \eqref{equ:Bnotation} has the following effects:
\begin{itemize}
\setlength\itemsep{4pt}

  \item  Vertices $m$ in $\Gamma_{ \cC }$ are mapped bijectively to chambers $B^{ \ell }( m )$  in $G_{ \cC }^{ \ell }$.
  
  \item  (Green) paths $\mu$ in $\Gamma_{ \cC }$ are mapped bijectively to (minimal) galleries $B^{ \ell }( \mu )$ in $G^{ \ell }_{ \cC }$, see Observation \ref{obs:path1} and Lemma \ref{lem:H5}.

\end{itemize}
Hence, applying $B^{ \ell }( - )$ translates $\Gamma_{ \cC }$ and $\cG_{ \cC }$, as defined in Section \ref{sec:green}, to $\Gamma( \cH^{ \ell } )$ and the Deligne groupoid of $\cH^{ \ell }$. We refer the reader to \cite[sec.\ 2]{Paris1} for  a detailed exposition.
\end{proof}

\begin{Lemma}
\label{lem:G12}
Assume that $\cC$ is maximal rigid finite.  Given $m,m' \in \mrig \cC$, there is a green path $m = m(1) \xrightarrow{} \cdots \xrightarrow{} m(p) = m'$ in $\Gamma_{ \cC }$.  
\end{Lemma}

\begin{proof}
Set $m(1) = m$ and $\underline{ A } = \cC\big( m(1),m(1) \big)$, and consider
\begin{equation}
\label{equ:G12a}
  \mbox{a green path $m(1) \xrightarrow{} \cdots \xrightarrow{} m(r)$ such that $\silt[m(1)]{m(r)} \geqslant \silt[m(1)]{m'}$.}
\end{equation}
This exists for $r=1$ because $\silt[m(1)]{m(1)} = \underline{ A }$ is maximal in $\2silt \underline{ A }$ whence $\silt[m(1)]{m(1)} \geqslant \silt[m(1)]{m'}$.  Suppose that \eqref{equ:G12a} has been defined for some $r \geqslant 1$.

If $m(r) = m'$ then \eqref{equ:G12a} gives the path claimed in the lemma with $p = r$.  

If $m(r) \neq m'$, then \eqref{equ:G12a} can be extended by one step: Lemma \ref{lem:AIR} implies $\silt[m(1)]{m(r)} \neq \siltlittle[m(1)]{m'}$, so we have $\silt[m(1)]{m(r)} > \siltlittle[m(1)]{m'}$.  By \cite[prop.\ 2.36]{AI} there is a left mutation $P$ of $\silt[m(1)]{m(r)}$ such that $\silt[m(1)]{m(r)} > P \geqslant \siltlittle[m(1)]{m'}$.  Since $\silt[m(1)]{m(r)}$ and $\silt[m(1)]{m'}$ are $2$-term silting complexes, so is $P$.  By Lemma \ref{lem:AIR} there is a mutation $m(r+1) \in \mrig \cC$  of $m(r)$ in $\cC$ such that $\silt[m(1)]{m(r)} > \silt[m(1)]{m(r+1)} \geqslant \silt[m(1)]{m'}$, extending \eqref{equ:G12a} by one step.

This process must terminate with $m(r) = m'$ for some $r$, for otherwise it would produce infinitely many pairwise non-isomorphic objects $\silt[m(1)]{m(i)} \in \2silt \underline{A}$, hence by Lemma \ref{lem:AIR} infinitely many pairwise non-isomorphic objects $m(i) \in \mrig \cC$, contradicting that $\cC$ is maximal rigid finite.  
\end{proof}

\begin{Definition}
[$\greenstart$ and $\greenend$]
\label{def:Alpha_Omega}
If $\mu = m(1) \xrightarrow{} \cdots \xrightarrow{} m(p)$ is a green path in $\Gamma_{ \cC }$, then we set
\begin{align*}
  \greenstart( \mu ) & = \{\, n \in \mrig \cC \mid \mbox{there is a green path $m(1) \xrightarrow{} n \xrightarrow{} \cdots \xrightarrow{} m(p)$ in $\Gamma_{ \cC }$} \,\}, \\[0.5mm]
  \greenend( \mu ) & = \{\, n \in \mrig \cC \mid \mbox{there is a green path $m(1) \xrightarrow{} \cdots \xrightarrow{} n \xrightarrow {} m(p)$ in $\Gamma_{ \cC }$} \,\}.
\end{align*}
Note that the green paths appearing in the displayed formulae are only required to coincide with $\mu$ at the first and last vertices.  The elements of $\greenstart( \mu )$ are mutations of $m(1)$ and the elements of $\greenend( \mu )$ are mutations of $m(p)$.
\end{Definition}

\begin{Lemma}
\label{lem:G13}
Assume that $\cC$ is maximal rigid finite.  Let $\mu = m(1) \xrightarrow{} \cdots \xrightarrow{} m(p)$ be a green path in $\Gamma_{ \cC }$, and let $n \in \mrig \cC$ be a mutation of $m(p)$ in $\cC$.  Precisely one of the following is true.
\begin{enumerate}
\setlength\itemsep{4pt}

  \item  $\overbrace{ m(1) \xrightarrow{} \cdots \xrightarrow{} m(p) }^{ \mu } \xrightarrow{} n$ is a green path in $\Gamma_{ \cC }$.

  \item  $n \in \greenend( \mu )$.

\end{enumerate}
\end{Lemma}

\begin{proof}
If $\silt[m(1)]{m(p)} > \siltlittle[m(1)]{n}$ then (i) is true and (ii) is false.

If $\siltlittle[m(1)]{n} > \silt[m(1)]{m(p)}$ then (i) is false.  On the other hand, there is a green path $m(1) \xrightarrow{} \cdots \xrightarrow{} n$ by Lemma \ref{lem:G12}, and it can be extended to a green path $m(1) \xrightarrow{} \cdots \xrightarrow{} n \xrightarrow{} m(p)$, so (ii) is true.
\end{proof}

\begin{Lemma}
\label{lem:backwards}
If condition {\rm \hyper} is satisfied, then $\mu = m(1) \xrightarrow{} \cdots \xrightarrow{} m(p)$ is a green path if and only if the {\em opposite path} $\mu^{ \opp } = m(p) \xrightarrow{} \cdots \xrightarrow{} m(1)$ is a green path.
\end{Lemma}

\begin{proof}
Let $\ell \in \cC$ be a basic maximal rigid object.  It is clear that $B^{ \ell }( \mu )$ is a minimal gallery if and only if so is $B^{ \ell }( \mu^{ \opp } )$, so the lemma follows from Lemma \ref{lem:H5}.
\end{proof}

\begin{Lemma}
\label{lem:H9}
Assume that condition {\rm \hyper} is satisfied.  Let $\mu = m(1) \xrightarrow{} \cdots \xrightarrow{} m(p)$ be a green path in $\Gamma_{ \cC }$, and let $n \in \mrig \cC$ be a mutation of $m(1)$ in $\cC$.  Precisely one of the following is true.
\begin{enumerate}
\setlength\itemsep{4pt}

  \item  $n \xrightarrow{} \overbrace{ m(1) \xrightarrow{} \cdots \xrightarrow{} m(p) }^{ \mu }$ is a green path in $\Gamma_{ \cC }$.

  \item  $n \in \greenstart( \mu )$.

\end{enumerate}
\end{Lemma}

\begin{proof}
Combine Lemmas \ref{lem:G13} and \ref{lem:backwards}. 
\end{proof}

\begin{Lemma}
\label{lem:H11}
If condition {\rm \hyper} is satisfied, then each path $\mu$ of length $\geqslant 1$ in $\Gamma_{ \cC }$ has the form $\mu = \mu \langle t \rangle \circ \cdots \circ \mu \langle 1 \rangle$ for some $t \geqslant 1$, where $\mu \langle i \rangle$ is a green path and $\greenend( \mu \langle i \rangle ) = \greenstart( \mu \langle i+1 \rangle )$ for each $i$.
\end{Lemma}

\begin{proof}
Lemma \ref{lem:H5} translates from paths to galleries, where the lemma states a known property of the so-called Deligne normal form of a gallery, see \cite[lem.\ 6.3]{Au2}.
\end{proof}

\section{Homological properties of green paths}
\label{sec:green_paths_homology}

\begin{Definition}
[Simple objects associated to arrows]
\label{def:simple}
Let
\[
  n \xrightarrow{ \beta } m
\]  
be an arrow in $\Gamma_{ \cC }$ and consider the indecomposable objects $n_{ \mut }$ and $m_{ \mut }$ in $\cC$ such that $n \cong n_{ \mut } \oplus \widetilde{ n }$ and $m \cong m_{ \mut } \oplus \widetilde{ n }$ in $\cC$.  There is an indecomposable projective $\cC( m,m )$-right module $\cC( m,m_{ \mut } )$.  Its simple top
\[
  \simple[\beta] = \top_{ \cC( m,m ) }\cC( m,m_{ \mut } ) \in \mod \cC( m,m )
\]
is defined up to isomorphism.  It can also be viewed as a simple $\cE( m,m )$-right module via the canonical surjection $\cE( m,m ) \xrightarrow{} \cC( m,m )$.
\end{Definition}

\begin{Lemma}
\label{lem:simple}
Up to isomorphism, each simple $\cC( m,m )$-right module has the form $\simple[ \beta ]$ for an arrow $n \xrightarrow{ \beta } m$ in $\Gamma_{ \cC }$.
\end{Lemma}

\begin{proof}
Each simple is the top of an indecomposable projective $\cC( m,m )$-right module.  Up to isomorphism, each such has the form $\cC( m,m_{ \mut } )$ for an indecomposable summand $m_{ \mut }$ of $m$ in $\cC$, corresponding to an arrow $n \xrightarrow{ \beta } m$ in $\Gamma_{ \cC }$.
\end{proof}

\begin{Lemma}
\label{lem:SP}
Let $n \xrightarrow{ \beta } m$ be an arrow in $\Gamma_{ \cC }$ and consider the indecomposable objects $n_{ \mut }$ and $m_{ \mut }$ in $\cC$ such that $n \cong n_{ \mut } \oplus \widetilde{ n }$ and $m \cong m_{ \mut } \oplus \widetilde{ n }$ in $\cC$.  If $m'$ is an indecomposable summand of $m$ in $\cC$, then there is an isomorphism
\[
  \simple[ \beta ] \Tensor{ \cC( m,m ) } \cC( m',m )
  \cong \left\{
          \begin{array}{cl}
            k & \mbox{if $m' \cong m_{ \mut }$,} \\[1mm]
            0 & \mbox{otherwise}
          \end{array}
        \right.
\]
in $\mod k$.
\end{Lemma}

\begin{proof}
This holds since $\simple[ \beta ]$ is the simple top of the indecomposable projective $\cC( m,m )$-right module $\cC( m,m_{ \mut } )$.
\end{proof}

\begin{Definition}
[Condition \somefiniteness]
\label{def:somefiniteness}
We say that condition \somefiniteness\ is satisfied if, for $m \in \mrig \cC$, each finitely generated $\cC( m,m )$-module is finitely presented as an $\cE( m,m )$-module.
\end{Definition}

\begin{Lemma}
\label{lem:F22}
Assume that condition {\rm \somefiniteness} is satisfied.  Let 
$
  \xymatrix {
  n
    \ar^{ \beta }[r]<0.5ex> &
    m
    \ar^{ \alpha }[l]<0.5ex>
            }
$
be arrows in $\Gamma_{ \cC }$ and consider
\[
  F( \beta ) = \cE( n,m ),
\]  
which is a two-sided $A$-$B$-tilting complex over $A = \cE( m,m )$ and $B = \cE( n,n )$, see Proposition \ref{pro:G1} and Theorem \ref{thm:G11}.  It satisfies: 
\begin{enumerate}
\setlength\itemsep{4pt}

  \item  $\simple[ \beta ] \LTensor{ A } F( \beta ) \cong \Sigma \simple[ \alpha ]$ in $\cD( B )$. 
  
  \item  $\simple[ \beta' ] \LTensor{ A } F( \beta )$ is concentrated in degree $0$ when $n' \xrightarrow{ \beta' } m$ is an arrow in $\Gamma_{ \cC }$ different from $\beta$. 

\end{enumerate}
\end{Lemma}

\begin{proof}
Consider the indecomposable objects $n_{ \mut }$ and $m_{ \mut }$ in $\cC$ such that $n \cong n_{ \mut } \oplus \widetilde{ n }$ and $m \cong m_{ \mut } \oplus \widetilde{ n }$ in $\cC$.  Set $\underline{ A } = \cC( m,m )$, $\underline{ B } = \cC( n,n )$, and let $f \in \underline{ B }$ be projection onto $\widetilde{ n }$.

There is an exchange triangle $n_{ \mut } \xrightarrow{} b \xrightarrow{} m_{ \mut } \xrightarrow{} \Sigma n_{ \mut }$ in $\cC$ with $b \in \add_{ \cC }( \widetilde{ n } )$ by \cite[cor.\ 3.3]{ZZ}.  Hence there is a triangle $n_{ \mut } \oplus \widetilde{ n } \xrightarrow{} b \oplus \widetilde{ n } \xrightarrow{} m_{ \mut } \xrightarrow{} \Sigma( n_{ \mut } \oplus \widetilde{ n } )$, that is,
\begin{equation}
\label{equ:F22d}
  n \xrightarrow{} b \oplus \widetilde{ n } \xrightarrow{} m_{ \mut } \xrightarrow{} \Sigma n
\end{equation}
in $\cC$.  Observe that $b \oplus \widetilde{ n }$ and $m_{ \mut }$ are in $\add_{ \cC }( m )$.

Lemma \ref{lem:F21} applies with $x = m$, $y = n$, and the triangle \eqref{equ:F21a} given by \eqref{equ:F22d}.  Then ${}_{ \underline{ A } }P$ of Equation \eqref{equ:F21e} has the form ${}_{ \underline{ A } }P = \cC( m_{ \mut },m ) \xrightarrow{} \cC( b \oplus \widetilde{ n },m )$, and for $M_{ \underline{ A } } \in \cD( \underline{ A } )$ there is an isomorphism in $\cD( k )$:
\begin{equation}
\label{equ:F22f}
  M \LTensor{ A } F( \beta ) \cong M \Tensor{ \underline{ A } } \big( \cC( m_{ \mut },m ) \xrightarrow{} \cC( b \oplus \widetilde{ n },m ) \big).
\end{equation}

Lemma \ref{lem:F17} applies with $x = m$, $y = n$, $z' = \widetilde{ n }$, so if $M_{ \underline{ A } } \in \mod \underline{ A }$ then:
\begin{align}
\label{equ:F22b}
  & \mbox{Each $\H_i\!\big( M \LTensor{ A } F( \beta ) \big)$ is a $\underline{ B }$-module.} \\
\label{equ:F22c}
  & \mbox{$\H_1\!\big( M \LTensor{ A } F( \beta ) \big)$ is annihilated by $f\underline{ B }$.}
\end{align}

To prove (i), note that Equation \eqref{equ:F22f} and Lemma \ref{lem:SP} give isomorphisms in $\cD( k )$:
\[
  \simple[ \beta ] \LTensor{ A } F( \beta ) \cong \simple[ \beta ] \Tensor{ \underline{ A } } \big( \cC( m_{ \mut },m ) \xrightarrow{} \cC( b \oplus \widetilde{ n },m ) \big) \cong ( k \xrightarrow{} 0 ) = \Sigma k.
\]
It follows that there is a $B$-right module $U_B$ such that $\simple[ \beta ] \LTensor{ A } F( \beta ) \cong \Sigma U_B$ in $\cD( B )$ and $\dim_k U_B = 1$.  Equations \eqref{equ:F22b} and \eqref{equ:F22c} imply that $U_B$ is a $\underline{ B }$-module annihilated by $f\underline{ B }$ whence $U_B \cong \simple[ \alpha ]$, proving (i).

To prove (ii), note that Equation \eqref{equ:F22f} and Lemma \ref{lem:SP} give isomorphisms in $\cD( k )$:
\[
  \simple[ \beta' ] \LTensor{ A } F( \beta ) \cong \simple[ \beta' ] \Tensor{ \underline{ A } } \big( \cC( m_{ \mut },m ) \xrightarrow{} \cC( b \oplus \widetilde{ n },m ) \big) \cong ( 0 \xrightarrow{} V ) = V
\]
for some $V \in \mod k$, proving (ii).
\end{proof}

\begin{Observation}
\label{obs:F22}
Assume that condition {\rm \somefiniteness} is satisfied.  Let 
$
  \xymatrix {
  n
    \ar^{ \beta }[r]<0.5ex> &
    m
    \ar^{ \alpha }[l]<0.5ex>
            }
$
be arrows in $\Gamma_{ \cC }$ and consider the path $\mu = \beta\alpha$ from $m$ to $m$ in $\Gamma_{ \cC }$.

Then $F( \mu )$ is an element of infinite order in the group $\DPic\big( \cE( m,m ) \big)$.  This is a consequence of the following computation using the notation of Lemma \ref{lem:F22}.
\[
  \simple[ \beta ] \LTensor{ A } F( \mu )
  \cong \simple[ \beta ] \LTensor{ A } F( \beta ) \LTensor{ B } F( \alpha ) \cong \Sigma \simple[ \alpha ] \LTensor{ B } F( \alpha ) \cong \Sigma^2 \simple[ \beta ]
\]
The two last isomorphisms are by Lemma \ref{lem:F22}  verbatim and with $\alpha$ and $\beta$ interchanged.
\end{Observation}

\begin{Definition}
[$\inf$ and $\sup$ of a complex]
\label{def:sup}
If $X$ is an object of a derived category, then
\begin{align*}
  \inf X & = \inf \{\, i \mid \H_i(X) \neq 0 \,\}, \\
  \sup X & = \sup \{\, i \mid \H_i(X) \neq 0 \,\}.
\end{align*}
Note that the zero object has $\inf 0 = \infty$ and $\sup 0 = -\infty$, while any non-zero object $X$ has $\inf X \leqslant \sup X$.  
\end{Definition}

\begin{Lemma}
\label{lem:inequality}
If $X \xrightarrow{} Y \xrightarrow{} Z$ a triangle in a derived category, then
\[
  \sup Y \leqslant \max\{ \sup X,\sup Z \}.
\]
\end{Lemma}

\begin{proof}
Immediate from the long exact homology sequence induced by the triangle.
\end{proof}

\begin{Setup}
\label{set:green_path}
Lemmas \ref{lem:inf_sup} through \ref{lem:preH13} concern a green path in $\Gamma_{ \cC }$:
\[
  \mu = m(1) \xrightarrow{} \cdots \xrightarrow{} m(p).
\]
Recall from Theorem \ref{thm:G11} that
\[
  F( \mu ) = \cE\big( m(1),m(p) \big)
\]  
is a two-sided $A_p$-$A_1$-tilting complex over $A_p = \cE\big( m(p),m(p) \big)$ and $A_1 = \cE\big( m(1),m(1) \big)$.
\end{Setup}

\begin{Lemma}
\label{lem:inf_sup}
\begin{enumerate}
\setlength\itemsep{4pt}

  \item  $\inf F( \mu ) = \sup F( \mu ) = 0$.

  \item  If $M \in \cD( A_p )$, then
\[
  \inf M \leqslant \inf M \LTensor{ A_p } F( \mu ) \;\;,\;\; \sup M \LTensor{ A_p } F( \mu ) \leqslant \sup M + 1.
\]

  \item  If $N \in \cD( A_1 )$, then
\[
  \inf N - 1 \leqslant \inf \RHom_{ A_1 }\!\big( F( \mu ),N \big) \;\;,\;\;
  \sup \RHom_{ A_1 }\!\big( F( \mu ),N \big) \leqslant \sup N.
\]

\end{enumerate}
\end{Lemma}

\begin{proof}
Part (i) holds because $F( \mu )$ is isomorphic to a module, parts (ii) and (iii) because $F( \mu )$ has projective resolutions over $A_1$ and $A_p$ which are concentrated in homological degrees $0$ and $1$, see Lemma \ref{lem:F9_10}(iii) and its dual.
\end{proof}

\begin{Lemma}
\label{lem:G14}
Assume that condition {\rm \somefiniteness} is satisfied.  If $n \xrightarrow{ \beta } m(p)$ is an arrow in $\Gamma_{ \cC }$, then
\begin{enumerate}
\setlength\itemsep{4pt}

  \item  $n \in \greenend( \mu ) \Leftrightarrow \simple[ \beta ] \LTensor{ A_p } F( \mu )$ is concentrated in homological degree $1$.
  
  \item  $n \not\in \greenend( \mu ) \Leftrightarrow \simple[ \beta ] \LTensor{ A_p } F( \mu )$ is concentrated in homological degree $0$.

\end{enumerate}
\end{Lemma}

\begin{proof}
Set $B = \cE( n,n )$.  We show $\Rightarrow$ in each of (i) and (ii), which is enough to prove the lemma.

(i):  The condition $n \in \greenend( \mu )$ means there is a green path $\mu' = m(1) \xrightarrow{} \cdots \xrightarrow{} n \xrightarrow{ \beta } m(p)$ in $\Gamma_{ \cC }$.  The subpath $\mu'' = m(1) \xrightarrow{} \cdots \xrightarrow{} n$ is also a green path.  Since $\mu \sim \mu' = \beta \circ \mu''$ in $\cF_{ \cC }$, we have $F( \mu ) = F( \mu' ) = F( \beta ) \LTensor{ B } F( \mu'' )$ by Theorem \ref{thm:G11}.  This gives the first equality in the following computation:
\begin{align*}
  \inf \simple[ \beta ] \LTensor{ A_p } F( \mu )
  & = \inf \simple[ \beta ] \LTensor{ A_p } F ( \beta ) \LTensor{ B } F( \mu'' ) \\
  & \geqslant \inf \simple[ \beta ] \LTensor{ A_p } F( \beta ) \\
  & = 1.
\end{align*}
The inequality is by Lemma \ref{lem:inf_sup}(ii), and the second equality by Lemma \ref{lem:F22}(i).  On the other hand,
\[
  \sup \simple[ \beta ] \LTensor{ A_p } F( \mu ) \leqslant 1
\]
by Lemma \ref{lem:inf_sup}(ii).  These estimates imply (i).

(ii):  The condition $n \not\in \greenend( \mu )$ implies that $\mu$ can be extended to the green path $\mu''' = m(1) \xrightarrow{} \cdots \xrightarrow{} m(p) \xrightarrow{ \alpha } n$ by Lemma \ref{lem:G13}.  Since $\mu''' = \alpha \circ \mu$ in $\cF_{ \cC }$, we have $F( \mu''' ) = F( \alpha ) \LTensor{ A_p } F( \mu )$.  This gives the first of the following isomorphisms:
\[
  \simple[ \alpha ] \LTensor{ B } F( \mu''' ) \cong \simple[ \alpha ] \LTensor{ B } F( \alpha ) \LTensor{ A_p } F( \mu ) \cong \Sigma \simple[ \beta ] \LTensor{ A_p } F( \mu ),
\]
and the second isomorphism holds by Lemma \ref{lem:F22}(i).  But the left hand side is concentrated in homological degree $1$ by part (i) applied to $\mu'''$, so $\simple[ \beta ] \LTensor{ A_p }F( \mu )$ is concentrated in homological degree $0$ as claimed.  
\end{proof}

Lemma \ref{lem:G14} has the following dual.  We omit the proof, which is also dual.  However, note that the proof would need to invoke Lemma \ref{lem:H9} instead of Lemma \ref{lem:G13}, so we have to assume condition \hyper.

\begin{Lemma}
\label{lem:H10}
Assume that conditions {\rm \hyper} and {\rm \somefiniteness} are satisfied.  If $n \xrightarrow{ \beta } m(1)$ is an arrow in $\Gamma_{ \cC }$, then
\begin{enumerate}
\setlength\itemsep{4pt}

  \item  $n \in \greenstart( \mu ) \Leftrightarrow \RHom_{ A_1 }\!\big( F( \mu ), \simple[ \beta ] \big)$ is concentrated in homological degree $-1$.
  
  \item  $n \not\in \greenstart( \mu ) \Leftrightarrow \RHom_{ A_1 }\!\big( F( \mu ), \simple[ \beta ] \big)$ is concentrated in homological degree $0$.

\end{enumerate}
\end{Lemma}

\begin{Lemma}
\label{lem:preH13}
Assume that conditions {\rm \hyper} and {\rm \somefiniteness} are satisfied.  If $n \xrightarrow{ \beta } m(p)$ is an arrow in $\Gamma_{ \cC }$ with $n \in \greenend( \mu )$, then there is an arrow $n' \xrightarrow{ \beta' } m(1)$  with $n' \in \greenstart( \mu )$ and a triangle in $\cD( A_1 )$:
\[
  \Sigma \simple[ \beta' ] \xrightarrow{} \simple[ \beta ] \LTensor{ A_p } F( \mu ) \xrightarrow{} Z'
\]
with $Z'$ concentrated in homological degree $1$.
\end{Lemma}

\begin{proof}
Lemma \ref{lem:G14}(i) says that
\begin{equation}
\label{equ:preH13a}
  \Sigma U \cong \simple[ \beta ] \LTensor{ A_p } F( \mu ) 
\end{equation}
in $\cD( A_1 )$ for some non-zero $A_1$-right module $U$.  Lemma \ref{lem:F17} applies with $x = m(p)$, $y = m(1)$ and implies that $U$ is an $\underline{ A }_1$-right module where $\underline{ A }_1 = \cC\big( m(1),m(1) \big)$.

Since $\underline{ A }_1$ is a finite dimensional $k$-algebra, there is a short exact sequence
\begin{equation}
\label{equ:preH13b}
  0 \xrightarrow{} T \xrightarrow{} U \xrightarrow{} V \xrightarrow{} 0
\end{equation}
in $\Mod \underline{ A }_1$ with $T$ simple.  By Lemma \ref{lem:simple} we have
\begin{equation}
\label{equ:preH13c}
  T \cong \simple[ \beta' ]
\end{equation}
for some arrow $n' \xrightarrow{ \beta' } m(1)$ in $\Gamma_{ \cC }$.  We can view \eqref{equ:preH13b} as a short exact sequence in $\Mod A_1$.  It induces a triangle in $\cD( A_1 )$ whose first suspension reads
\begin{equation}
\label{equ:preH13d}
  \Sigma \simple[ \beta' ] \xrightarrow{} \simple[ \beta ] \LTensor{ A_p } F( \mu ) \xrightarrow{} \Sigma V
\end{equation}
by Equations \eqref{equ:preH13a} and \eqref{equ:preH13c}.  Setting $Z' = \Sigma V$ gives the triangle in the lemma.

To complete the proof, we must show $n' \in \greenstart( \mu )$.  Applying to \eqref{equ:preH13d} the functor $\RHom_{ A_1 }\!\big( F( \mu ),- \big)$, which is quasi-inverse to $- \LTensor{ A_p } F( \mu )$, gives a triangle
\[
  \RHom_{ A_1 }\!\big( F( \mu ),\Sigma \simple[ \beta' ] ) \xrightarrow{} \simple[ \beta ] \xrightarrow{} \RHom_{ A_1 }\!\big( F( \mu ),\Sigma V \big).
\]
Rolling one step left and desuspending gives the triangle
\[
  \Sigma^{ -1 }\RHom_{ A_1 }\!\big( F( \mu ),V \big) \xrightarrow{} \RHom_{ A_1 }\!\big( F( \mu ), \simple[ \beta' ] \big) \xrightarrow{} \Sigma^{ -1 } \simple[ \beta ].
\]
The first term satisfies $\sup \Sigma^{ -1 }\RHom_{ A_1 }\!\big( F( \mu ),V \big) \leqslant -1$ by Lemma \ref{lem:inf_sup}(iii), and the third term satisfies $\sup \Sigma^{ -1 } \simple[ \beta ] = -1$, so the middle term satisfies 
\[
  \sup \RHom_{ A_1 }\!\big( F( \mu ), \simple[ \beta' ] \big) \leqslant -1
\]
by Lemma \ref{lem:inequality}.  However,
\[
  -1 \leqslant \inf \RHom_{ A_1 }\!\big( F( \mu ), \simple[ \beta' ] \big)
\]
holds by Lemma \ref{lem:inf_sup}(iii), so $\RHom_{ A_1 }\!\big( F( \mu ), \simple[ \beta' ] \big)$ is concentrated in homological degree $-1$ whence $n' \in \greenstart( \mu )$ by Lemma \ref{lem:H10}(i).  
\end{proof}

\section{Faithfulness}
\label{sec:fidelity}

\begin{Setup}
\label{set:mu}
Lemma \ref{lem:H13} concerns a path $\mu$ of length $\geqslant 1$ in $\Gamma_{ \cC }$ and the isomorphism class of two-sided tilting complexes $F( \mu )$.  By Lemma \ref{lem:H11} we have
\begin{equation}
\label{equ:mub}
  \mu = \mu \langle t \rangle \circ \cdots \circ \mu \langle 1 \rangle
\end{equation}
for some $t \geqslant 1$, where
\[
  \mu \langle i \rangle = m \langle i \rangle \xrightarrow{} \cdots \xrightarrow{} m \langle i+1 \rangle
\]
is a green path in $\Gamma_{ \cC }$, and
\begin{equation}
\label{equ:mua}
  \greenend( \mu \langle i \rangle ) = \greenstart( \mu \langle i+1 \rangle )
\end{equation}
for each $i$.  Set $A_i = \cE( m \langle i \rangle,m \langle i \rangle )$.
\end{Setup}

\begin{Lemma}
\label{lem:H13}
Assume that conditions {\rm \hyper} and {\rm \somefiniteness} are satisfied.  Let $n \xrightarrow{ \beta } m \langle t+1 \rangle$ be an arrow in $\Gamma_{ \cC }$.
\begin{enumerate}
\setlength\itemsep{4pt}

  \item  If $n \in \greenend( \mu \langle t \rangle )$, then $\sup \simple[ \beta ] \LTensor{ A_{ t+1 } } F( \mu ) = t$.  

  \item  If $n \not\in \greenend( \mu \langle t \rangle )$, then $\sup \simple[ \beta ] \LTensor{ A_{ t+1 } } F( \mu ) \leqslant t-1$.  

\end{enumerate}
\end{Lemma}

\begin{proof}
(i):  We will prove the following stronger statement by induction on $t$:
\begingroup
\renewcommand{\labelenumi}{(\alph{enumi})}
\begin{enumerate}
\setlength\itemsep{4pt}

  \item  If $n \xrightarrow{ \beta } m \langle t+1 \rangle$ is an arrow in $\Gamma_{ \cC }$ with $n \in \greenend( \mu \langle t \rangle )$, then there is an arrow $n'' \xrightarrow{ \beta'' } m \langle 1 \rangle$  with $n'' \in \greenstart( \mu \langle 1 \rangle )$ and a triangle in $\cD( A_1 )$:
\begin{equation}
\label{equ:H13a}
  \Sigma^t \simple[ \beta'' ] \xrightarrow{} \simple[ \beta ] \LTensor{ A_{ t+1 } } F( \mu ) \xrightarrow{} Z''
\end{equation}
with $\sup Z'' \leqslant t$.

\end{enumerate}
\endgroup
This implies part (i) because, on the one hand, Lemma \ref{lem:inequality} applied to \eqref{equ:H13a} gives $\sup \simple[ \beta ] \LTensor{ A_{ t+1 } } F( \mu ) \leqslant t$, and, on the other hand, the long exact homology sequence of \eqref{equ:H13a} contains
\[
  \H_{ t+1 }Z'' \xrightarrow{} \H_t\!\big( \Sigma^t \simple[ \beta'' ] \big) \xrightarrow{} \H_t\!\big( \simple[ \beta ] \LTensor{ A_{ t+1 } } F( \mu ) \big),
\]
that is
\[
  0 \xrightarrow{} \simple[ \beta'' ] \xrightarrow{} \H_t\!\big( \simple[ \beta ] \LTensor{ A_{ t+1 } } F( \mu ) \big),
\]
proving $\H_t\!\big( \simple[ \beta ] \LTensor{ A_{ t+1 } } F( \mu ) \big) \neq 0$ whence $\sup \simple[ \beta ] \LTensor{ A_{ t+1 } } F( \mu ) \geqslant t$. 

To prove (a), note that for $t=1$ it holds by Lemma \ref{lem:preH13}.

Suppose that $t \geqslant 2$ and that (a) holds at level $t-1$, so applies to 
\begin{equation}
\label{equ:H13e}
  \mu' = \mu \langle t-1 \rangle \circ \cdots \circ \mu \langle 1 \rangle.
\end{equation}
Let $n \xrightarrow{ \beta } m \langle t+1 \rangle$ with $n \in \greenend( \mu \langle t \rangle )$ be an arrow in $\Gamma_{ \cC }$.  By Lemma \ref{lem:preH13} applied to the green path $\mu \langle t \rangle = m \langle t \rangle \xrightarrow{} \cdots \xrightarrow{} m \langle t+1 \rangle$, there is an arrow $n' \xrightarrow{ \beta' } m \langle t \rangle$ with $n' \in \greenstart( \mu \langle t \rangle )$ and a triangle in $\cD( A_t )$:
\[
  \Sigma \simple[ \beta' ] \xrightarrow{} \simple[ \beta ] \LTensor{ A_{ t+1 } } F( \mu \langle t \rangle ) \xrightarrow{} Z'
\]
with $Z'$ concentrated in homological degree $1$, so in particular
\begin{equation}
\label{equ:H13c}
  \sup Z' \leqslant 1.
\end{equation}
We have $\mu \langle t \rangle \circ \mu' = \mu$ whence $F( \mu \langle t \rangle ) \LTensor{ A_t } F( \mu' ) = F( \mu )$, so applying the functor $- \LTensor{ A_t } F( \mu' )$ to the previous triangle gives a triangle
\begin{equation}
\label{equ:H13b}
  \Sigma \simple[ \beta' ] \LTensor{ A_t } F( \mu' ) \xrightarrow{} \simple[ \beta ] \LTensor{ A_{ t+1 } } F( \mu ) \xrightarrow{} Z' \LTensor{ A_t } F( \mu' ).
\end{equation}
Equation \eqref{equ:mua} gives $n' \in \greenend( \mu \langle t-1 \rangle )$ so by (a) applied to $\mu'$ there is an arrow in $\Gamma_{ 
\cC }$,
\[
  \mbox{$n'' \xrightarrow{ \beta'' } m \langle 1 \rangle$ with $n'' \in \greenstart( \mu \langle 1 \rangle )$,}
\]
and a triangle in $\cD( A_1 )$:
\[
  \Sigma^{ t-1 } \simple[ \beta'' ] \xrightarrow{} \simple[ \beta' ] \LTensor{ A_t } F( \mu' ) \xrightarrow{} Z''
\]
with
\begin{equation}
\label{equ:H13d}
  \sup Z'' \leqslant t-1.
\end{equation}
Its first suspension is the top row of the following diagram, and \eqref{equ:H13b} is the middle column.  The diagram exists by the octahedral axiom, and each row and column is a triangle.
\[
\vcenter{
  \xymatrix @+0.2pc {
    \Sigma^t \simple[ \beta'' ] \ar[r] \ar@{=}[d] & \Sigma \simple[ \beta' ] \LTensor{ A_t } F( \mu' ) \ar[r] \ar[d] & \Sigma Z'' \ar[d] \\
    \Sigma^t \simple[ \beta'' ] \ar[r] \ar[d] & \simple[ \beta ] \LTensor{ A_{ t+1 } } F( \mu ) \ar[r] \ar[d] & Z''' \ar[d] \\
    0 \ar[r] & Z' \LTensor{ A_t } F( \mu' ) \ar@{=}[r] & Z' \LTensor{ A_t } F( \mu' ) \\
                    }
        }
\]
The middle row is a triangle in $\cD( A_1 )$, which establishes (a) at level $t$.  To finish the proof of this, we only need to show $\sup Z''' \leqslant t$, which will follow from Lemma \ref{lem:inequality} applied to the last column if we can see $\sup \Sigma Z'' \leqslant t$ and $\sup Z' \LTensor{ A_t } F( \mu' ) \leqslant t$.  The former is immediate by \eqref{equ:H13d}.  The latter can be seen as follows:
\[
  \sup Z' \LTensor{ A_t } F( \mu' )
  \stackrel{\rm (a)}{=} \sup Z' \LTensor{ A_t } F( \mu \langle t-1 \rangle ) \LTensor{ A_{ t-1 } }\cdots \LTensor{ A_2 } F( \mu \langle 1 \rangle )
  \stackrel{\rm (b)}{\leqslant} \sup Z' +t-1
  \stackrel{\rm (c)}{\leqslant} t,
\]
where (a) is by Equation \eqref{equ:H13e} and (b) is by repeated application of Lemma \ref{lem:inf_sup}(ii), while (c) is by \eqref{equ:H13c}.

(ii):  If $n \not\in \greenend( \mu \langle t \rangle )$, then 
\begin{align*}
  \sup \simple[ \beta ] \LTensor{ A_{ t+1 } } F( \mu )
  & \stackrel{ \rm (a) }{=} \sup \big( \simple[ \beta ] \LTensor{ A_{ t+1 } } F( \mu \langle t \rangle ) \big) \LTensor{ A_t } F( \mu \langle t-1 \rangle ) \LTensor{ A_{ t-1 } } \cdots \LTensor{ A_2 } F( \mu \langle 1 \rangle ) \\
  & \stackrel{ \rm (b) }{\leqslant} \sup \big( \simple[ \beta ] \LTensor{ A_{ t+1 } } F( \mu \langle t \rangle ) \big) + t-1 \\
  & \stackrel{ \rm (c) }{=} t-1,
\end{align*}
where (a) is by Equation \eqref{equ:mub} and (b) is by $t-1$ applications of Lemma \ref{lem:inf_sup}(ii), while (c) is by Lemma \ref{lem:G14}(ii).
\end{proof}

\begin{Lemma}
\label{lem:postH13}
Assume that conditions {\rm \hyper} and {\rm \somefiniteness} are satisfied.  If $\mu$ is a path in $\Gamma_{ \cC }$, then
\[
  \mbox{$\mu$ has length $0$ $\Leftrightarrow$ $F( \mu )$ is an identity morphism in $\cP_{ \cE }$.}
\]
\end{Lemma}

\begin{proof}
If $\mu$ has length $0$, then it is an identity morphism in $\cF_{ \cC }$, so $F( \mu )$ is an identity morphism in $\cP_{ \cE }$.  If $\mu$ has length $\geqslant 1$, then Setup \ref{set:mu} applies to it, and Lemma \ref{lem:H13}(i) shows that $F( \mu )$ is not an identity morphism in $\cP_{ \cE }$.
\end{proof}

\begin{Theorem}
\label{thm:H18}
If conditions {\rm \hyper} and {\rm \somefiniteness} are satisfied, then the functor $\cG_{ \cC } \xrightarrow{ G } \cP_{ \cE }$ of Theorem \ref{thm:G11}(iii) is faithful.
\end{Theorem}

In the case where $\cE$ is the category of maximal Cohen--Macaulay modules over the base of a 3-fold flopping contraction,
this theorem first appeared in \cite[thm.\ 6.5]{HW}.

\begin{proof}[Proof of Theorem \ref{thm:H18}]
Recall that the canonical functor $\cF_{ \cC } \xrightarrow{ Q_{ \cC } } \cG_{ \cC }$ of Definition \ref{def:green_groupoid} and the functors $F$ and $G$ of Theorem \ref{thm:G11} are linked by $GQ_{ \cC } = F$.  By the proof of \cite[lem.\ 2.11]{HW} it is enough to let $\mu$ and $\nu$ be paths in $\Gamma_{ \cC }$, assume $GQ_{ \cC }( \mu ) = GQ_{ \cC }( \nu )$, and prove $Q_{ \cC }( \mu ) = Q_{ \cC }( \nu )$; in other words, to assume $F( \mu ) = F( \nu )$ and prove $\mu \sim \nu$.  We do so by induction on the length of $\nu$.

Suppose $\nu$ has length $0$.  Then $\nu = \id_m$ for some $m \in \mrig \cC$ whence $F( \mu ) = F( \nu ) = \id_{ Fm }$.  By Lemma \ref{lem:postH13} this implies that $\mu$ has length $0$ whence $\mu = \id_m = \nu$.  In particular, $\mu \sim \nu$.

Suppose $\nu$ has length $\geqslant 1$.  Since $F( \mu ) = F( \nu )$, Lemma \ref{lem:postH13} implies that $\mu$ has length $\geqslant 1$.  Hence Setup \ref{set:mu} applies to $\mu$ and $\nu$.  We will use the notation from the setup for $\mu$, and analogous notation for $\nu$, writing $\nu = \nu \langle u \rangle \circ \cdots \circ \nu \langle 1 \rangle$ where $\nu \langle i \rangle = n \langle i \rangle \xrightarrow{} \cdots \xrightarrow{} n \langle i+1 \rangle$ is a green path in $\Gamma_{ \cC }$ and $\greenend( \nu \langle i \rangle ) = \greenstart( \nu \langle i+1 \rangle )$.  Note that there is an $m \in \mrig \cC$ such that $m \langle t+1 \rangle = n \langle u+1 \rangle = m$, so
\[
  \mu \langle t \rangle = m \langle t \rangle \xrightarrow{} \cdots \xrightarrow{} m
  \;\;,\;\;
  \nu \langle u \rangle = n \langle u \rangle \xrightarrow{} \cdots \xrightarrow{} m.  
\]
Since $F( \mu ) = F( \nu )$, Lemma \ref{lem:H13} implies that $\greenend( \mu \langle t \rangle ) = \greenend( \nu \langle u \rangle )$, so there is an $n \in \mrig \cC$ and green paths
\[
  \mu \langle t \rangle' = m \langle t \rangle  \xrightarrow{} \cdots \xrightarrow{} n \xrightarrow{ \beta } m
  \;\;,\;\;
  \nu \langle u \rangle' = n \langle u \rangle \xrightarrow{} \cdots \xrightarrow{} n \xrightarrow{ \beta } m.
\]
Definition \ref{def:green_groupoid} implies $\mu \langle t \rangle' \sim \mu \langle t \rangle$ and $\nu \langle u \rangle' \sim \nu \langle u \rangle$.  Setting
\[
  \mu' = \mu \langle t \rangle' \circ \mu \langle t-1 \rangle \circ \cdots \mu \langle 1 \rangle
  \;\;,\;\;
  \nu' = \nu \langle u \rangle' \circ \nu \langle u-1 \rangle \circ \cdots \nu \langle 1 \rangle,
\]
we have
\begin{equation}
\label{equ:H13f}
  \mu' \sim \mu \;\;,\;\; \nu' \sim \nu
\end{equation}
and
\begin{equation}
\label{equ:H18c}
  \mu' = \beta \circ \mu''
  \;\;,\;\;
  \nu' = \beta \circ \nu''
\end{equation}
for certain paths $\mu'',\nu''$ from $m \langle 1 \rangle = n \langle 1 \rangle$ to $n$.  The last two displayed equations imply that it is enough to show $\mu'' \sim \nu''$, and this will follow by induction if we can show that $F( \mu'' ) = F( \nu'' )$ and that $\nu''$ has smaller length than $\nu$.

However, by Theorem \ref{thm:G11}(iii),
Equation \eqref{equ:H13f} implies $F( \mu' ) = F( \mu )$, $F( \nu' ) = F( \nu )$, and since $F( \mu ) = F( \nu )$ we have
\[
  F( \mu' ) = F( \nu' ).
\]
Equation \eqref{equ:H18c} implies 
\[
  F( \mu' ) = F( \beta ) \circ F( \mu'' )
  \;\;,\;\;
  F( \nu' ) = F( \beta ) \circ F( \nu'' ),
\]
and since $F( \beta )$ is invertible, the last two displayed equations imply $F( \mu'' ) = F( \nu'' )$.

Lemma \ref{lem:H5} implies that $\nu \langle u \rangle'$ has the same length as $\nu \langle u \rangle$, so $\nu'$ has the same length as $\nu$.  But $\nu''$ has smaller length than $\nu'$ by Equation \eqref{equ:H18c}, hence smaller length than $\nu$.  
\end{proof}

\section{$g$-vector fans and hyperplane arrangements}
\label{sec:hyperplane}
This section does not use the Frobenius model $\cE$, but works under Setup \ref{set:blanket}(iv) alone (in which case $\mrig \cC$ should just be a set containing one object from each isomorphism class of basic maximal rigid objects of $\cC$).

For a basic maximal rigid object $\ell$ of $\cC$ with indecomposable decomposition $\ell \cong \ell_1 \oplus \cdots \oplus \ell_d$, define its {\em{Cartan matrix}} 
(or, equivalently, the Cartan matrix of the basic algebra $\cC(\ell,\ell)$) to be the $d \times d$-matrix $\Cart_\ell$
with entries 
\begin{equation}
\label{equ:Cartan}
\dim \cC(\ell_i, \ell_j) \; \;  \mbox{for} \; \;  1 \leqslant i,j \leqslant d.
\end{equation}
We identify $\Cart_\ell$ with the linear operator on 
$\Ksp( \add_{ \cC } \ell)_{ \BR }$ having that matrix in the basis $\{ [\ell_1], \ldots, [\ell_d] \}$.

\begin{Theorem}
\label{thm:gfan-to-hyperplane} Assume that the 2-Calabi--Yau category $\cC$ is maximal rigid finite (i.e. has finitely many classes of basic maximal rigid objects) and that 
$\Sigma^2 n \cong n$ for all rigid objects $n$ of $\cC$. Then the following hold:
\begin{enumerate}
\setlength\itemsep{4pt}
\item For every maximal rigid object $\ell$ of $\cC$, the matrix $C_\ell + C_\ell^T$ is positive definite, and as a consequence, 
the Cartan matrix $C_\ell$ is nondegenerate.
\item For all basic maximal rigid objects $\ell$ of $\cC$ and rigid objects $n$ of $\cC$,
\[
\index_\ell( \Sigma n) = - \index_\ell(n).
\]
\item The category $\cC$ satisfies condition  {\rm \hyper}, i.e. the $g$-vector fan decomposition $G^{ \ell }_{ \cC }$ 
comes from a simplicial hyperplane arrangement in $\Ksp( \add_{ \cC }\ell )_\BR$ for all $\ell \in \mrig \cC$.
\end{enumerate}
\end{Theorem}
Here and below, $A^T$ denotes the transpose of a matrix $A$. For a basic maximal rigid object $\ell$ of $\cC$, denote the bilinear form
\begin{equation}
\label{equ:ell-form}
\langle -, - \rangle_\ell : \Ksp( \add_{ \cC } \ell )_{\BR} \times \Ksp( \add_{ \cC } \ell )_{\BR}  \to \BR
\end{equation}
given by $\langle [\ell_i], [\ell_j] \rangle_\ell = \dim C (\ell_i, \ell_j)$
for $1 \leqslant j,k \leqslant d$. In other words, $\langle \xi, \eta \rangle_\ell = \xi^T C_\ell \eta$ for $\xi, \eta \in \Ksp( \add_{ \cC } \ell )_{\BR}$.

Along the way to proving Theorem \ref{thm:gfan-to-hyperplane}, we will establish the following facts which are 
of independent interest:

\begin{Proposition}
\label{prop:gfan-to-hyperplane2} Assume that the 2-Calabi--Yau category $\cC$ is such that
$\Sigma^2 n \cong n$ for all rigid objects $n$ of $\cC$. Then the following hold:
\begin{enumerate}
\setlength\itemsep{4pt}
\item For every pair of basic maximal rigid objects $\ell, m$ of $\cC$ which are reachable from each other 
by mutation, we have 
\[
\langle \index_\ell(n_1) , \index_\ell(n_2) \rangle_\ell =  \langle \index_m(n_1) , \index_m(n_2) \rangle_m
\]
for all rigid objects $n_1, n_2$ of $\cC$.
\item For every mutation between basic maximal rigid objects,
\begin{equation}
\label{equ:mutation}
\ell \cong \oplus_j \ell_j \mapsto \ell^* \cong (\oplus_{j \neq i} \ell_j) \oplus \ell^*_i, 
\end{equation}
the corresponding exchange triangles
\begin{equation}
\label{equ:ex-triangles}
\ell_i \rightarrow b \rightarrow \ell^*_i  \rightarrow \Sigma \ell_i \quad \mbox{and} \quad 
\ell^*_i \rightarrow b' \rightarrow \ell_i  \rightarrow \Sigma \ell^*_i
\end{equation}
are such that $b \cong b'$.
\end{enumerate}
\end{Proposition}

We will need several auxiliary results for the proofs of Theorem \ref{thm:gfan-to-hyperplane} and Proposition \ref{prop:gfan-to-hyperplane2}.
\begin{Definition}
\label{def:cones} \hfill 
\begin{enumerate}
\setlength\itemsep{4pt}

\item For the simplicial cone $\{\, a_1 v_1 + \cdots + a_d v_d \mid a_i \in \BR_{ \geqslant 0 } \,\} \subset \BR^d$, where $\{ v_1, \ldots, v_d \}$ in $\BR^d$ are linearly independent vectors, the rays $\BR_{ \geqslant 0 } v_i$ are called {\em{extremal rays}} of the cone.

\item For a chamber decomposition $S = \{ B_1, \ldots, B_N \}$ of $\BR^d$ and an extremal ray $\eta$ of one of the chambers, 
the {\em{bouquet $Y(\eta)$ of $\eta$}} is the union of the cones of $S$ for which $\eta$ is an extremal ray.
By a face of $Y(\eta)$ we will mean a face of one of the cones $B_i$ contained in $Y(\eta)$. 

\end{enumerate}
\end{Definition}

Analogously to condition {\rm \hyper}, we will say that a  chamber decomposition $S$ of $\BR^d$ comes from a 
hyperplane arrangement $\cH$ if the cones $B_1, \ldots, B_N$ of $S$ are those of $\cH$. If this is the case, then 
(a) the hyperplanes of $\cH$ are unions of codimension one faces of $S$ and (b) 
each hyperplane of $\cH$ is the hyperplane defined by a codimension one face of $S$. Moreover we have:

{\em{A chamber decomposition $S= \{ B_1,\ldots , B_N \}$ of $\BR^d$ comes from a hyperplane arrangement $\cH$ if and only if each hyperplane defined by a
codimension one face of $S$ is a union of codimension one faces of $S$.}}

The implication ``only if'' is obvious. For the ``if'', denote by $\cH$ the hyperplane arrangement consisting of all hyperplanes defined 
by codimension one faces of $S$. By the assumption, the union of those planes is precisely the union of the codimension one faces of $S$.
The boundary of every cone of $\cH$ is a subset of the union of codimension one faces of $S$. Hence each 
cone of $\cH$ is a union of some of the cones $B_1, \ldots, B_N$. It equals exactly one of those cones, because, by the construction of $\cH$, its interior does not intersect
the union of codimension one faces of $S$. Therefore the cones of $\cH$ are precisely $B_1, \ldots, B_N$.  

We recall the following, which was proved in \cite{DIJ} in the case of cluster tilting subcategories.  The case of basic maximal rigid objects is analogous, see Proposition \ref{pro:cones1}(v) and Lemma \ref{lem:G12}.   

\begin{Lemma}
\label{lem:gfan-decomp} If $\cC$ is maximal rigid finite ($|\mrig \cC|< \infty$), then 
the $g$-vector fan $G^{ \ell }_{ \cC }$ defines a chamber decomposition 
of $\Ksp( \add_{ \cC }\ell )_{ \BR }$ for all $\ell \in \mrig \cC$. 
\end{Lemma}
%
\begin{Corollary} 
\label{cor:rigid} 
Assume that $\cC$ is maximal rigid finite and that $\ell$ is a basic maximal rigid object of $\cC$. 
There is a bijection between the set of rigid objects of $\cC$ and $\Ksp( \add_{ \cC }\ell )_{ \BR }$ given by
\[
n \mapsto \index_\ell(n).
\]
\end{Corollary}
\begin{proof}
This map is injective by \cite[thm 2.3]{DK}. (In \cite{DK} this fact was proved in the presence of cluster-tilting 
subcategories of $\cC$, but the same proof works for basic maximal rigid objects.) 

Each chamber $B^\ell(m)$ of $G^{ \ell }_{ \cC }$ satisfies
\[
\{\, \index_\ell (n) \mid n \in \add_\cC m \,\} = B^\ell(m) \cap \Ksp( \add_\cC \ell ), 
\]
because $\{\index_\ell(m_1), \ldots, \index_\ell(m_d) \}$ is a basis of the free abelian group $\Ksp( \add_\cC \ell )$ by \cite[thm 2.6]{DK}. 
(Once again the proof of \cite{DK} works for basic maximal rigid objects.) 
It follows from Lemma \ref{lem:gfan-decomp} that
\[
\{\, \index_\ell (n)  \mid n \in \cC \; \mbox{rigid} \,\} = \Ksp( \add_\cC \ell ).
\qedhere
\]
\end{proof}

\begin{Lemma} 
\label{lem:bouquet} \hfill
\begin{enumerate}
\setlength\itemsep{4pt}
\item Let $d >2$. A chamber decomposition $S$ of $\BR^d$ comes from a hyperplane arrangement if and only if it has the property:

{\rm (*)} For every extremal ray $\eta$ of $S$ and codimension one face $F$ of $S$ containing $\eta$, 
the intersection of the bouquet $Y(\eta)$ with the hyperplane determined by $F$ is a union of faces of $S$. 

\item  Let $d =2$. A chamber decomposition $S$ of $\BR^d$ comes from a hyperplane arrangement if and only if it has the property:

{\rm (**)} For every extremal ray $\eta$ of $S$, $-\eta$ is also an extremal ray of $S$.
\end{enumerate}
\end{Lemma}

It is easy to see that for $d>2$, property (**) is a necessary condition for a chamber decomposition $S$ of $\BR^d$ to come from a 
hyperplane arrangement, but not a sufficient one. 

\begin{proof} (i) Obviously, a chamber decomposition $S$ which comes from a hyperplane arrangement has property (*). 
Assume that $S$ does not come from a hyperplane arrangement. This means that there exists a hyperplane $\alpha$ 
in $\BR^d$ that contains a face of $S$ but is not a union of faces of $S$. Denote by $\beta \subset \alpha$ the union of those faces of $S$ 
that are subsets of $\alpha$. Since $\beta \neq \alpha$ there exists an extremal ray $\eta$ of $S$  that is on the boundary 
of $\beta$ inside $\alpha \cong \BR^{d-1}$. By construction, there exists a codimension one face $F$ of $S$, that lies inside $\beta$ and contains 
$\eta$. The property (*) for the ray $\eta$ and the face $F$ is violated: $Y(\eta) \cap \alpha$ is not a union of faces of $S$
because this property would imply that $\eta \backslash \{ 0 \}$ is in the interior of $\beta$ inside $\alpha$. 

Part (ii) is direct and is left to the reader.
\end{proof}
Consider a mutation $\ell \mapsto \ell^*$ between basic maximal rigid objects as in \eqref{equ:mutation} and the corresponding exchange triangles  
\eqref{equ:ex-triangles}. 
Following \cite[sec. 3]{DK}, consider the linear isomorphisms $\phi_\pm^{\ell^*, \ell} :  \Ksp( \add_{ \cC }\ell )_{\BR} \rightarrow \Ksp( \add_{ \cC }\ell^*)_{\BR}$ such that
$\phi_+^{\ell^*, \ell} ([\ell_j]) = \phi^{\ell^*, \ell}_- ([\ell_j])  = [\ell_j]$ for $j \neq i$ and 
\[
\phi_+^{\ell^*, \ell} ([\ell_i]) = [b'] - [\ell^*_i],
\quad
\phi_-^{\ell^*, \ell} ([\ell_i]) = [b] - [\ell^*_i].
\]
We will repeatedly make use of the following formula from \cite[thm. 3]{DK}:
\begin{equation}
\label{phi+-}
\index_{\ell^*} (n) = 
\begin{cases}
\phi_+^{\ell^*, \ell}\big( \index_\ell (n) \big) &\mbox{if} \; \; [\index_\ell(n) : [\ell_i]] \geqslant 0,
\\[1.5mm]
\phi_-^{\ell^*, \ell}\big( \index_\ell (n) \big) &\mbox{if} \; \; [\index_\ell(n) : [\ell_i]] \leqslant 0
\end{cases}
\end{equation}
for all rigid objects $n$ of $\cC$.
In \cite{DK} this fact is proved for mutations of cluster tilting subcategories, but the same proof works for
mutations of basic maximal rigid objects. 
\begin{proof}[Proof of Proposition \ref{prop:gfan-to-hyperplane2}(i)] Denote $ \Ksp( \cC )_{ \BR } = \Ksp( \cC ) \Tensor{ \BZ } \BR$ where 
$\Ksp(  \cC )$ is the split Grothendieck group of the additive category $\cC$.  Consider the bilinear form
\[
\langle -, - \rangle : \Ksp(  \cC )_\BR \times \Ksp(  \cC )_\BR \to \BR
\]
given by $\langle [r], [s] \rangle = \dim \cC(r,s)$ for $r, s \in \cC$. Clearly, for every basic maximal rigid object 
$\ell$ of $\cC$, the form $\langle -, - \rangle_\ell$ equals the restriction of $\langle -, - \rangle$ 
to $\Ksp( \add_\cC  \ell )_\BR \hookrightarrow \Ksp(  \cC )_\BR$.

It is sufficient to prove the identity when $m = \ell^*$ for a 
mutation $\ell \mapsto \ell^*$ as in \eqref{equ:mutation} with mutation triangles \eqref{equ:ex-triangles}.
Denote $\widetilde{\ell}= \oplus_{j \neq i} \ell_j$. 

For $r \in \add_{\cC} \widetilde{\ell}$, applying the functor $\cC(r, -)$ to the first and second triangles in \eqref{equ:ex-triangles} gives
\[
0 = \cC(r, \Sigma^{-1} \ell_i^*) \rightarrow \cC(r, \ell_i) \rightarrow \cC(r, b) \rightarrow \cC(r, \ell_i^*) \rightarrow \cC(r, \Sigma \ell_i) =0
\]
and 
\[
0 = \cC(r, \Sigma^{-1} \ell_i) \rightarrow \cC(r, \ell_i^*) \rightarrow \cC(r, b') \rightarrow \cC(r, \ell_i) \rightarrow \cC(r, \Sigma \ell_i^*) =0.
\]
Thus,
\begin{equation}
\label{equ:3-4}
\langle \eta, [\ell_i] \rangle = \langle \eta, [b] - [\ell_i^*] \rangle = \langle \eta, [b'] - [\ell_i^*] \rangle
\quad \mbox{for $\eta \in \Ksp( \add_\cC  \widetilde{\ell} )_\BR \subset \Ksp(  \cC )_\BR$}.
\end{equation}
Similarly, applying the functor $\cC(-, r)$ to the triangles in \eqref{equ:ex-triangles} for $r \in \add_{\cC} \widetilde{\ell}$, gives
\[
0 = \cC(\Sigma \ell_i, r) \rightarrow \cC(\ell_i^*, r) \rightarrow \cC(b, r) \rightarrow \cC(\ell_i, r) \rightarrow \cC(\Sigma^{-1} \ell_i^*, r) =0
\]
and 
\[
0 = \cC(\Sigma \ell^*_i, r) \rightarrow \cC(\ell_i, r) \rightarrow \cC(b', r) \rightarrow \cC(\ell^*_i, r) \rightarrow \cC(\Sigma^{-1} \ell_i, r) =0.
\]
Therefore, 
\begin{equation}
\label{equ:5-6}
\langle [\ell_i], \eta \rangle = \langle [b] - [\ell_i^*], \eta \rangle = \langle [b'] - [\ell_i^*], \eta \rangle
\quad \mbox{for $\eta \in \Ksp( \add_\cC  \widetilde{\ell} )_\BR \subset \Ksp(  \cC )_\BR$}.
\end{equation}
Applying the functor $\cC(\ell_i, -)$ to the triangles in \eqref{equ:ex-triangles} yields
\[
0 = \cC(\ell_i, \Sigma^{-1} b) \rightarrow  \cC(\ell_i, \Sigma^{-1} \ell_i^*) \rightarrow 
\cC(\ell_i, \ell_i) \rightarrow \cC(\ell_i, b) \rightarrow \cC(\ell_i, \ell_i^*) \rightarrow \cC(\ell_i, \Sigma \ell_i) =0
\]
and 
\[
0 = \cC(\ell_i, \Sigma^{-1} \ell_i) \rightarrow \cC(\ell_i, \ell_i^*) 
\rightarrow \cC(\ell_i, b') \rightarrow \cC(\ell_i, \ell_i) \rightarrow \cC(\ell_i, \Sigma \ell_i^*)  \rightarrow \cC(\ell_i, \Sigma b')  =0.
\]
These sequences and the isomorphism $\Sigma^{-1} \ell^*_i \cong \Sigma \ell^*_i$ give
\begin{equation}
\label{equ:9-10}
\langle [\ell_i], [\ell_i] - [\Sigma \ell^*_i] \rangle = 
\langle [\ell_i], [b] - [\ell^*_i] \rangle = \langle [\ell_i], [b'] - [\ell_i^*] \rangle. 
\end{equation}
Similarly, by applying the functor $\cC(-, \ell^*_i)$ to the triangles in \eqref{equ:ex-triangles} we obtain
\begin{equation}
\label{equ:7-8}
\langle [\ell_i] - [\Sigma \ell_i], [\ell^*_i]  \rangle = 
\langle [b] - [\ell^*_i], [\ell^*_i] \rangle = \langle [b'] - [\ell_i^*], [\ell^*_i] \rangle. 
\end{equation}
Hence, for all $b_1, b_2 \in \{ b, b' \}$, we have 
\begin{align}
&\langle [b_1] - [\ell^*_i], [b_2] - [\ell^*_i] \rangle = \langle [b_1] - [\ell^*_i], [b_2] \rangle - \langle [b_1] - [\ell^*_i], [\ell^*_i] \rangle \stackrel{\eqref{equ:5-6}, \eqref{equ:7-8}}{=} 
\nonumber
\\
&\langle [\ell_i], [b_2] \rangle - \langle [\ell_i] - [\Sigma \ell_i], [\ell^*_i] \rangle = \langle [\ell_i], [b_2] - [\ell^*_i] \rangle 
+ \langle [\Sigma \ell_i], [\ell^*_i ] \rangle \stackrel{\eqref{equ:9-10}}{=} 
\nonumber
\\ 
&\langle [\ell_i], [\ell_i] - [\Sigma \ell^*_i] \rangle + \langle [\Sigma \ell_i], [\ell^*_i ] \rangle =  \langle [\ell_i], [\ell_i] \rangle.
\label{equ:long}
\end{align}
Suppose 
\[
\index_\ell(n_s) = \eta_s + \alpha_s [\ell_i] \quad
\mbox{for $s=1,2$ and some $\eta_s \in \Ksp( \add_\cC  \widetilde{\ell})$, $\alpha_i \in \BZ$}.  
\]
It follows from \eqref{phi+-} that
\[
\index_{\ell^*}(n_s) = \eta_s + \alpha_s ( [b_s] - [\ell^*_i]) 
\]
for some $b_1, b_2 \in \{b, b' \}$. Therefore, 
\begin{align*}
&\langle \index_{\ell^*}(n_1), \index_{\ell^*}(n_2) \rangle_{\ell^*} = 
\\
& \langle \eta_1, \eta_2 \rangle + \alpha_2 \langle \eta_1,  [b_2] - [\ell^*_i] \rangle + \alpha_1 \langle  [b_1] - [\ell^*_i], \eta_2 \rangle + 
\alpha_1 \alpha_2 \langle  [b_1] - [\ell^*_i],  [b_2] - [\ell^*_i] \rangle \stackrel{\eqref{equ:3-4}, \eqref{equ:5-6}, \eqref{equ:long}}{=}
\\
& \langle \eta_1, \eta_2 \rangle + \alpha_2 \langle \eta_1, [\ell_i] \rangle +  \alpha_1 \langle  [\ell_i], \eta_2 \rangle
+ \alpha_1 \alpha_2 \langle [\ell_i], [\ell_i] \rangle = \langle \index_\ell(n_1), \index_\ell(n_2) \rangle_\ell.
\end{align*}
\end{proof}
\begin{proof}[Proof of Theorem \ref{thm:gfan-to-hyperplane}(i)] For all nonzero rigid objects $n$ of $\cC$, we have 
\begin{equation}
\label{equ:positive-inner}
\langle \index_\ell(n), \index_\ell(n) \rangle_\ell = \dim \cC(n,n) >0.
\end{equation}
To see this, we choose a basic maximal rigid object $m$ of $\cC$ such that $n \in \add_\cC m$, so $\index_m(n) = [n]$.  
We apply Proposition \ref{prop:gfan-to-hyperplane2}(i) for this choice of $m$, using the fact that $m$ is reachable from $\ell$
by mutations by Lemma \ref{lem:G12}.  

We claim that the symmetric bilinear form on $\Ksp( \add_\cC \ell )_\BR$, given by 
\[
(\eta_1, \eta_2) \mapsto \langle \eta_1, \eta_2 \rangle_\ell + \langle \eta_2, \eta_1 \rangle_\ell,
\] 
is positive definite. Since $\langle [\ell_i] , [\ell_i] \rangle_\ell >0$ and the matrix $C_\ell + C^T_\ell$ of the form 
in the standard basis of $\Ksp( \add_\cC \ell )_\BR$ is integral, if the form is not positive definite, then 
$\langle \eta, \eta \rangle_\ell =0$ for some $\eta \in \Ksp(\add_\cC \ell)$. However, this contradicts 
\eqref{equ:positive-inner} because Corollary \ref{cor:rigid} says $\eta = \index_\ell(n)$ for some nonzero rigid object $n$.  

Therefore the matrix $C_\ell + C^T_\ell$ is positive definite. As a consequence, the matrix $C_\ell$ is 
nondegenerate, because $\eta^T C_\ell \eta = \langle \eta, \eta \rangle_\ell >0$ for all 
$\eta \in \Ksp(\add_\cC \ell)_\BR$, $\eta \neq 0$. 
\end{proof}

\begin{proof}[Proof of Proposition \ref{prop:gfan-to-hyperplane2}(ii)] The identity \eqref{equ:3-4} implies that
\[
\langle \eta, [b]- [b'] \rangle_\ell =0 \quad \mbox{for all $\eta \in \Ksp( \add_\cC  \widetilde{\ell} )_\BR \subset \Ksp( \add_\cC \ell )_\BR $}. 
\]
Similarly, \eqref{equ:9-10} gives 
\[
\langle [\ell_i], [b] - [b'] \rangle_\ell =0.  
\]
Hence, 
\[
\langle \eta, [b]- [b'] \rangle_\ell =0 \quad \mbox{for all $\eta \in \Ksp( \add_\cC \ell )_\BR $}. 
\]
Since the matrix $C_\ell$ is nondegenerate by Theorem \ref{thm:gfan-to-hyperplane}(i), the bilinear form \eqref{equ:ell-form} is nondegenerate. 
Thus $[b] = [b']$ as elements of $\Ksp( \add_\cC \ell )$, so $b \cong b'$. 
\end{proof}

We will say that a {\em{rigid object $n$ of $\cC$ is in the mutation class of the maximal rigid object $\ell$}} if 
is in the additive envelope of a basic maximal rigid object which is in the mutation class of $\ell$.
\begin{Lemma}
\label{lem:equiv-index} Let $\ell$ be a basic maximal rigid object of the 2-Calabi--Yau category $\cC$. The following are equivalent:
\begin{enumerate}
\setlength\itemsep{4pt}
\item For all rigid objects $n$ of $\cC$ in the mutation class of $\ell$ and all maximal rigid objects $m$
in the mutation class of $\ell$,
\[
\index_m( \Sigma n) = - \index_m(n).
\]
\item For every mutation between basic maximal rigid objects $m \cong \oplus_j m_j \mapsto m^* \cong (\oplus_{j \neq i} m_j) \oplus m^*_i$
in the mutation class of $\ell$, the corresponding exchange triangles
\[
m_i \rightarrow c \rightarrow m^*_i  \rightarrow \Sigma m_i \quad \mbox{and} \quad 
m^*_i \rightarrow c' \rightarrow m_i  \rightarrow \Sigma m^*_i
\]
satisfy $c \cong c'$.
\item For every pair of maximal rigid objects $m, m'$ in the mutation class of $\ell$, there exists a linear isomorphism 
$\phi^{m', m} :  \Ksp( \add_{ \cC } m) \rightarrow \Ksp( \add_{ \cC } m')$ such that
\[
\index_{m'} (n) =  \phi^{m', m}( \index_{m} (n))
\]
for all rigid objects $n$ of $\cC$ in the mutation class of $\ell$.
\end{enumerate}
\end{Lemma}
\begin{proof}[Proof of Lemma \ref{lem:equiv-index}] (i) $\Rightarrow$ (ii). Applying property (i) for
the basic maximal rigid object $m^*$ and $n=m_i$ gives
\[
-([c] - [m^*_i]) = \index_{m^*} (\Sigma m_i) = - \index_{m^*} ( m_i) = -([c'] - [m^*_i])
\]
in $\Ksp( \add_{ \cC } m^* )$. So $[c'] = [c]$, and thus, $c' \cong c$. 

(ii) $\Rightarrow$ (iii).  Property (ii) implies that $\phi_+^{m^*, m} = \phi_-^{m^*, m}$ for all basic maximal rigid objects $m$
in the mutation class of $\ell$. Property (iii) for $m' = m^*$ follows from \eqref{phi+-}. In the general case, 
the isomorphims $\phi^{m',m}$ are defined to be compositions of the isomorphisms $\phi_+$ for one step mutations
and one iterates the index formulas for one-step mutations.

(iii) $\Rightarrow$ (i). In the setting of (i), $n$ is in the additive envelope of a basic maximal rigid object $\ell'$ 
which is in the mutation class of $\ell$. Clearly, $\index_{\ell'}( \Sigma n) = - \index_{\ell'}(n)$. The statement 
in (i) follows by applying $\phi^{\ell, \ell'}$ to this identity and using the property (iii).
\end{proof}
Theorem \ref{thm:gfan-to-hyperplane}(ii) follows from Theorem \ref{thm:gfan-to-hyperplane}(i), Proposition \ref{prop:gfan-to-hyperplane2}(ii), and Lemmas \ref{lem:G12} 
and \ref{lem:equiv-index}.
\begin{Corollary}
\label{Cor:congruent} Assume that the 2-Calabi--Yau category $\cC$ has the property that $\Sigma^2 n \cong n$ for all rigid objects $n$ of $\cC$.
\begin{enumerate}
\setlength\itemsep{4pt}
\item The Cartan matrices of each two maximal rigid objects $\ell, m \in \cC$ are congruent. More precisely, 
\begin{equation}
\label{equ:congru}
\Cart_m = P^T \Cart_\ell P
\end{equation}
where $P$ is the matrix of the linear isomorphism $\phi^{\ell,m}$ in the standard bases of $\Ksp( \add_{ \cC } m)$ and $\Ksp( \add_{ \cC } \ell)$.
The matrix $P$ has determinant $\pm 1$, and, as a consequence, $\det \Cart_{\ell}$ is independent on the choice of maximal rigid object $\ell$.
\item Let $K_0^{\mathrm{tri}}( \cC)$ denote the Grothendieck group of the triangulated category $\cC$. 
The canonical map $\Ksp( \add_{ \cC } \ell) \to K_0^{\mathrm{tri}}( \cC)$, given by $[n] \mapsto [n]$ for $n \in \add_{\cC} \ell$, 
is an isomorphism.
\end{enumerate}
\end{Corollary}
\begin{proof} (i) The identity \eqref{equ:congru} follows from parts (i) and (ii) of Proposition \ref{prop:gfan-to-hyperplane2} and Lemma \ref{lem:equiv-index}.
The determinant of $P$ equals $\pm 1$ because $\{\index_\ell(m_1), \ldots, \index_\ell(m_d) \}$ is a basis of the free abelian group $\Ksp( \add_\cC \ell )$
by \cite[thm 2.6]{DK}. 

(ii) Proposition \ref{prop:gfan-to-hyperplane2}(ii) and Lemma \ref{lem:equiv-index} imply that $\cC$ has the property (ii) in Lemma \ref{lem:equiv-index}.
Now the isomorphism follows \cite[thm. 10]{Pa}. 
In \cite{Pa} the last fact is proved in the presence of cluster tilting subcategories, but the same proof works for basic maximal rigid objects. 
\end{proof}
 
For a basic maximal rigid object $\ell$ of $\cC$ and $\ell_\circ \in \add_\cC \ell$, denote $\cD = \add_{\cC} \ell_\circ$, 
$\cZ = {}^\perp \Sigma \cD$, and
consider the Calabi--Yau reduction of \cite[sec. 4]{IY}:
\[
\cU = \cZ/ [\add_{ \cC }( \ell_\circ )].
\]
It satisfies Setup \ref{set:blanket}(iv) by \cite[thm. 4.7]{IY}. 
There is an inclusion-preserving bijection $\cV \mapsto \cV/[\cD]$ between 
\begin{align}
&\{ \mbox{additive subcategories $\cV \subseteq \cZ$ such that $\cD \subseteq \cV$} \}  
\nonumber
\\
&\mbox{and} \quad \{ \mbox{additive subcategories of $\cU$} \}.
\label{equ:bij0}
\end{align}
Note that, if $\cV$ is an additive rigid subcategory of $\cC$ such that $\cD \subseteq \cV$, then $\cV \subseteq \cZ$, because
$\cV \subseteq {}^\perp \Sigma \cV \subseteq {}^\perp \Sigma \cD = \cZ$. 
By \cite[thm. 4.9(ii)]{IY}, the bijection \eqref{equ:bij0} restricts to an inclusion preserving bijection
\begin{align}
&\{ \mbox{(additive) rigid subcategories $\cV \subseteq \cC$ such that $\cD \subseteq \cV$} \}  
\nonumber
\\
&\mbox{and} \quad \{ \mbox{(additive) rigid subcategories of $\cU$} \}.
\label{equ:bij}
\end{align}
Denote the image in $\cU$ of an object $m$ of $\cZ$ by $\overline{m}$. The canonical projection $\pi : \Ksp( \add_{ \cC }\ell )_\BR \to \Ksp( \add_{ \cC } \ell /[\cD] )_\BR$ 
given by $\pi([m]) = [\overline{m}]$ has kernel $\Ksp( \cD )_{\BR} \hookrightarrow  \Ksp( \add_{ \cC }\ell )_{\BR}$.

\begin{Lemma}
\label{lem:index-CY-red} In the above setting the following hold:
\begin{enumerate}
\setlength\itemsep{4pt}
\item For every rigid object $n$ of $\cC$ that lies in $\cZ$, 
\[
\index_{\overline{\ell}} (\overline{n}) = \pi\big( \index_{\ell}(n) \big) ,\;\;
\index_{\overline{\ell}} (\Sigma_\cU \overline{n} ) = \pi \big( \index_\ell( \Sigma_\cC n) \big).
\]
\item If 
\[
\index_\ell( \Sigma_\cC n) = -  \index_\ell( n) \quad  \mbox{for every rigid object $n$ of $\cC$}, 
\]
then 
\[
\index_{ \overline{\ell} } ( \Sigma_\cU n') = -  \index_{\overline{\ell} } ( n') \quad \mbox{for every rigid object $n'$ of $\cU$}. 
\]
\end{enumerate}
\end{Lemma}

\begin{proof} (i): Since $n$ is rigid in $\cC$, there is a triangle $\ell' \xrightarrow{} \ell'' \xrightarrow{} n$ in $\cC$ with $\ell',\ell'' \in \add_{ \cC }\ell$.  The terms of the triangle are in $\cZ$, so it induces a triangle $\overline{\ell'} \xrightarrow{} \overline{\ell''} \xrightarrow{} \overline{n}$ in $\cU$.  The two triangles show
\[
  \index_{\overline{\ell}}( \overline{n} ) = [\overline{\ell''}] - [\overline{\ell'}] = \pi( [\ell''] - [\ell'] ) = \pi\big( \index_{\ell}( n ) \big),
\]
proving the first equation in (i).

There is a triangle $n \xrightarrow{} d \xrightarrow{} z$ in $\cC$ with $d \in \cD$, $z \in \cZ$,
and in terms of it $\Sigma_\cU \overline{n} \cong \overline{z}$, see \cite[def. 4.1]{IY}.  Hence there is a triangle $z \xrightarrow{} \Sigma_{ \cC }n \xrightarrow{} \Sigma_{ \cC }d$.
Since $\Sigma_\cC n$ is also a rigid object, there is a triangle
$\ell' \xrightarrow{} \ell'' \xrightarrow{} \Sigma_\cC n$ in $\cC$ with $\ell', \ell'' \in \add_{ \cC } \ell$.  In the following diagram, this gives the lower right square, which is commutative because $\cC( \ell'',\Sigma_{ \cC }d ) = 0$ since $\ell'',d \in \add_{ \cC } \ell$.  The diagram exists by the octahedral axiom, and each row and column is a triangle.
\[
\vcenter{
  \xymatrix @+0.2pc {
    \ell'     \ar@{=}[r] \ar[d] &  \ell'    \ar[r] \ar[d] & 0 \ar[d]  \\
    \ell'' \oplus d     \ar[r] \ar[d] &  \ell''    \ar^0[r] \ar[d] & \Sigma_\cC d \ar@{=}[d]  \\
    z   \ar[r]  &  \Sigma_\cC n    \ar[r]  & \Sigma_\cC d           \\
                    }
        }
\]
Hence we have a triangle $\ell' \xrightarrow{}  \ell'' \oplus d \xrightarrow{} z$ 
in $\cC$, which gives rise to the triangle $\overline{\ell'} \xrightarrow{}  \overline{\ell''} \xrightarrow{} \overline{z}$ in $\cU$ 
because $\overline{d} \cong 0$. Taking into account that $\Sigma_\cU \overline{n} \cong \overline{z}$ gives
\[
\index_{\overline{\ell}} ( \Sigma_\cU \overline{n} ) = [ \overline{\ell''}] - [\overline{\ell'}] =
\pi( [\ell''] - [\ell']) = \pi \big(  \index_\ell( \Sigma_\cC n)  \big),   
\]
where the last equality is obtained from the triangle $\ell' \xrightarrow{} \ell'' \xrightarrow{} \Sigma_\cC n$.  This proves the second equation in (i).

Part (ii) directly follows from the first part and the bijection \eqref{equ:bij}.
\end{proof}

Theorem \ref{thm:gfan-to-hyperplane}(iii) follows from the following fact, Proposition \ref{prop:gfan-to-hyperplane2}(ii) and Lemmas 
\ref{lem:G12} and \ref{lem:equiv-index}.

\begin{Proposition}
\label{prop:gfan-to-hyperplane3} Assume that $\cC$ is maximal rigid finite (i.e. $|\mrig \cC| < \infty$) and that it has the property 
\[
\index_\ell( \Sigma n) = - \index_\ell(n) 
\]
for all rigid objects $n$ of $\cC$ and $\ell \in \mrig \cC$. 
Then  $\cC$ satisfies condition  {\rm \hyper}, i.e. the $g$-vector fan decomposition $G^{ \ell }_{ \cC }$ 
comes from a simplicial hyperplane arrangement in $\Ksp( \add_{ \cC }\ell )_{ \BR }$ for
all $\ell \in \mrig \cC$.
\end{Proposition}
In addition to being a stronger statement, Proposition \ref{prop:gfan-to-hyperplane3} 
has an advantage over Theorem \ref{thm:gfan-to-hyperplane}(iii) in that the former can be proved by 
induction on the rank of the category $\cC$ because its assumptions behave well under Calabi--Yau reduction.
\begin{proof} Lemma \ref{lem:G12} and Proposition \ref{prop:gfan-to-hyperplane2}(ii) imply that all basic maximal rigid objects 
of $\cC$ are reachable from $\ell$ and that conditions (i)--(iii) in Lemma \ref{lem:equiv-index} are satisfied. Denote by $d$ 
the rank of $\Ksp( \add_{ \cC }\ell )$.

Case $d=2$. The proposition follows from Lemma \ref{lem:bouquet}(ii) because condition (i) in Lemma \ref{lem:equiv-index} 
is satisfied.

Case $d>2$. We argue by induction on $d$, assuming that the statement of the proposition holds for ranks $\leqslant  d-1$.
We verify that the $g$-vector fan $G^{ \ell }_{ \cC }$ satisfies condition (*) in  Lemma \ref{lem:bouquet}. Every extremal ray of $G^{ \ell }_{ \cC }$
has the form $\eta = \BR_{ \geqslant 0 } \index_\ell(n)$ where $n$ is an indecomposable summand of a basic maximal rigid 
object of $\cC$. After applying the linear isomorphisms from Lemma \ref{lem:equiv-index}(iii) we can assume that $n = \ell_i$ 
for some $1 \leqslant i \leqslant n$. Fix a face of codimension one of $Y(\eta)$ containing $\eta$. It has the
form  $F=B^{ \ell }( m )$ where $m$ is a direct summand of a basic maximal
rigid object $\ell'$ such that $m$ contains $\ell_i$ and has $d-1$
indecomposable summands. After applying the linear isomorphism
$\phi^{\ell',\ell}$ from Lemma \ref{lem:equiv-index}(iii), we can assume
without loss of generality that $\ell' = \ell$. Denote by $\alpha$ the hyperplane in $\Ksp( \add_{ \cC }\ell )_{ \BR }$ determined by $F$. 

Consider the Calabi-Yau reduction for $\ell_\circ = \ell_i$; that is $\cU = \cZ/ [\cD]$, where 
$\cD = \add_{\cC} \ell_i$, $\cZ = {}^\perp \Sigma \ell_i$. The corresponding projection 
$\pi : \Ksp( \add_{ \cC }\ell )_\BR \to \Ksp( \add_{ \cU } \ell /[\cD] )_\BR$ 
has kernel $\BR [\ell_i] \subset \Ksp( \add_{ \cC }\ell )$. The 2-Calabi--Yau category $\cU$ satisifies
the assumptions of the proposition by Lemma \ref{lem:index-CY-red}(ii) and the bijection \eqref{equ:bij}.

The hyperplane $\pi(\alpha)$ in $\Ksp( \add_{ \cU } \ell /[\cD] )_\BR$ contains the 
face $B^{ \overline{\ell} }( \overline{m} )$ of the $g$-vector fan $G^{ \overline{\ell} }_{ \cU }$. 
The statement of the proposition holds in dimension $d-1$, so $\pi(\alpha)$ is a union of codimension one faces of 
$G^{ \overline{\ell} }_{ \cU }$. By the bijection \eqref{equ:bij}, these faces are of the form 
$B^{ \overline{\ell} }( \overline{m_s} )$, $1 \leqslant s \leqslant t$ where $m_s \in \cZ$ is a direct 
summand of a basic maximal rigid object of $\cC$ with the properties
that $\ell_i$ is a direct summand of $m_s$ and $m_s$ has $d-1$ indecomposable direct summands. 
We have
\[
\pi(\alpha) = \cup_{s=1}^t B^{ \overline{\ell} }( \overline{m_s} ) \quad \mbox{and} \quad  
\pi^{-1} \big( B^{ \overline{\ell} }( \overline{m_s} ) \big) \cap Y(\eta) = B^\ell(m_s).
\]
Hence
\[
\alpha \cap Y(\eta) = \cup_{s=1}^t  B^\ell(m_s). 
\]
Since $B^\ell(m_s)$, $1 \leqslant s \leqslant t$ are faces of the $g$-vector fan $G^{ \ell }_{ \cC }$,  
this proves that $G^{ \ell }_{ \cC }$ satisfies condition (*) in  Lemma \ref{lem:bouquet}
and the proposition now follows from Lemma \ref{lem:bouquet}(i).
\end{proof}

\section{Example: The cluster category of Dynkin type $A_2$}
\label{sec:A2}

This section demonstrates our theory for the cluster category of Dynkin type $A_2$.

\subsection{The quiver $Q$}

Set $k = \BC$ and consider the quiver $Q$ in Figure \ref{fig:JKSquiver}, which comes from \cite[sec.\ 2]{JKS}.  In the notation of \cite{JKS} we set $n=5$, $k=2$.  This clashes with our use of $k$ to denote the ground field, but we will not need to repeat it.
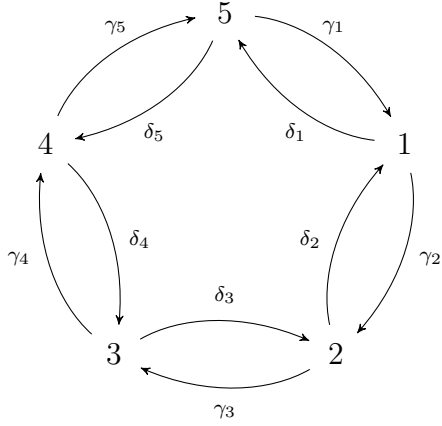
\begin{figure}
  \centering
    \begin{tikzpicture}[>=stealth']
      \node[name=s, color=white, shape=regular polygon, regular polygon sides=5, minimum size=5cm, draw] {};
      \draw[shift=(s.corner 1)] node {$5$};
      \draw[shift=(s.corner 2)] node {$4$};
      \draw[shift=(s.corner 3)] node {$3$};
      \draw[shift=(s.corner 4)] node {$2$};
      \draw[shift=(s.corner 5)] node {$1$};
      \draw[shorten >=0.4cm,shorten <=0.4cm,->] (s.corner 1) to[bend left] node[above=1mm] {$\scriptstyle \gamma_1$} (s.corner 5);
      \draw[shorten >=0.4cm,shorten <=0.4cm,->] (s.corner 5) to[bend left] node[below=1mm] {$\scriptstyle \delta_1$} (s.corner 1);
      \draw[shorten >=0.4cm,shorten <=0.4cm,->] (s.corner 5) to[bend left] node[right=1mm] {$\scriptstyle \gamma_2$} (s.corner 4);
      \draw[shorten >=0.4cm,shorten <=0.4cm,->] (s.corner 4) to[bend left] node[left=1mm] {$\scriptstyle \delta_2$} (s.corner 5);
      \draw[shorten >=0.4cm,shorten <=0.4cm,->] (s.corner 4) to[bend left] node[below=1mm] {$\scriptstyle \gamma_3$} (s.corner 3);
      \draw[shorten >=0.4cm,shorten <=0.4cm,->] (s.corner 3) to[bend left] node[above=1mm] {$\scriptstyle \delta_3$} (s.corner 4);
      \draw[shorten >=0.4cm,shorten <=0.4cm,->] (s.corner 3) to[bend left] node[left=1mm] {$\scriptstyle \gamma_4$} (s.corner 2);
      \draw[shorten >=0.4cm,shorten <=0.4cm,->] (s.corner 2) to[bend left] node[right=1mm] {$\scriptstyle \delta_4$} (s.corner 3);
      \draw[shorten >=0.4cm,shorten <=0.4cm,->] (s.corner 2) to[bend left] node[above=1mm] {$\scriptstyle \gamma_5$} (s.corner 1);
      \draw[shorten >=0.4cm,shorten <=0.4cm,->] (s.corner 1) to[bend left] node[below=1mm] {$\scriptstyle \delta_5$} (s.corner 2);
    \end{tikzpicture} 
  \caption{The quiver $Q$.}
\label{fig:JKSquiver}
\end{figure}
Let $\widehat{ kQ }$ be the completion of the path algebra at the arrow ideal, and set $R = \widehat{ kQ }/\fr$, where $\fr$ is the completion of the ideal defined by the ten relations which can be written in shorthand as $\gamma\delta = \delta\gamma$ and $\gamma^2 = \delta^3$.

\subsection{The Frobenius category $\cE$}

Let
\[
  \cE = \operatorname{CM}( R ) = \{\, x \in \mod\,R \mid \Ext_R^{ >0 }( x,R ) = 0 \,\}
\]
be the category of maximal Cohen--Macaulay modules.  By \cite[cor.\ 3.7]{JKS}, the category $\cE$, with conflations defined to be the sequences of objects of $\cE$ which are short exact in $\mod\,R$, is a Frobenius category in which the projective-injective objects coincide with the projective $R$-right modules, so $\cE$ satisfies Setup \ref{set:blanket}(ii) with $r = R$.

By \cite[rmk.\ 3.3]{JKS}, the category $\cE$ is Krull--Schmidt, and by \cite[exa.\ 5.3]{JKS} its isomorphism classes of indecomposable objects can be  indexed by two-element subsets of $\{ 1, \ldots, 5 \}$.  We write $ij$ for a suitable object in the isomorphism class indexed by $\{ i,j \}$, with the stipulation $ij = ji$.  Hence a set of representatives of the isomorphism classes of indecomposable objects of $\cE$ is given by $\{ 12, 23, 34, 45, 51, 13, 14, 24, 25, 35 \}$.  The first five, $12, 23, 34, 45, 51$, are projective and $R = 12 \oplus 23 \oplus 34 \oplus 45 \oplus 51$.

\subsection{The $2$-Calabi--Yau triangulated category $\cC$}

By \cite[exa.\ 5.3]{JKS} the stable category is
\[
  \cC = \underline{ \cE } = \cC( A_2 ),
\]
the cluster category of Dynkin type $A_2$, which satisfies Setup \ref{set:blanket}(iv), see \cite[sec.\ 1]{BMRRT}.  The AR quiver of $\cC$ is shown in Figure \ref{fig:CA2quiver}; note that it is cyclic.
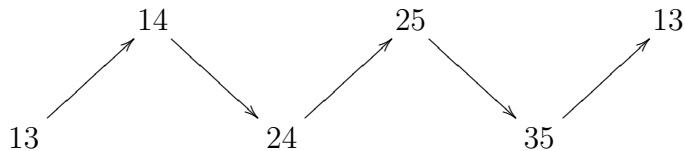
\begin{figure}
\[
\vcenter{
  \xymatrix @+0.5pc {
    & 14 \ar[dr] & & 25 \ar[dr] & & 13 \\
    13 \ar[ur] & & 24 \ar[ur] & & 35 \ar[ur] \\
                    }
        }
\]
  \caption{The Auslander--Reiten quiver of the cluster category $\cC = \cC( A_2 )$ of Dynkin type $A_2$.}
\label{fig:CA2quiver}
\end{figure}

The maximal rigid objects of $\cC$ are cluster tilting; up to isomorphism they are $\{ 13 \oplus 14, 14 \oplus 24, 24 \oplus 25, 25 \oplus 35, 35 \oplus 13 \}$, see \cite[rmk.\ 5.3]{CCS}.  To get $\mrig \cC$ we must add the module $R$ to each of these, so
\[
  \mrig \cC = \{ 13 \oplus 14 \oplus R, 14 \oplus 24 \oplus R, 24 \oplus 25 \oplus R, 25 \oplus 35 \oplus R, 35 \oplus 13 \oplus R \}.
\]

\subsection{The green groupoid $\cG_{ \cC }$}

Recall from Definition \ref{def:green_groupoid} the quiver $\Gamma_{ \cC }$, shown in Figure \ref{fig:GammaCA2}, and its free groupoid $\cF_{ \cC }$.
\begin{figure}
  \centering
    \begin{tikzpicture}[scale=3.5]
      \draw (90:1cm) node[fill=white,rectangle,inner sep=0.07cm] {$13 \oplus 14 \oplus R$};
      \draw (162:1cm) node[fill=white,rectangle,inner sep=0.07cm] {$14 \oplus 24 \oplus R$};      
      \draw (234:1cm) node[fill=white,rectangle,inner sep=0.07cm] {$24 \oplus 25 \oplus R$};
      \draw (306:1cm) node[fill=white,rectangle,inner sep=0.07cm] {$25 \oplus 35 \oplus R$};
      \draw (378:1cm) node[fill=white,rectangle,inner sep=0.07cm] {$35 \oplus 13 \oplus R$};

      \draw[->] (381:1.1cm) to[bend right] (71:1.1cm);      
      \draw[->] (74:0.9cm) to[bend right] (386:0.9cm);      

      \draw[->] (109:1.1cm) to[bend right] (159:1.1cm);      
      \draw[->] (154:0.9cm) to[bend right] (106:0.9cm);      

      \draw[->] (168:1.1cm) to[bend right] (222:1.1cm);
      \draw[->] (233:0.9cm) to[bend right] (165:0.9cm);      

      \draw[->] (238:1.05cm) to[bend right] (302:1.05cm);      
      \draw[->] (304:0.85cm) to[bend right] (236:0.85cm);      

      \draw[->] (318:1.1cm) to[bend right] (372:1.1cm);      
      \draw[->] (375:0.9cm) to[bend right] (309:0.9cm);
      
      \draw (46:1.24cm) node{$\scriptstyle \alpha_5$};
      \draw (53:0.64cm) node{$\scriptstyle \beta_5$};

      \draw (134:1.24cm) node{$\scriptstyle \alpha_1$};
      \draw (127:0.64cm) node{$\scriptstyle \beta_1$};

      \draw (198:1.23cm) node{$\scriptstyle \alpha_2$};
      \draw (198:0.51cm) node{$\scriptstyle \beta_2$};

      \draw (270:1.13cm) node{$\scriptstyle \alpha_3$};
      \draw (270:0.49cm) node{$\scriptstyle \beta_3$};

      \draw (342:1.23cm) node{$\scriptstyle \alpha_4$};
      \draw (342:0.51cm) node{$\scriptstyle \beta_4$};
                  
    \end{tikzpicture}
  \caption{The quiver $\Gamma_{ \cC }$.  Its free groupoid is $\cF_{ \cC }$, and the green groupoid $\cG_{ \cC }$ is obtained from $\cF_{ \cC }$ by identifying parallel green paths.  The resulting relations can be written $\alpha^3 = \beta^2$.}
\label{fig:GammaCA2}
\end{figure}
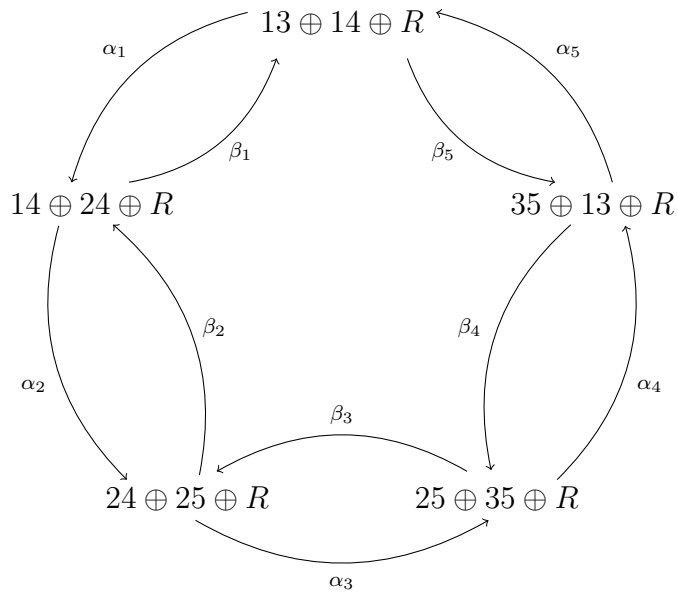
Note that the resemblance between Figures \ref{fig:JKSquiver} and \ref{fig:GammaCA2} is an artefact of the small size of the example: If, in the notation of \cite{BKM}, we set $n=6$, then Figure \ref{fig:JKSquiver} would grow by one vertex in an obvious way, but Figure \ref{fig:GammaCA2} would become far more complicated, with a vertex for each of the $14$ cluster tilting subcategories of $\cC( A_3 )$.

To obtain the green groupoid $\cG_{ \cC } = \cF_{ \cC }/\sim$, we must set equal certain pairs of green paths in $\cF_{ \cC }$.
\begin{figure}
  \centering
    \begin{tikzpicture}[scale=5.0]
      \draw[->] (0,0) to (1cm,0);
      \draw[->] (0,0) to (0,1cm);
      \draw[->] (0,0) to (-1cm,1cm);
      \draw[->] (0,0) to (-1cm,0);      
      \draw[->] (0,0) to (0,-1cm);

      \draw (1.21cm,0) node{$\index_{ \ell }( 13 )$};
      \draw (0,1.07cm) node{$\index_{ \ell }( 14 )$};
      \draw (-1.01cm,1.07cm) node{$\index_{ \ell }( 24 )$};
      \draw (-1.20cm,0) node{$\index_{ \ell }( 25 )$};
      \draw (0,-1.07cm) node{$\index_{ \ell }( 35 )$};

      \draw (0.5cm,0.51cm) node{$B^{ \ell }( 13 \oplus 14 \oplus R ) = B^{ \ell }_+$};
      \draw (-0.33cm,0.72cm) node{$B^{ \ell }( 14 \oplus 24 \oplus R )$};
      \draw (-0.75cm,0.37cm) node{$B^{ \ell }( 24 \oplus 25 \oplus R )$};
      \draw (-0.5cm,-0.51cm) node{$B^{ \ell }( 25 \oplus 35 \oplus R ) = B^{ \ell }_-$};
      \draw (0.5cm,-0.51cm) node{$B^{ \ell }( 35 \oplus 13 \oplus R )$};

    \end{tikzpicture}
  \caption{The $g$-vector fan of $\cC$ with respect to $\ell = 13 \oplus 14 \oplus R$.}
\label{fig:CA2fan}
\end{figure}
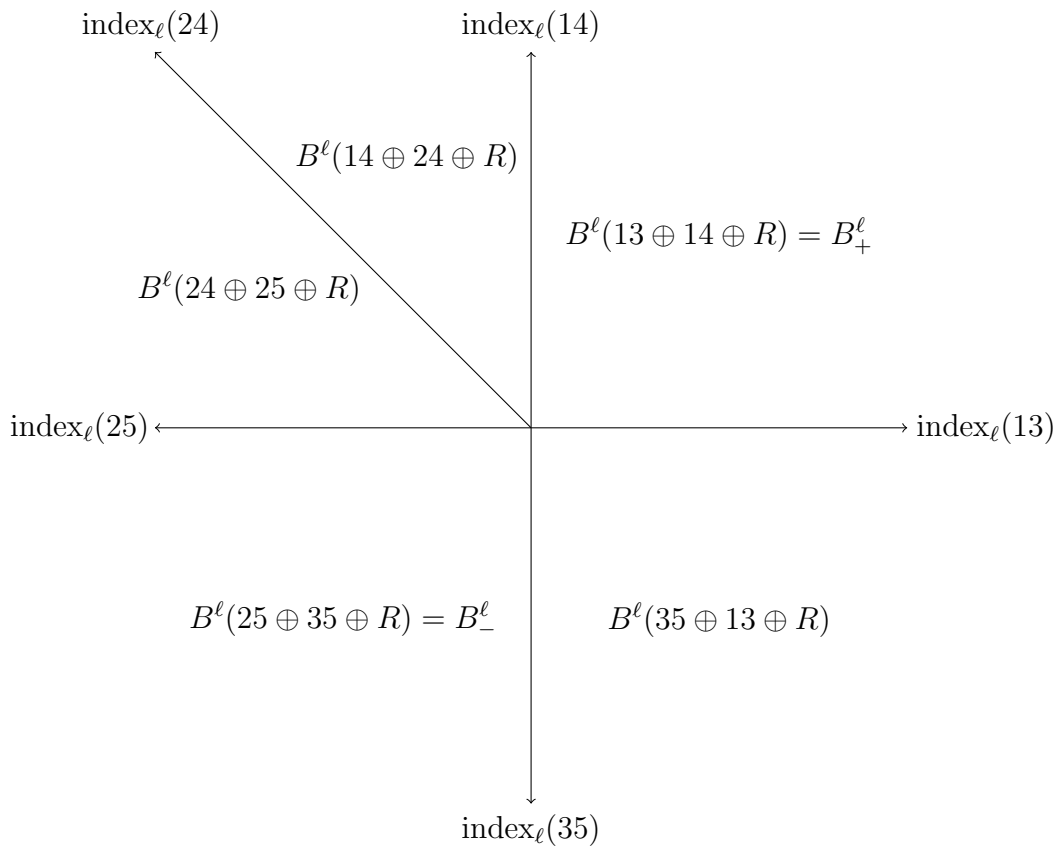
The $g$-vector fan of $\cC$ with respect to $\ell = 13 \oplus 14 \oplus R$ is shown in Figure \ref{fig:CA2fan}, and Lemma \ref{lem:G8} shows that there are green paths
\[
  13 \oplus 14 \oplus R \xrightarrow{ \alpha_1 } 14 \oplus 24 \oplus R \xrightarrow{ \alpha_2 } 24 \oplus 25 \oplus R \xrightarrow{ \alpha_3 } 25 \oplus 35 \oplus R
\]
and
\[
  13 \oplus 14 \oplus R \xrightarrow{ \beta_5 } 35 \oplus 13 \oplus R \xrightarrow{ \beta_4 } 25 \oplus 35 \oplus R,
\]
which must be set equal in $\cG_{ \cC }$.  We get four other relations by computing the $g$-vector fan with respect to each of the other four objects in $\mrig \cC$ and using Lemma \ref{lem:G8}.  The resulting five relations in $\cG_{ \cC }$ are
\begin{equation}
\label{equ:alphabeta}
  \alpha_{i+3} \alpha_{i+2} \alpha_{i+1}  = \beta_{i-1} \beta_i
\end{equation}
for $1 \leq i \leq 5$.
Here and below the index $i$ is taken modulo $5$. Denote by $i$ the vertex on Figure \ref{fig:GammaCA2} which is the head of $\alpha_i$ and the tail of $\beta_i$. 
Denote the morphisms from $i$ to $i-2$
\[
a_i = \alpha_{i-1}^{-1} \beta_i,
\quad
b_i = \beta_{i+3}^{-1} \alpha_{i+2} \alpha_i.
\]
In terms of those 
\[
\alpha_i = a_{i-3} b_{i-1}, \quad \beta_i = a_{i-4} b_{i-2} a_i.
\]
Therefore, the free groupoid with generators $\alpha_i, \beta_i$ is isomorphic to the free groupoid with generators $a_i, b_i$. 
In terms of the morphisms $a_i, b_i$, the relations \eqref{equ:alphabeta} become
\begin{equation}
\label{equ:ab}
a_{i-4} b_{i-2} a_i = b_{i-4} a_{i-2} b_i.
\end{equation}
In other words, the green groupoid $\cC(A_2)$ is the groupoid with generators
\[
5 \xrightarrow{a_5,b_5} 3 \xrightarrow{a_3,b_3} 1 \xrightarrow{a_1, b_1} 4 \xrightarrow{a_4,b_4} 2 \xrightarrow{a_2, b_2} 5
\]
and relations given by \eqref{equ:ab} for $1 \leq i \leq 5$. It can be viewed is a 5 point blow-up of the braid group $\cB_2$.
For $m \in \mrig \cC$, the green group $\cG_{ \cC(A_2) }( m,m )$ can be computed using \cite[thm.\ 7]{Higgins}; in each case it
is isomorphic to the braid groups $\cB_2$
\[
  \cG_{ \cC(A_2) }( m,m ) \cong \cB_2 = \langle a,b \mid aba = bab \rangle.
\]

\subsection{The endomorphism algebras $\cE( m,m )$}

Recall from Theorem \ref{thm:G11}(iii) the functor $\cG_{ \cC } \xrightarrow{ G } \cP_{ \cE }$ where $\cP_{ \cE }$ is the derived Picard groupoid with
\[
  \cP_{ \cE }( \ell,m ) = \{ \mbox{isomorphism classes of two-sided tilting complexes ${}_{ \cE( m,m ) }T_{ \cE( \ell,\ell ) }$} \}.
\]
We will determine the endomorphism algebras $\cE( m,m )$.  Since the example is small, they are isomorphic for all $m \in \mrig \cC$, and by \cite{BKM}, they are dimer algebras.  Specifically, consider $m = 14 \oplus 24 \oplus R$ and let $D$ be the corresponding $(2,5)$-Postnikov diagram shown in black in Figure \ref{fig:dimer}, see \cite[p.\ 213 and def.\ 2.1]{BKM}.
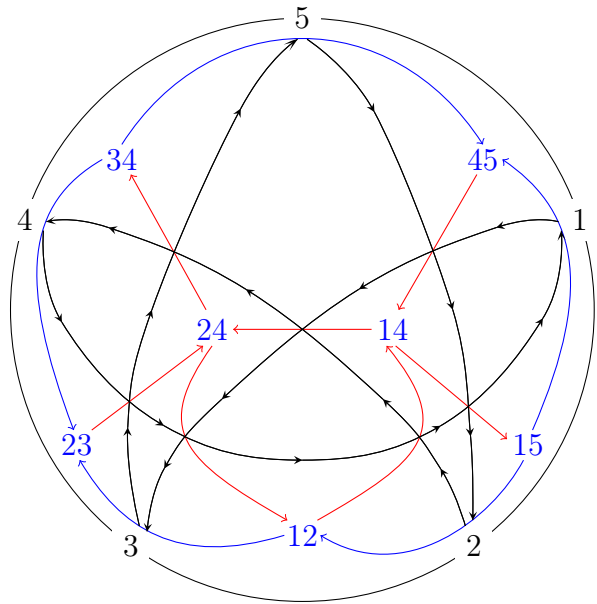
\begin{figure}
  \centering
    \begin{tikzpicture}[scale=4]

      \draw (0,0) circle (0.97cm);
      
      \draw (18:0.97cm) node[fill=white] {1};
      \draw (306:0.97cm) node[fill=white] {2};
      \draw (234:0.97cm) node[fill=white] {3};
      \draw (162:0.97cm) node[fill=white] {4};
      \draw (90:0.97cm) node[fill=white] {5};                  

	  \draw (0.3,-0.065) node[blue] {14};
	  \draw (-0.3,-0.065) node[blue] {24};
	  \draw (0.6,0.5) node[blue] {45};
	  \draw (-0.6,0.5) node[blue] {34};
	  \draw (-0.75,-0.45) node[blue] {23};
	  \draw (0,-0.75) node[blue] {12};
	  \draw (0.75,-0.45) node[blue] {15};

      \draw[->,blue] (-0.6,0.55) .. controls (-0.3,1.02) and (0.3,1.02) .. (0.6,0.55);
      \draw[->,blue] (0.765,-0.40) .. controls (6:0.99cm) and (21:0.96cm) .. (0.665,0.5);
      \draw[->,blue] (-0.665,0.5) .. controls (159:0.985cm) and (174:0.955cm) .. (-0.76,-0.40);
      \draw[->,blue] (-0.06,-0.75) .. controls (244:0.955cm) and (230:0.945cm) .. (-0.74,-0.5);
      \draw[->,blue] (0.735,-0.5) .. controls (312:0.94cm) and (290:0.985cm) .. (0.06,-0.75);

      \draw[->,red] (0.23,-0.065) to (-0.23,-0.065);
      \draw[->,red] (0.58,0.45) to (0.32,0);
      \draw[->,red] (-0.32,0) to (-0.57,0.45);
      \draw[->,red] (-0.7,-0.4) to (-0.33,-0.12);
      \draw[->,red] (0.31,-0.12) to (0.68,-0.43);

      \draw[->,red] (-0.3,-0.12) .. controls (-0.5,-0.4) and (-0.4,-0.5) .. (-0.05,-0.7);
      \draw[->,red] (0.05,-0.7) .. controls (0.4,-0.5) and (0.51,-0.4) .. (0.28,-0.12);

	  \begin{scope}

        \draw[black,thin] plot [smooth,tension=0.7]
        coordinates {(19:0.9cm) (23:0.7cm) (20:0.2cm) (228:0.4cm) (229:0.7cm) (235:0.9cm)}
        [postaction={on each segment={draw,-{stealth[black,bend]}}}];

        \draw[black,thin] plot [smooth,tension=0.7]
        coordinates {(307:0.9cm) (311:0.7cm) (312:0.4cm) (160:0.2cm) (157:0.7cm) (161:0.9cm)}
        [postaction={on each segment={draw,-{stealth[black,bend]}}}];

        \draw[black,thin] plot [smooth,tension=0.7]
        coordinates {(233:0.9cm) (214:0.7cm) (180:0.5cm) (107:0.7cm) (91:0.9cm)}
        [postaction={on each segment={draw,-{stealth[black,bend]}}}];

        \draw[black,thin] plot [smooth,tension=0.7]
        coordinates {(163:0.9cm) (183:0.8cm) (220:0.6cm) (270:0.5cm) (-40:0.6cm) (-2:0.8cm) (17:0.9cm)}
        [postaction={on each segment={draw,-{stealth[black,bend]}}}];

        \draw[black,thin] plot [smooth,tension=0.6]
        coordinates {(89:0.9cm) (70:0.7cm) (0:0.5cm) (-37:0.7cm) (-51:0.9cm)}
        [postaction={on each segment={draw,-{stealth[black,bend]}}}];

	  \end{scope}
	      
    \end{tikzpicture} 
  \caption{A $(2,5)$-Postnikov diagram $D$ (black) and its associated quiver $Q(D)$ (blue and red).}
\label{fig:dimer}
\end{figure}
Its associated quiver $Q(D)$ is shown in blue and red with red signifying internal arrows, see \cite[def.\ 2.4]{BKM}.  Figure \ref{fig:Emm} shows $Q(D)$ on its own.
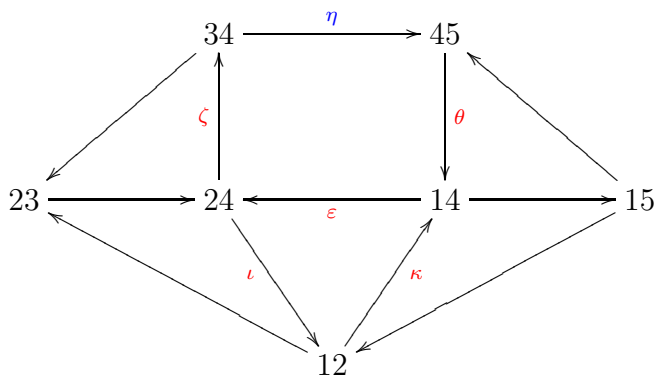
\begin{figure}
\[
\vcenter{
  \xymatrix @R=4pc @C=2pc {
    && *+[blue]{34} \ar@[blue][dll] \ar@[blue]^{\color{blue}\eta}[rr] && *+[blue]{45} \ar@[red]^{\color{red}\theta}[d] \\
    *+[blue]{23} \ar@[red][rr] && *+[blue]{24} \ar@[red]^{\color{red}\zeta}[u] \ar@[red]_<<<<<<<{\color{red}\iota}[dr] && *+[blue]{14} \ar@[red]^{\color{red}\varepsilon}[ll] \ar@[red][rr] && *+[blue]{15} \ar@[blue][llu] \ar@[blue][llld] \\
    &&& *+[blue]{12} \ar@[blue][lllu] \ar@[red]_>>>>>>>{\color{red}\kappa}[ru] \\
                    }
        }
\]
  \caption{The quiver $Q(D)$ on its own.}
\label{fig:Emm}
\end{figure}
Let $\widehat{ kQ(D) }$ be the completion of the path algebra at the arrow ideal.  Then \cite[thm.\ 11.2]{BKM} says that $\cE( m,m )$ is the dimer algebra of $Q(D)$, that is,
\[
  \cE( m,m ) \cong \widehat{ kQ(D) }/\fc
\]
where $\fc$ is the completion of the ideal defined by certain commutativity relations.  Each internal arrow gives such a relation, stating that the two alternative routes from target to source are equal.  For instance, the arrow $\varepsilon$ gives the relation $\zeta\eta\theta = \iota\kappa$, see \cite[def.\ 3.5]{BKM}.

\appendix

\numberwithin{equation}{section}
\renewcommand{\theequation}{\Alph{section}.\arabic{equation}}

\renewcommand{\thesubsection}{\Alph{section}.\Roman{subsection}}

\section{Frobenius categories}
\label{sec:Frobenius}

This appendix collects results about the Frobenius category $\cE$ and its relation to the stable category $\cC = \underline{ \cE }$, still assuming Setup \ref{set:blanket}.  Recall that if $\alpha$ is a morphism in $\cE$, then $\underline{ \alpha }$ denotes its image in $\cC$.  The following lemma is a portmanteau of well-known results, and we omit the proof.

\begin{Lemma}
\label{lem:portmanteau}
\begin{enumerate}
\setlength\itemsep{4pt}

  \item  Given $x,x' \in \cE$ we have $x \cong x'$ in $\cC$ if and only if there exist projective-injective objects $p,p' \in \cE$ such that $x \oplus p \cong x' \oplus p'$ in $\cE$.  

  \item  If $x \in \cC$ is good then $\add_{ \cC }( x ) = \add_{ \cE }( x )$ (the equation makes sense because $\cC$ and $\cE$ have the same objects).

  \item  If $x \xrightarrow{ \xi } z$ is a morphism, $y \xrightarrow{ \upsilon } z$ a deflation in $\cE$, then $x \oplus y \xrightarrow{ ( \xi,\upsilon ) } z$ is a deflation in $\cE$.  

  \item  Given $x \in \cE$, set $A = \cE( x,x )$.  The functor $\cE( x,- )$ induces a $k$-linear map
\[
  \cE( y,z ) \xrightarrow{} \Hom_A\!\big( \cE(x,y),\cE(x,z) \big),
\]
which is natural in $y,z \in \cE$, and is bijective for $y \in \add_{ \cE }( x )$.  

  \item  Given $y \in \cE$, set $B = \cE( y,y )$.  Composition of morphisms induces a $k$-linear map
\[
  \cE( y,z ) \Tensor{B} \cE( x,y ) \xrightarrow{} \cE( x,z ),
\]
which is natural in $x,z \in \cE$, and is bijective for $x \in \add_{ \cE }( y )$ or $z \in \add_{ \cE }( y )$.  

  \item  Given $x \in \cE$, set $A = \cE( x,x )$ and $\underline{ A } = \cC( x,x )$, and observe that there is a canonical surjection $A \xrightarrow{} \underline{ A }$.  There is an isomorphism of $\underline{ A }$-left modules
\[
  {}_{ \underline{ A } }\underline{ A }_A \Tensor{ A } \cE( x^0,x ) \xrightarrow{\cong} \cC( x^0,x ),
\]
which is natural in $x^0 \in \add_{ \cE }( x )$.  

  \item  Each triangle in $\cC$,
\[  
  x' \xrightarrow{} y' \xrightarrow{} z' \xrightarrow{} \Sigma x',
\]
is isomorphic to the triangle induced by a conflation in $\cE$ of the form
\[
  0 \xrightarrow{} x' \xrightarrow{} y \xrightarrow{} z \xrightarrow{} 0,
\]
and to the triangle induced by a conflation in $\cE$ of the form
\[
  0 \xrightarrow{} x \xrightarrow{} y \xrightarrow{} z' \xrightarrow{} 0.
\]

  \item  Given a conflation
\[
  0 \xrightarrow{} x \xrightarrow{ \xi } y \xrightarrow{ \upsilon } z \xrightarrow{} 0
\]
in $\cE$, consider the induced triangle
\[
  x \xrightarrow{ \underline{\xi} } y \xrightarrow{ \underline{\upsilon} } z \xrightarrow{} \Sigma x
\]  
in $\cC$.  For each $u \in \cC$ we have
\begin{align*}
  \mbox{$\cE( \xi,u )$ surjective } &
  \mbox{$\; \Leftrightarrow \;\cC( \underline{\xi},u )$ surjective, } \\[1mm]
  \mbox{$\cE( u,\upsilon )$ surjective } &
  \mbox{$\; \Leftrightarrow \; \cC( u,\underline{\upsilon} )$ surjective. }
\end{align*}

\end{enumerate}
\end{Lemma}

\begin{Lemma}
\label{lem:F9_10}
Let $x \in \cC$ be a good maximal rigid object, $y \in \cC$ a rigid object.
\begin{enumerate}
\setlength\itemsep{4pt}

  \item  There exists a conflation in $\cE$ with $x_i \in \add_{ \cE }( x )$:
\begin{equation}
\label{equ:F9_10a}
  0 \xrightarrow{} x_1 \xrightarrow{ \xi_1 } x_0 \xrightarrow{ \xi_0 } y \xrightarrow{} 0.
\end{equation}  
  
  \item  Each such conflation induces an augmented projective resolution of $\cE( x,y )$ as an $\cE( x,x )$-right module:
\begin{equation}
\label{equ:F9_10b}
  0 \xrightarrow{} \cE( x,x_1 ) \xrightarrow{ \cE( x,\xi_1 ) } \cE( x,x_0 )\xrightarrow{ \cE( x,\xi_0 ) } \cE( x,y ) \xrightarrow{} 0.
\end{equation}  

  \item  $\cE( x,y )$ is a finitely presented $\cE( x,x )$-right module of projective dimension at most $1$.

\end{enumerate}
\end{Lemma}

\begin{proof}
(i):  By Lemma \ref{lem:ZZ} there exists a triangle $x'_1 \xrightarrow{} x'_0 \xrightarrow{} y \xrightarrow{} \Sigma x'_1$ in $\cC$ with $x'_i \in \add_{ \cC }( x )$.  By Lemma \ref{lem:portmanteau}(vii) there exists a conflation \eqref{equ:F9_10a} in $\cE$ with $x_i \cong x'_i$ in $\cC$.  In particular, $x_i \in \add_{ \cC }( x )$ whence $x_i \in \add_{ \cE }( x )$ by Lemma \ref{lem:portmanteau}(ii).

(ii):  The conflation \eqref{equ:F9_10a} induces a triangle $x_1 \xrightarrow{} x_0 \xrightarrow{ \underline{ \xi }_0 } y \xrightarrow{} \Sigma x_1$ in $\cC$.  We have $\cC( x,\Sigma x_1 ) = 0$ because $x$ is rigid in $\cC$ while $x_1 \in \add_{ \cE }( x )$ whence $x_1 \in \add_{ \cC }( x )$ by Lemma \ref{lem:portmanteau}(ii).  Hence $\cC( x, \underline{ \xi }_0 )$ is surjective whence $\cE( x,\xi_0 )$ is surjective by Lemma \ref{lem:portmanteau}(viii).  This implies (ii).

(iii):  Immediate by (ii).
\end{proof}

\begin{Lemma}
\label{lem:F8}
If $x \in \cE$ is good, then each $v_i \in \cE$ permits a conflation $0 \xrightarrow{} v_{ i+1 } \xrightarrow{} x_i \xrightarrow{ \varphi_i } v_i \xrightarrow{} 0$ where $\varphi_i$ is an $\add_{ \cE }( x )$-precover.
\end{Lemma}

\begin{proof}
Since $\cC$ is $\Hom$-finite, there is a triangle $w_{ i+1 } \xrightarrow{} y_i \xrightarrow{ \underline{ \theta }_i } v_i \xrightarrow{} \Sigma w_{ i+1 }$ in $\cC$ where $\underline{ \theta }_i$ is an $\add_{ \cC }( x )$-precover.  By Lemma \ref{lem:portmanteau}(vii) it is isomorphic to the triangle induced by a conflation $0 \xrightarrow{} v_{ i+1 } \xrightarrow{} x_i \xrightarrow{ \varphi_i } v_i \xrightarrow{} 0$ in $\cE$.

Observe that $\cC( x,\underline{ \theta }_i )$ is surjective, so $\cE( x,\varphi_i )$ is surjective by Lemma \ref{lem:portmanteau}(viii), so $\cE( x',\varphi_i )$ is surjective for $x' \in \add_{ \cE }( x )$.  Moreover, since $x_i \cong y_i$ in $\cC$, we have $x_i \oplus p \cong y_i \oplus q$ for certain projective-injective objects $p,q \in \cE$ by Lemma \ref{lem:portmanteau}(i).  By definition, $y_i \in \add_{ \cC }( x )$, so since $x$ is good, $y_i \in \add_{ \cE }( x )$.  Also, since $x$ is good, $q \in \add_{ \cE }( x )$.  Hence $y_i \oplus q \in \add_{ \cE }( x )$, so $x_i \oplus p \in \add_{ \cE }( x )$, so $x_i \in \add_{ \cE }( x )$.  
\end{proof}

\begin{Lemma}
\label{lem:F16}
Let $x \in \cE$ be a good object.  Set $A = \cE( x,x )$, $\fa = [\cP]( x,x )$, $\underline{ A } = \cC( x,x ) = A/\fa$. 

If $M_{ \underline{ A }} \in \Mod \underline{ A }$ is finitely presented over $A$, then there is a sequence $\cdots \xrightarrow{} x_2 \xrightarrow{ \partial_2 } x_1 \xrightarrow{ \partial_1 } x_0$ in $\cE$ with the following properties.
\begin{enumerate}
\setlength\itemsep{4pt}

  \item  There are conflations $0 \xrightarrow{} v_2 \xrightarrow{ \psi_2 } x_1 \xrightarrow{ \partial_1 } x_0 \xrightarrow{} 0$ and $0 \xrightarrow{} v_{ i+1 } \xrightarrow{ \psi_{ i+1 } } x_i \xrightarrow{ \varphi_i } v_i \xrightarrow{} 0$ for $i \geqslant 2$, such that $\partial_i = \psi_i \varphi_i$ for $i \geqslant 2$ and each $\varphi_i$ is an $\add_{ \cE }( x )$-precover.

  \item  $x_i \in \add_{ \cE }( x )$ for each $i$.

  \item  There is an augmented projective resolution $\cdots \xrightarrow{} \cE( x,x_2 ) \xrightarrow{ ( \partial_2 )_* } \cE( x,x_1 ) \xrightarrow{ ( \partial_1 )_* } \cE( x,x_0 ) \xrightarrow{} M_A \xrightarrow{} 0$ of $A$-right modules.

\end{enumerate}
\end{Lemma}

\begin{proof}
{\em Step 1: Constructing the conflation $0 \xrightarrow{} v_2 \xrightarrow{ \psi_2 } x_1 \xrightarrow{ \partial_1 } x_0 \xrightarrow{} 0$. }  Since $M_A$ is finitely presented, there is a morphism $x'_1 \xrightarrow{} x'_0$ in $\add_{ \cE }( x )$ such that
\begin{equation}
\label{equ:F16a}
  \cE( x,x'_1 ) \xrightarrow{} \cE( x,x'_0 ) \xrightarrow{} M_A \xrightarrow{} 0
\end{equation}
is an exact sequence of $A$-right modules.  Pick a deflation $p_1 \xrightarrow{} x'_0$ with $p_1$ projective-injective.  Set $x_1 \xrightarrow{ \partial_1 } x_0$ equal to $x'_1 \oplus p_1 \xrightarrow{} x'_0$.  Then:
\begingroup
\renewcommand{\labelenumi}{(\roman{enumi})'}
\begin{enumerate}
\setlength\itemsep{4pt}

  \item  There is a conflation $0 \xrightarrow{} v_2 \xrightarrow{ \psi_2 } x_1 \xrightarrow{ \partial_1 } x_0 \xrightarrow{} 0$ because $x_1 \xrightarrow{ \partial_1 } x_0$ is a deflation by Lemma \ref{lem:portmanteau}(iii).

  \item  $x_0,x_1 \in \add_{ \cE }( x )$ because $x'_0,x'_1 \in \add_{ \cE }( x )$ by construction and $p_1 \in \add_{ \cE }( x )$ since $x$ is good.  

  \item  There is an exact sequence
\[
  \cE( x,x_1 ) \xrightarrow{ ( \partial_1 )_* } \cE( x,x_0 ) \xrightarrow{} M_A \xrightarrow{} 0
\]
of $A$-right modules:  The sequence reads $\cE( x,x'_1 \oplus p_1 ) \xrightarrow{} \cE( x,x'_0 ) \xrightarrow{} M_A \xrightarrow{} 0$.  It is exact since \eqref{equ:F16a} is exact and since $\cE( x,p_1 ) \xrightarrow{} \cE( x,x'_0 )$ maps to $[\cP]( x,x'_0 )$, which is easily seen to be contained in $\cE( x,x'_0 ) \cdot \fa$, while $M_A$ is annihilated by $\fa$ because it is an $\underline{ A }$-module.

\end{enumerate}
\endgroup

{\em Step 2:  Constructing the conflation $0 \xrightarrow{} v_{ i+1 } \xrightarrow{ \psi_{ i+1 } } x_i \xrightarrow{ \varphi_i } v_i \xrightarrow{} 0$ for $i \geqslant 2$.}  Suppose that $v_i$ has already been constructed.  Then:
\begingroup
\renewcommand{\labelenumi}{(\roman{enumi})''}
\begin{enumerate}
\setlength\itemsep{4pt}

  \item  There is a conflation $0 \xrightarrow{} v_{ i+1 } \xrightarrow{ \psi_{ i+1 } } x_i \xrightarrow{ \varphi_i } v_i \xrightarrow{} 0$ where $\varphi_i$ is an $\add_{ \cE }( x )$-precover by Lemma \ref{lem:F8}. 
  
  \item  $x_i \in \add_{ \cE }( x )$ by construction.
  
  \item  There is a short exact sequence
\[  
  0 \xrightarrow{} \cE( x,v_{ i+1 } ) \xrightarrow{ ( \psi_{ i+1 } )_* } \cE( x,x_i ) \xrightarrow{ ( \varphi_i )_* } \cE( x,v_i ) \xrightarrow{} 0
\]
of $A$-right modules by construction.
\end{enumerate}
\endgroup
Combining the primed and double primed items establishes (i) and (ii) in the lemma.  It also gives the exact sequence in (iii), which is an augmented projective resolution over $A$ because $x_i \in \add_{ \cE }( x )$ for each $i$ by (ii), whence $\cE( x,x_i )$ is a projective $A$-right module.
\end{proof}

\begin{Lemma}
\label{lem:F17}
Let $x,y \in \cE$ be good objects
and set
\begin{align*}
  & A = \cE( x,x ) \;,\; \fa = [\cP]( x,x ) \;,\; \underline{ A } = \cC( x,x ) = A/\fa, \\
  & B = \cE( y,y ) \;,\; \fb = [\cP]( y,y ) \;,\; \underline{ B } = \cC( y,y ) = B/\fb.
\end{align*}
Suppose that $x \cong x' \oplus z'$, $y \cong y' \oplus z'$ in $\cC$ for some $z' \in \add_{ \cC }( x ) \cap \add_{ \cC }( y )$, and let $e \in \underline{ A }$, $f \in \underline{ B }$ be the projections onto $z'$.

Consider the $A$-$B$-bimodule ${}_{ A }X_B = \cE( y,x )$.  
\begin{enumerate}
\setlength\itemsep{4pt}

  \item  If $M_{ \underline{ A } } \in \Mod \underline{ A }$ is finitely presented over $A$, then:
\begin{enumerate}
\setlength\itemsep{4pt}

  \item  $\H_i( M \LTensor{ A } {}_{ A }X_B )$ is annihilated by $\fb$ for each $i$, so is a $\underline{ B }$-right module.

  \item  $\H_i( M \LTensor{ A } {}_{ A }X_B )$ is annihilated by $f\underline{ B }$ for each $i \geqslant 1$.

\end{enumerate}

  \item  If $N_{ \underline{ B } } \in \mod \underline{ B }$ is such that $\dual\!N$ is finitely presented over $B^{ \opp }$, then:
\begin{enumerate}
\setlength\itemsep{4pt}

  \item  $\H^i \RHom_B( {}_{ A }X_B,N )$ is annihilated by $\fa$ for each $i$, so is an $\underline{ A }$-right module.

  \item  $\H^i \RHom_B( {}_{ A }X_B,N )$ is annihilated by $e\underline{ A }$ for each $i \geqslant 1$.

\end{enumerate}
\end{enumerate}
\end{Lemma}

\begin{proof}
First note that by definition, $f \in \underline{ B }$ has a factorisation $y \xrightarrow{} z' \xrightarrow{} y$ in $\cC$.  If we lift $f$ to an element $\beta \in B$, this means that $\beta$ is the sum of an element with a factorisation $y \xrightarrow{} z' \xrightarrow{} y$ in $\cE$ and an element which factors through a projective-injective object.  But $x$ is good, so $\beta$ has a factorisation
\begin{equation}
\label{equ:F17a}
\vcenter{
  \xymatrix @+0.5pc {
    y \ar^{ \beta' }[r] \ar@/_1.5pc/[rr]_{ \beta } & z'' \ar^{ \beta'' }[r] & y \\
                    }
        }
\end{equation}
in $\cE$ with $z'' \in \add_{ \cE }( x )$.

(i) and (ii):  Lemma \ref{lem:F16} applies to $M_{ \underline{ A } }$, and using the notation from there we have
\begin{align*}
  M \LTensor{ A } {}_{ A }X_B & \cong \big( \cdots \xrightarrow{} \cE( x,x_2 ) \xrightarrow{ ( \partial_2 )_* } \cE( x,x_1 ) \xrightarrow{ ( \partial_1 )_* } \cE( x,x_0 ) \big) \Tensor{ A } \cE( y,x ) \\
  & \cong \cdots \xrightarrow{} \cE( y,x_2 ) \xrightarrow{ ( \partial_2 )_* } \cE( y,x_1 ) \xrightarrow{ ( \partial_1 )_* } \cE( y,x_0 ),
\end{align*}
where the second isomorphism is by Lemma \ref{lem:portmanteau}(v).  Consider a homology class in the module $\H_i( M \LTensor{ A } {}_{ A }X_B )$.  It is represented by a morphism $\upsilon \in \cE( y,x_i )$.  

If $i \geqslant 1$ then $\upsilon \in \Ker ( \partial_i )_*$, and Lemma \ref{lem:F16}(i) shows that $\upsilon$ factors $y \xrightarrow{ \upsilon' } v_{ i+1 } \xrightarrow{ \psi_{ i+1 } } x_i$.  Let $\beta \in B$ be either the lift of $f$ from the start of the proof, or an element of $\fb$.  In the former case, we saw that $\beta$ has the factorisation \eqref{equ:F17a} with $z'' \in \add_{ \cE }( x )$.  In the latter case, $\beta$ has the same factorisation because it factors through a projective-injective object, and such an object is in $\add_{ \cE }( x )$ since $x$ is good.  We must show that $\beta$ annihilates the homology class of $\upsilon$, that is, $\upsilon\beta$ is zero in homology.  This holds by the following diagram, using that $\varphi_{ i+1 }$ is an $\add_{ \cE }( x )$-precover.
\[
\vcenter{
  \xymatrix @+0.5pc {
    & y \ar^{ \beta' }[d] \ar@/_-2.5pc/[dd]^{ \beta } \\
    & z'' \ar^{ \beta'' }[d] \ar@{-->}[ddl] \\
    & y \ar^>>>>{ \upsilon' }[d] \ar^{ \upsilon }[dr] \\
    x_{ i+1 } \ar_{ \varphi_{ i+1 } }[r] \ar@/_2.5pc/[rr]_{ \partial_{ i+1 } } & v_{ i+1 }  \ar_{ \psi_{ i+1 } }[r] & x_i \\
                    }
        }
\]

If $i = 0$ then let $\beta \in \fb$ be given.  Then $\beta \in \cE( y,y )$ has a factorisation $y \xrightarrow{ \beta' } p \xrightarrow{ \beta'' } y$ with $p$ projective-injective.  We must again show  that $\beta$ annihilates the homology class of $\upsilon$.  This holds by the following simpler diagram, using that $\partial_1$ is a deflation.
\[
\vcenter{
  \xymatrix @+0.5pc {
    & y \ar^{ \beta' }[d] \ar@/_-2.5pc/[dd]^{ \beta } \\
    & p \ar^{ \beta'' }[d] \ar@{-->}[ddl] \\
    & y \ar^{ \upsilon }[d] \\
    x_1 \ar_{ \partial_1 }[r] & x_0 \\
                    }
        }
\]

(iii) and (iv): We have $N \cong \dual\!\dual\!N \cong \RHom_k( \dual\!N,k )$ so by adjointness,
\[
  \RHom_B( {}_{ A }X_B,N ) \cong \RHom_B\big( {}_{ A }X_B, \RHom_k( \dual\!N,k ) \big) \cong \RHom_k( {}_{ A }X_B \LTensor{ B } \dual\!N,k ).
\]
Combine this with the opposite versions of (i) and (ii), where $x$ and $y$ have been interchanged.  
\end{proof}

\begin{Lemma}
\label{lem:F21}
Let $x \in \cC$ be a good maximal rigid object, $y \in \cC$ a rigid object, set
\begin{align*}
  & A = \cE( x,x ) \;,\; \underline{ A } = \cC( x,x ), \\
  & B = \cE( y,y ) \;,\; \underline{ B } = \cC( y,y ),
\end{align*}
and let
\begin{equation}
\label{equ:F21a}
  y \xrightarrow{ \underline{ \upsilon } } x^0 \xrightarrow{ \underline{ \xi } } x^1 \xrightarrow{} \Sigma y
\end{equation}
be a triangle in $\cC$ with $x^i \in \add_{ \cC }( x )$, see Lemma \ref{lem:ZZ}.  

Consider the $A$-left module ${}_{ A }X = \cE( y,x )$ and the complex of $\underline{ A }$-left modules:
\begin{equation}
\label{equ:F21e}
  {}_{ \underline{ A } }P = \cC( x^1,x ) \xrightarrow{ \cC( \underline{ \xi },x ) } \cC( x^0,x ).
\end{equation}
Then:
\begin{enumerate}
\setlength\itemsep{4pt}

  \item  For $M_{ \underline{ A } } \in \cD( \underline{ A } )$ there is an isomorphism $M_A \LTensor{ A } {}_{ A }X \cong M_{ \underline{ A } } \Tensor{ \underline{ A } } {}_{ \underline{ A } }P$ in $\cD( k )$.

  \item  For ${}_{ \underline{ A } }N \in \cD( \underline{ A }^{ \opp } )$ there is an isomorphism $\RHom_{ A^{ \opp } }( {}_{ A }X,{}_{ A }N ) \cong \Hom_{ \underline{ A }^{ \opp } }( {}_{ \underline{ A } }P,{}_{ \underline{ A } }N )$ in $\cD( k )$.

\end{enumerate}
\end{Lemma}

\begin{proof}
By the dual of Lemma \ref{lem:F9_10} there exists a conflation in $\cE$ with $\widetilde{ x }^i \in \add_{ \cE }( x )$,
\[
  0 \xrightarrow{} y \xrightarrow{ \widetilde{ \upsilon } } \widetilde{ x }^0 \xrightarrow{ \widetilde{ \xi } } \widetilde{ x }^1 \xrightarrow{} 0,
\]
which induces an augmented projective resolution $0 \xrightarrow{} \cE( \widetilde{ x }^1,x ) \xrightarrow{ \cE( \widetilde{ \xi },x ) } \cE( \widetilde{ x }^0,x ) \xrightarrow{ \cE( \widetilde{ \upsilon },x ) } \cE( y,x ) \xrightarrow{} 0$ of $A$-left modules.  Set
\[
  {}_{ A }\widetilde{ Q } = \cE( \widetilde{ x }^1,x ) \xrightarrow{ \cE( \widetilde{ \xi },x ) } \cE( \widetilde{ x }^0,x ).
\]
The conflation also induces a triangle in $\cC$,
\begin{equation}
\label{equ:F21d}
  y \xrightarrow{ \underline{ \widetilde{ \upsilon } } }\widetilde{ x }^0 \xrightarrow{ \underline{ \widetilde{ \xi } } }\widetilde{ x }^1 \xrightarrow{} \Sigma y.
\end{equation}
Applying the functor $\cC( -,x )$ gives a long exact sequence, and $\cC( \Sigma^{ -1 }\widetilde{ x }^1,x ) = 0$ since $x$ is rigid in $\cC$ and $\widetilde{ x }^1 \in \add_{ \cE }( x )$ whence $\widetilde{ x }^1 \in \add_{ \cC }( x )$.  Hence the morphism $y \xrightarrow{ \underline{ \widetilde{ \upsilon } } }\widetilde{x}^0$ is an $\add_{ \cC }( x )$-preenvelope.  The same applies to the morphism $y \xrightarrow{ \underline{ \upsilon } } x^0$ from \eqref{equ:F21a}.  It follows that these two morphisms agree up to trivial summands of the form $0 \xrightarrow{} x'$ with $x' \in \add_{ \cC }( x )$, whence the triangles \eqref{equ:F21a} and \eqref{equ:F21d} agree up to trivial summands of the form $0 \xrightarrow{} x' \xrightarrow{ = } x' \xrightarrow{} 0$.  This implies that if we set
\[
  {}_{ \underline{ A } }\widetilde{ P } = \cC( \widetilde{ x }^1,x ) \xrightarrow{ \cC( \underline{ \widetilde{ \xi } },x ) } \cC( \widetilde{ x }^0,x ),
\]
then 
\begin{equation}
\label{equ:F21f}
  {}_{ \underline{ A } }P \cong {}_{ \underline{ A } }\widetilde{ P }
\end{equation}
in $\Kb( \proj \underline{ A }^{ \opp } )$.  Moreover, Lemma \ref{lem:portmanteau}(vi) implies that
\begin{equation}
\label{equ:F21g}
  {}_{ \underline{ A } }\underline{ A }_A \Tensor{ A } {}_{ A }\widetilde{ Q } \cong {}_{ \underline { A } }\widetilde{ P }.
\end{equation}

Part (i) now follows:
\[
  M_A \LTensor{ A } {}_{ A }X
  \cong M_A \Tensor{ A } {}_{ A }\widetilde{ Q }
  \cong M_{ \underline{ A } } \Tensor{ \underline{ A } } {}_{ \underline{ A } }A_A \Tensor{ A } {}_{ A }\widetilde{ Q }
  \stackrel{ \rm (a) }{ \cong } M_{ \underline{ A } } \Tensor{ \underline{ A } } {}_{ \underline{ A } }\widetilde{ P }
  \stackrel{ \rm (b) }{ \cong } M_{ \underline{ A } } \Tensor{ \underline{ A } } {}_{ \underline{ A } }P,
\]
where (a) and (b) are by Equations \eqref{equ:F21g} and \eqref{equ:F21f}.  Part (ii) similarly follows:
\begin{align*}
  \RHom_{ A^{ \opp } }( {}_{ A }X,{}_{ A }N )
  & \cong \Hom_{ A^{ \opp } }( {}_{ A } \widetilde{ Q },{}_{ A }N ) \\
  & \cong \Hom_{ A^{ \opp } }\big( {}_{ A }\widetilde{ Q },\Hom_{ \underline{ A }^{ \opp } }( {}_{ \underline{ A } }\underline{ A }_A,{}_{ \underline{ A } }N ) \big) \\
  & \stackrel{ \rm (a) }{ \cong } \Hom_{ \underline{ A }^{ \opp } }( {}_{ \underline{ A } }\underline{ A }_A \Tensor{ A } {}_{ A }\widetilde{ Q },{}_{ \underline{ A } }N ) \\
  & \stackrel{ \rm (b) }{ \cong } \Hom_{ \underline{ A }^{ \opp } }( {}_{ \underline{ A } }\widetilde{ P },{}_{ \underline{ A } }N ) \\
  & \stackrel{ \rm (c) }{ \cong } \Hom_{ \underline{ A }^{ \opp } }( {}_{ \underline{ A } }P,{}_{ \underline{ A } }N ),
\end{align*}
where (a) is by adjointness, (b) and (c) by Equations \eqref{equ:F21g} and \eqref{equ:F21f}. 
\end{proof}

\medskip
\noindent
{\bf Acknowledgement.}
We thank Jenny August and Matthew Pressland for discussions during the preparation of this paper, and Martin Kalck and Amnon Yekutieli for comments on the first version.

This work was supported by EPSRC grant EP/P016014/1 ``Higher Dimensional Ho\-mo\-lo\-gi\-cal Algebra'', NSF grant DMS-1901830, and Bulgarian Science Fund grant DN02/05.

\end{document}